\documentclass[11pt,a4paper]{amsart}

\usepackage{latexsym}
\usepackage{amsmath}
\usepackage{amsthm}
\usepackage{amscd}
\usepackage{amssymb}
\usepackage{mathtools}
\usepackage[colorlinks=true,citecolor=black,linkcolor=black,urlcolor=black]{hyperref}
\usepackage[capitalize,nameinlink,noabbrev]{cleveref}
\usepackage{mdframed}
\usepackage{graphicx}
\usepackage{hyperref}
\usepackage{enumerate}
\usepackage{mathrsfs}
\usepackage{url}
\usepackage{tikz}
\usepackage{tikz-cd}
\usepackage{calc}
\usepackage{color}
\usepackage{tcolorbox}
\usepackage{bm}
\usepackage[all]{xy}
\usepackage[makeroom]{cancel}
\usepackage{verbatim}
\usepackage{enumitem}
\usepackage{bold-extra}

\usepackage[textwidth=27mm, textsize=scriptsize, color=green!20]{todonotes}

\newtheorem{introtheorem}{Theorem}

\newtheorem{introconjecture}[introtheorem]{Conjecture}

\newtheorem*{theorem*}{Theorem}
\newtheorem{theorem}{Theorem}[section]
\newtheorem{proposition}[theorem]{Proposition}
\newtheorem{lemma}[theorem]{Lemma}
\newtheorem{claim}[theorem]{Claim}
\newtheorem{corollary}[theorem]{Corollary}

\theoremstyle{definition}
\newtheorem{definition}[theorem]{Definition}
\theoremstyle{definition}
\newtheorem{remark}[theorem]{Remark}
\theoremstyle{definition}
\newtheorem{example}[theorem]{Example}
\theoremstyle{definition}

\DeclareMathOperator{\stab}{stab}

\DeclareMathOperator{\Ker}{Ker}
\DeclareMathOperator{\Spec}{Spec}

\DeclareMathOperator{\NP}{NP}
\DeclareMathOperator{\val}{val}
\DeclareMathOperator{\ord}{ord}

\DeclareMathOperator{\codim}{codim}
\DeclareMathOperator{\supp}{supp}

\DeclareMathOperator{\vol}{vol}
\DeclareMathOperator{\MV}{MV}
\DeclareMathOperator{\MI}{MI}

\DeclareMathOperator{\Gal}{Gal}

\DeclareMathOperator{\h}{h}

\DeclareMathOperator{\Frob}{Frob}

\DeclareMathOperator{\m}{m}

\DeclareMathOperator{\Ron}{Ron}

\DeclareMathOperator{\dist}{dist}
\DeclareMathOperator{\interior}{int}

\DeclareMathOperator{\rank}{rank}

\newcommand{\longhookrightarrow}{\lhook\joinrel\longrightarrow}

\renewcommand{\and}{\quad \text{and} \quad}

\renewcommand{\div}{\operatorname{div}}

\newcommand{\Gm}{\mathbb{G}_{\mathrm{m}}}

\newcommand{\tors}{\mathrm{tors}}

\newcommand{\an}{\mathrm{an}}

\newcommand{\red}{\mathrm{red}}

\newcommand{\chern}{\mathrm{c}_{1}}

\newcommand{\wA}{{\mathscr{A}\hspace{-2.5mm}\widetilde{\hphantom{ii}\vphantom{H}}}}
\newcommand{\wAV}{{\mathscr{A}_{V}\hspace{-4.5mm}\widetilde{\hphantom{ii}\vphantom{H}}}\hspace{1.8mm}}

\newcommand{\wHt}{{\mathscr{H}(t)\hspace{-8mm}\widetilde{\hphantom{HH}\vphantom{H}}}\hspace{1mm}}

\makeatletter
\def\subsection{\@startsection{subsection}{2}
  \z@{.5\linespacing\@plus.7\linespacing}{-.5em}
   {\bfseries}}
\makeatother

\begin{document}

\title[Heights of complete intersections]{Heights of
  complete intersections\\ in toric varieties}

\author[Gualdi]{Roberto~Gualdi}
\address{\hspace*{-6.3mm} Departament de Matem\`atiques, Universitat Polit\`ecnica de Catalunya, 08034 Barcelona, Spain \vspace*{-2.8mm}}
\address{\hspace*{-6.3mm} {\it Email address:} {\tt roberto.gualdi@upc.edu}}

\author[Sombra]{Mart{\'\i}n~Sombra}
\address{\hspace*{-6.3mm} ICREA, 
  08010 Barcelona, Spain \vspace*{-2.8mm}}
\address{ 
\hspace*{-6.3mm}  Departament de Matem\`atiques i
  Inform\`atica, Universitat de Barcelona,  08007
  Bar\-ce\-lo\-na, Spain \vspace*{-2.8mm}}
\address{\hspace*{-6.3mm} Centre de Recerca Matem\`atica, 08193 Bellaterra, Spain \vspace*{-2.8mm}}
\address{\hspace*{-6.3mm} {\it Email address:} {\tt martin.sombra@icrea.cat}}


\date{\today}

\begin{abstract}
  The height of a toric variety and that of its hypersurfaces can be
  expressed in convex-analytic terms as an adelic sum of mixed
  integrals of their roof functions and duals of their Ronkin
  functions.

  Here we extend these results to the $2$-codimensional situation by
  presenting a limit formula predicting the typical height of the
  intersection of two hypersurfaces on a toric variety.  More
  precisely, we prove that the height of the intersection cycle of two
  effective divisors translated by a strict sequence of torsion points
  converges to an adelic sum of mixed integrals of roof and duals of
  Ronkin functions.
  
  This partially confirms a previous conjecture of the authors about the average
  height of families of complete intersections in toric varieties.
\end{abstract}

\maketitle

\setcounter{tocdepth}{1}
\tableofcontents

\section*{Introduction}

The study of the Arakelov geometry of toric varieties was inaugurated
by Maillot in~\cite{Maillot}, where among other results he obtained
upper bounds for the canonical height of a complete
intersection. These can be viewed as an arithmetic analogue of the
inequality part in the classical Bernstein--Kushnirenko--Khovanskii
theorem~\cite{Bernst}.

Finding an arithmetic version of the equality part in this theorem is
a more challenging problem. A first step towards this goal was taken
by Burgos Gil, Philippon and the second author in the monograph
\cite{BPS}, where they studied semipositive metrized line bundles on a
toric variety which are invariant under the action of the compact
torus and gave a convex-analytic formula for the corresponding height
of this ambient variety.  Subsequently, their formula was extended by the
first author to compute the height of a
hypersurface in similar terms~\cite{Gualdi}.

For higher-codimensional complete intersections in toric varieties, one can not expect a
general convex-analytic formula only depending on the arithmetic size
of the defining Laurent polynomials, as shown in
\cite[Example~5.1.1]{Guathesis} and \cite[Example~2.1]{GualdiSombra}.
However, these two references have suggested that the latter
information could be enough to recover the \emph{limit} height of a
certain family of complete intersection cycles associated with the
Laurent polynomials.

To explain this precisely, let us fix a number field~$\mathbb{K}$ with
set of places~$\mathfrak{M}$, and a split algebraic torus $\mathbb{T}$
over $\mathbb{K}$ of dimension $n$ with character lattice $M$.  For a
Laurent polynomial $f=\sum_m \alpha_m\chi^m\in\mathbb{K}[M]$ and a
point~$t\in\mathbb{T}(\overline{\mathbb{K}})$, we can define the
\emph{twist} of $f$ by~$t$ as
\[
  t^*f=\sum_m\alpha_m\chi^m(t)\, \chi^m \in \overline{\mathbb{K}}[M].
\]

Now, for~$k\leq n$ let
\begin{displaymath}
\bm{f}= (f_1,\ldots,f_k) \and   \bm{t}=(t_1,\ldots,t_k) 
\end{displaymath}
be a family of nonzero Laurent polynomials in~$\mathbb{K}[M]$ and a
family of points of~$\mathbb{T}(\overline{\mathbb{K}})$, respectively.
Whenever the divisors of the twisted Laurent polynomials
$t_1^*f_1,\ldots,t_k^*f_k$ meet properly on the base
change~$\mathbb{T}_{\overline{\mathbb{K}}}$, intersection theory
allows to consider the $(n-k)$-dimensional cycle
\[
Z_{\mathbb{T}}(\bm{t}^*\bm{f})=\div(t_{1}^*f_1)\cdots\div(t_{k}^*f_k)\cdot \mathbb{T}_{\overline{\mathbb{K}}}.
\]
Geometrically, this is the cycle obtained by intersecting the
translations of the hypersufaces of each $f_{i}$ by the corresponding
point~$t_{i}^{-1}$.

To be able to consider the height of such cycles we need to fix a
compactification of the torus. Let then $X$ be a complete toric
variety with torus~$\mathbb{T}$ and
$\overline{D}_0,\ldots,\overline{D}_{n-k}$ a family of semipositive
metrized divisors on it. The corresponding height of
$Z_{\mathbb{T}}(\bm{t}^*\bm{f})$ is defined as the height of the
closure of this cycle in~$X_{\overline{\mathbb{K}}}$.

With this convention, the following statement is a strengthening of the guesses formulated
  in \cite[Conjecture~6.4.4]{Guathesis} and in
  \cite[Conjecture~11.8]{GualdiSombra}.

\begin{introconjecture}\label{conj: main conjecture}
  For~$k\leq n$ let $\bm{f}= (f_1,\ldots,f_k)$ be a family of nonzero
  Laurent polynomials in~$\mathbb{K}[M]$, and $X$ a complete
  toric variety with torus $\mathbb{T}$ equipped with a family
  $\overline{D}_0,\ldots,\overline{D}_{n-k}$ of semipositive toric
  metrized divisors.  Then for each sequence $(\bm{\omega}_\ell)_\ell$
  of torsion points in $\mathbb{T}(\overline{\mathbb{K}})^k$ whose
  projection to $(\mathbb{T}^k/\mathbb{T})(\overline{\mathbb{K}})$ via
  the diagonal action is strict we have
\[
\lim_{\ell\to\infty}\h_{\overline{D}_0,\ldots,\overline{D}_{n-k}}(Z_{\mathbb{T}}(\bm{\omega}_\ell^*\bm{f}))
=
\sum_{v\in\mathfrak{M}}n_v\MI_M(\vartheta_{\overline{D}_0,v},\ldots,\vartheta_{\overline{D}_{n-k},v},\rho_{f_1,v}^\vee,\ldots,\rho_{f_k,v}^\vee),
\]
where for each $v\in \mathfrak{M}$ we denote by $n_{v}$ the weight of
this place, by $\vartheta_{\overline{D}_i,v}$ the $v$-adic roof
function of $\overline{D}_{i}$, and by $\rho_{f_j,v}^{\vee}$ the
Legendre--Fenchel dual of the $v$-adic Ronkin function of~$f_j$.
\end{introconjecture}

The definitions for the $v$-adic roof functions, the Legendre--Fenchel
dual of the $v$-adic Ronkin functions and the mixed integral operator
$\MI_{M}$ appearing in this conjecture can be found in
\cite[Definitions~5.1.4 and 2.7.16]{BPS} and
\cite[Definition~2.7]{Gualdi}, their underlying ideas being briefly
recalled in \cref{sec:height-inters-cycl}.  The strictness requirement
on the sequence $(\bm{\omega}_\ell)_\ell$ ensures that the cycle
$Z_{\mathbb{T}}(\bm{\omega}_\ell^*\bm{f})$ is well-defined for $\ell$
sufficiently large thanks to \cref{cor: proper intersection for strict
  sequences}, and so the limit in the left hand side of the
conjectural equality makes sense.

In spirit, \cref{conj: main conjecture} affirms that the typical
height of the complete intersection defined by a family of Laurent
polynomials twisted by torsion points may be predicted in
convex-analytic terms from the knowledge of the arithmetic complexity
of these Laurent polynomials.  In this way it
can be considered as an arithmetic analogue of a
weak version of the Bernstein--Kushnirenko--Khovanskii theorem
computing the typical degree of
$Z_{\mathbb{T}}(\boldsymbol{t}^{*} \boldsymbol{f})$ for
$\boldsymbol{t}\in \mathbb{T}(\overline{\mathbb{K}})^{k}$ (\cref{cor:
  geometric degree of cycle}).

More precisely, \cref{conj: main conjecture} can be interpreted as the
statement that for all $\varepsilon >0$ most twists of $\bm{f}$ by
torsion points define complete intersections whose heights are closer
than~$\varepsilon$ to the proposed adelic sum of mixed integrals. In
particular, the average of these heights coincides with this
quantity. We refer to~\cite[Corollary~7.8]{GualdiSombra} for a
specific situation that can be easily adapted to this more general
setting.  On the other hand, it is not expected that the sequence of
heights in the conjecture is eventually constant, see for instance
\cite[Question 9.4]{GualdiSombra}.

\vspace{\baselineskip}

Till now, the conjecture was only known for
$f_1=f_2=x_1+x_2+1 \in \mathbb{Q}[x_{1}^{\pm1},x_{2}^{\pm1}]$ and the
canonical height on the projective plane~\cite[Corollary
5]{GualdiSombra}.  In this particular case, the limit height could
also be expressed as a quotient of special values of the Riemann zeta
function.

Our main result in this paper is the following. 

\begin{introtheorem}\label{thm: main theorem}
\cref{conj: main conjecture} holds for~$k=2$.
\end{introtheorem}

When specialized to bivariate Fermat polynomials of arbitrary degree,
this result confirms \cite[Conjecture~11.8]{GualdiSombra}.  It also
implies the validity in the $2$-codimensional situation of the
preliminary \cite[Conjecture~6.4.4]{Guathesis} via the argument
explained in \cite[Section~6]{GualdiSombra}.

Independently of us, Destic, Hultberg and Szachniewicz have proven the
conjecture in the general case \cite{DHS}. Their approach is radically
different, relying on the application of Yuan's equidistribution
theorem and on the theory of globally valued fields from unbounded
continuous logic introduced by Ben Yaacov and Hrushovski. 

\vspace{\baselineskip}

The strategy of the proof of \cref{thm: main theorem} is centered on
the study of the local contributions to the height.  More precisely,
applying the arithmetic B\'ezout theorem we can realize the height of
$Z_{\mathbb{T}}(\bm{\omega}_\ell^*\bm{f})$ as
\begin{equation}
\label{eq:138}
  \h_{\overline{D}_0,\ldots,\overline{D}_{n-2},\overline{D}^{\Ron}_{f_{2}}}(Z_{\mathbb{T}}(\omega_{\ell,1}^*f_1))
+
\sum_{v\in\mathfrak{M}}\frac{n_{v}}{\# O(\boldsymbol{\omega}_{\ell})_{v}}\sum_{\boldsymbol{\eta}\in O(\boldsymbol{\omega}_{\ell})_{v}} I_{v}(\boldsymbol{\eta}),
\end{equation}
where $\overline{D}^{\Ron}_{f_2}$ is the toric divisor associated with
the Newton polytope of~$f_2$ equipped with its {Ronkin metric},
$O(\bm{\omega}_{\ell})_v$ is the image of the Galois orbit of
$\bm{\omega}_{\ell}$ in the $v$-adic analytic torus
$\mathbb{T}_v^{\an}$, and
$I_v\colon \mathbb{T}(\mathbb{C}_{v})^{2}\to \mathbb{R}\setminus
\{-\infty\}$ is a function defined as an integral over the $v$-adic
analytic toric variety $X_{v}^{\an}$.  As the height of the
hypersurface defined by $f_1$ with respect to the metrized divisors
$\overline{D}_0,\ldots,\overline{D}_{n-2},\overline{D}^{\Ron}_{f_2}$
is well-understood in convex-analytic terms from~\cite{Gualdi}, we can
reduce the proof of \cref{thm: main theorem} to showing that the
second summand converges to zero along sequences
$(\bm{\omega}_\ell)_\ell$ with strict image by the diagonal action
(\cref{sec:height-inters-cycl}).

The study of the asymptotic behaviour of a function on strict
sequences of torsion points immediately calls for the application of
the equidistribution results.  Unfortunately the equidistribution of
roots of unity can not be applied directly to the
function~$I_v$. Indeed this function is only defined on the algebraic
points of~$\mathbb{T}$, and one first needs to extend it to the whole
analytic torus.  More seriously, this extension does not need to be
continuous, and in fact it can take the value $-\infty$.  Our focal
goal is then to prove that $I_v$ is (the restriction of) a function on
the $v$-adic analytic torus $\mathbb{T}_v^{\an}$ with at most
logarithmic singularities. We achieve this by applying Stoll's theorem
on the continuity of the fiber integral in the Archimedean case
\cite{Stoll} and non-Archimedean analytic and formal geometry
otherwise \cite{BGR,Ber}.  This is done in
\cref{sec:an-auxil-funct-1}.

As a consequence of this property we can apply the local logarithmic
equidistribution theorems for torsion points of Dimitrov and Habegger
in the Archimedean case \cite{DH} and of Tate and Voloch in the
non-Archimedean case \cite{Tate_Voloch}.  When combined, these two
statements show that each $v$-adic summand in \eqref{eq:138} converges
to~$0$ (\cref{sec: local log equidistribution}).

Still, this local asymptotic vanishing is not enough to conclude the proof of \cref{thm: main theorem}.
In fact, it can happen that there are infinitely many places for which
the corresponding local $v$-adic contribution does not eventually
vanish, and therefore we are obliged to keep track of infinitely many
places even for arbitrary large values of~$\ell$.  To ensure that the
second summand in \eqref{eq:138} asymptotically vanishes we therefore
need to prove that there exists a finite subset
$\mathfrak{S}\subset\mathfrak{M}$ such that
\[
  \lim_{\ell\to\infty} \sum_{v\in\mathfrak{M\setminus\mathfrak{S}}}
  \frac{n_{v}}{\#
    O(\boldsymbol{\omega}_{\ell})_{v}}\sum_{\boldsymbol{\eta}\in
    O(\boldsymbol{\omega}_{\ell})_{v}} I_{v}(\boldsymbol{\eta})=0.
\]
This can be done by explicit computation for a specific set
$\mathfrak{S}$ of bad places (\cref{sec: adelic equidistribution}).

\vspace{\baselineskip}

Because of its local nature, this approach can be used to obtain
formul\ae\ for the typical values of the local heights of the cycle
$Z_{\mathbb{T}}(\bm{\omega}^*\bm{f})$.  Moreover, as it
relies on equidistribution theorems for which quantitative versions
are available \cite{DH,Schefer}, it might yield explicit bounds for
the approximation of these local heights to their limit, as already
done by Lin for the special case $f_1=f_2=1+x_1+x_2$ and the canonical
height on the projective plane \cite{Lin}.

\vspace{\baselineskip}
\noindent\textbf{Acknowledgments.}
We are grateful to Francesco Amoroso, Joaquim
Ortega-Cerd\`a, Riccardo Pengo and Alain Yger for several discussions
and suggestions about the contents of this article.
We also thank  Pablo Destic, Nuno Hultberg and Micha{\l} Szachniewicz for sharing a preliminary version of their article \cite{DHS}.

This research was
carried out at the Universitat de Barcelona, the Universit\"at
Regensburg and the Universitat Polit\`ecnica de Catalunya.  We also
thank these institutions for their hospitality.

Roberto Gualdi was partially supported by the DFG collaborative research center SFB 1085, the MICINN research project
PID2023-147642NB-I00, the AGAUR research project 2021-SGR-00603 and
the UPC program ALECTORS--2023.
Mart\'{\i}n Sombra was partially supported by the MICINN research
projects PID2019-104047GB-I00 and PID2023-147642NB-I00, the AGAUR
research project 2021-SGR-01468 and the AEI project CEX2020-001084-M
of the Mar\'{\i}a de Maeztu program for centers and units of
excellence in R\&D.

\section*{Conventions and notations}

A \emph{variety} is an integral separated scheme $X$ of finite type
over a field $K$. A \emph{divisor} on $X$ is a Cartier divisor, unless
otherwise said.  For an algebraic closure $\overline{K}$ of $K$, the
points of $X(\overline{K})$ are called the \emph{algebraic points} of
$X$.

Set $X_{\overline{K}}=X\times_{K} \Spec(\overline{K})$ for the base
change by the field extension $\overline{K}/K$.  A \emph{cycle} $Z$ of
$X_{\overline{K}}$ is a formal finite sum of subvarieties of $X$ with
integer coefficients.  For an integer $r$, the cycle $Z$ is an
\emph{$r$-cycle} if these subvarieties can be taken of dimension $r$.
We denote by $|Z|$ the support of $Z$.  For an open subset
$U\subset X$ and a cycle $Y$ of $U$, we set~$\overline{Y}$ for its
closure in $X$.

Set $n=\dim(X)$. Then for an integer $0\le k\le n$ and a family of global sections
$s_1,\ldots,s_k$ of line bundles on $X_{\overline{K}}$ we denote by
$V_{X}(s_{1}, \dots, s_{k})$ its zero set. These global sections are
said to \emph{meet properly} if for every subset
$I\subseteq\{1,\ldots,k\}$ each irreducible component of
$V_{X}(\{s_{i}\}_{i\in I})$ has codimension~$\#I$.  When this is the
case, intersection theory allows to define the $(n-k)$-cycle of $X_{\overline{K}}$ 
\begin{equation}
  \label{eq:2}
  Z_{X}(s_1,\ldots,s_k)=\div(s_{1})\cdots \div(s_{k})\cdot X_{\overline{K}}.
\end{equation}
We have that $| Z_{X}(s_1,\ldots,s_k)|= V_{X}(s_1,\ldots,s_k)$.

We denote by $\mathbb{K}$ a number field and by $\mathfrak{M}$ its set
of (nontrivial) places. For each $v\in\mathfrak{M}$ we denote by $|\cdot|_v$ the
unique absolute value on $\mathbb{K}$ representing $v$ and extending
one of the standard absolute values on $\mathbb{Q}$.  We set
$\mathbb{K}_{v}$ for the completion of $\mathbb{K}$ with respect to
$|\cdot|_{v}$ and associate to  $v$ the positive real
weight
\[
n_v=\frac{[\mathbb{K}_v:\mathbb{Q}_{v_{0}}]}{[\mathbb{K}:\mathbb{Q}]}
\]
where $v_{0}$ is the restriction of this place to the field of
rational numbers.
We denote by $\mathbb{C}_{v}$ the completion of an algebraic
closure of~$\mathbb{K}_{v}$. The absolute value on $\mathbb{K}_v$ has a unique
continuous extension to $\mathbb{C}_{v}$, that is also denoted by
$|\cdot|_{v}$. When $v$ is non-Archimedean, we denote by $\mathbb{C}_v^\circ$ the corresponding valuation ring, and by $\widetilde{\mathbb{C}}_{v}$ the associated residue field.

For a variety $X$ over $\mathbb{K}$ and a place $v\in \mathfrak{M}$ we
denote by $X_{v}$ the base change with respect to $\mathbb{C}_{v}$ and
by $X_{v}^{\an}$ the corresponding analytification in the sense of
Berkovich~\cite{Ber}. Similarly, for a morphism $f$ of varieties over
$\mathbb{K}$ we denote by $f_{v}$ its base change with respect to
$\mathbb{C}_{v}$ and by $f_{v}^{\an}$ its analytification.

\numberwithin{equation}{section}

\section{Geometric results}

Here we recall the basic notions and constructions from
toric geometry that we will use throughout, referring to
\cite{Fulton_toric} and \cite[Chapter 3]{BPS} for the proofs and more
details.  In addition, we study the cycles of a toric variety given as
the closure of intersection cycles of the torus of generic twists of
Laurent polynomials.  Our main result here is \cref{thm: geometric
  main theorem}, which realizes them as the intersection cycles of the
same twists for the associated global sections.  In particular, this
allows to express their degree in combinatorial terms (\cref{cor:
  geometric degree of cycle}).

In this section we denote by ~$K$ an arbitrary field and by
$\mathbb{T}$ a split algebraic torus over $K$ of dimension $n\ge
0$. We set $M$ for the character lattice of $\mathbb{T}$ and
$N=M^{\vee}$ for the dual lattice. Set
also~$M_\mathbb{R}= M\otimes\mathbb{R}$ and
$N_{\mathbb{R}}= N\otimes\mathbb{R}$ for the associated vector spaces,
and let $\langle \cdot , \cdot \rangle $ denote their pairing.

\subsection{Geometric toric constructions}
\label{sec:geom-toric-constr}

Let $f$ be a nonzero Laurent polynomial in the group algebra~$K[M]$ and set
\begin{math}
Z_{\mathbb{T}}(f)=\div(f)\cdot \mathbb{T}
\end{math}
for its associated Weil divisor.

\begin{definition}
  \label{def:4}
Writing
$f=\sum_{m}\alpha_{m}\chi^{m}$, the \emph{twist} of $f$ by a point
$t\in \mathbb{T}({K})$ is defined as
\begin{math}
  t^{*}f = \sum_{m}\alpha_{m}\chi^{m}(t)\chi^{m} \in {K}[M].
\end{math}
\end{definition}

This twist coincides with the pullback of $f$ by the
translation-by-$t$ map $\tau_{t}\colon \mathbb{T}\to \mathbb{T}$
defined by $x\mapsto t\cdot x$. In particular, its Weil divisor
agrees with the pushforward of the Weil divisor of $f$ with respect
to the inverse translation, that~is
\begin{equation*}
Z_{\mathbb{T}}(t^{*}f)=  \tau_{t^{-1},*} Z_{\mathbb{T}}(f).
\end{equation*}

We denote by
$\NP(f) \subset M_{\mathbb{R}}$ the {Newton polytope} of~$f$ and we
set $\Psi_{\NP(f)}\colon N_{\mathbb{R}}\to \mathbb{R}$ for its
\emph{support function}, defined as
\begin{displaymath}
  \Psi_{\NP(f)}(u)=\inf_{x\in \NP(f)} \langle u,x\rangle \quad \text{ for } u\in N_{\mathbb{R}}.
\end{displaymath}
Let then $\Sigma$ be a complete fan on $N_\mathbb{R}$ that is
\emph{compatible} with~$\NP(f)$, in the sense that the support function
of this polytope is linear on each of the cones of this fan.

Set $X$ for the toric variety over $K$ associated with $\Sigma$. It is
an equivariant compactification of the torus~$\mathbb{T}$ that is both
normal and complete.  It admits an open covering by the affine toric
varieties corresponding to the cones of the fan:
\begin{equation}
  \label{eq:10}
  X=\bigcup_{\sigma\in \Sigma}X_{\sigma}.
\end{equation}
The affine toric variety $X_{0}$ corresponding to the zero cone is
canonically isomorphic to the torus, and is called the \emph{principal
  open subset} of $X$.

Since the fan is compatible with the Newton polytope of $f$, we can
associate to this Laurent polynomial a nef toric divisor $D_{f}$ on
$X$. It is defined~as
\begin{equation*}
  D_{f}=(X_{\sigma},\chi^{-m_{\sigma}})_{\sigma\in \Sigma},
\end{equation*}
where for each $\sigma$ we denote by $m_{\sigma}\in M$ the slope of
$\Psi_{\NP(f)}$ on the cone $\sigma$. Set also
$L_{f}=\mathcal{O}(D_f)$ for the associated toric line bundle and
$s_{D_{f}}$ for its \emph{canonical rational section}. The first is the
subsheaf of the sheaf of rational functions $\mathcal{K}_{X}$ that is
generated on each affine chart $X_{\sigma}$ by the monomial
$\chi^{m_{\sigma}}$, and the second is the rational section of $L_{f}$ induced by
the global section $1$ of $\mathcal{K}_{X}$. We have that
$\div(s_{D_{f}})=D_{f}$.

We also set $s_{f}$ for the rational section of $L_{f}$ similarly induced
by the Laurent polynomial~$f$ considered as a global section of
$\mathcal{K}_{X}$. It is related to the canonical rational section by
$s_{f}=f s_{D_{f}}$. It follows from \cite[Theorem~4.3]{Gualdi} that
$s_f$ is a global section with Weil divisor equal to
the closure in $X$ of the Weil divisor of $\mathbb{T}$ defined by~$f$,
that is  
\begin{displaymath}
s_{f}\in \Gamma(X,L_{f}) \and  Z_{X}(s_{f})=  \overline{Z_{\mathbb{T}}(f)}.
\end{displaymath}

The action of $\mathbb{T}$ on $L_{f}$ induces an action on the
global sections of this line bundle, which can be described at the
level of closed points as follows. For $t\in \mathbb{T}({K})$ we now
denote by $\tau_{t}\colon X\to X$ the translation-by-$t$ map on the
toric variety and by
\begin{math}
  \tau_{t}^{*}\colon \mathcal{K}_{X}\to  \mathcal{K}_{X}
\end{math}
the corresponding pullback morphism.  We have that $L_{f}$ is
invariant with respect to the latter, and so we can pullback global
sections by $\tau_{t}$.

\begin{definition}
  \label{def:3}
  For  $s \in \Gamma(X, L_{f})$ and $t\in \mathbb{T}(K)$, the
  \emph{twist} of $s$ by $t$ is the pullback global section
  $\tau_{t}^{*}s$. It is denoted as $t^{*}s$ for convenience.
\end{definition}

For the global section 
$s_{f}$ we have that
\begin{equation}\label{eq: twist of section}
t^*s_{f}= (t^{*} f ) s_{D_f}.
\end{equation}

The next lemma describes the Weil divisor defined by $s_{f}$ in terms
of the orbit decomposition of the toric variety. To state this
properly, we recall that to each cone $\sigma\in \Sigma$ we can attach the split algebraic
torus
\begin{displaymath}
O(\sigma) =\Spec(K[\sigma^{\bot}\cap M]),  
\end{displaymath}
where $\sigma^{\bot} \subset M_{\mathbb{R}}$ denotes the orthogonal
linear subspace.  On the other hand, the corresponding affine toric
variety in \eqref{eq:10} is $X_{\sigma}=\Spec(K[\sigma^{\vee}\cap M])$
where ${\sigma^{\vee}\subset M_{\mathbb{R}}}$ denotes the dual cone. In
this situation, the inclusion $\sigma^{\bot}\subset \sigma^{\vee}$
induces the homomorphism of algebras
$K[\sigma^\vee\cap M]\rightarrow K[\sigma^{\bot}\cap M]$ given by
\begin{align*}
\chi^m&\longmapsto
\begin{cases}
\chi^m&\text{if }m\in\sigma^\perp,
\\0&\text{otherwise},
\end{cases}
\end{align*}
and so a closed immersion
\begin{math}
\iota_{\sigma}\colon  O(\sigma)\hookrightarrow X_{\sigma}.
\end{math}

Varying the cones and identifying the corresponding tori with their
images in~$X$ gives the different orbits of the action of $\mathbb{T}$
on~$X$, producing a decomposition
\begin{displaymath}
  X=\bigsqcup_{\sigma\in \Sigma} O(\sigma).
\end{displaymath}

To each $\sigma$ we can also associate the face of the Newton polytope
of $f$ defined as
\[
\NP(f)^{\sigma}=\{x\in\NP(f):\langle u,x\rangle=\Psi_{\NP(f)}(u) \text{ for all } u\in\sigma\}.
\]
It is contained in $m_{\sigma}+\sigma^\perp \subset M_\mathbb{R}$ for
the slope $m_{\sigma}$ of $\Psi_{\NP(f)}$ on this cone, and in fact it
is the intersection between this affine subspace
and~$\NP(f)$. Restricting the monomial expansion of $f$ to this face
and translating the exponents we obtain the nonzero Laurent polynomial
\[
f(\sigma)=\sum_{m\in \NP(f)^{\sigma}\cap M} \alpha_m\chi^{m-m_\sigma}\in K[\sigma^\perp\cap M].
\]

The inclusion $\sigma^\perp\cap M\subset M$ also induces an inclusion
of algebras $K[\sigma^\perp\cap M] \subset K[ M]$ and so
a homomorphism of tori
${\pi_{\sigma}\colon\mathbb{T}\to O({\sigma})}$.  These constructions
are related to the twisting by
\begin{equation}
  \label{eq:5}
(t^{*}f)(\sigma)=\chi^{m_{\sigma}}(t) \cdot (\pi_{\sigma}(t))^{*} (f(\sigma)) . 
\end{equation}

\begin{lemma}\label{lem: restriction of twist to orbits}
  For each $\sigma\in \Sigma$ we have that
  \begin{displaymath}
V_{X}(s_{f})\cap O(\sigma)= V_{O(\sigma)}(f(\sigma)).    
  \end{displaymath}
  In particular, for $t\in \mathbb{T}(K)$ we have that
  $V_{X}(t^{*}s_{f}) \cap O(\sigma)=
  V_{O(\sigma)}(\pi_{\sigma}(t)^{*}f(\sigma))$.
\end{lemma}

\begin{proof}
  On $X_{\sigma}$, the line bundle $L_{f}$ is generated by
  $\chi^{m_{\sigma}}$ and its global section $s_{f}$ is defined by
  $f$. Hence
\begin{math}
V_{X_{\sigma}}(s_{f})=V_{X_{\sigma}}(\chi^{-m_{\sigma}}f),   
\end{math}
where the latter denotes the zero set of the regular function
$\chi^{-m_{\sigma}}f $ on $X_{\sigma}$.  Restricting to the orbit
$O(\sigma)$ we get
\begin{displaymath}
  V_{X}(s_{f}) \cap {O(\sigma)}= V_{X_{\sigma}}(\chi^{-m_{\sigma}}f)  \cap {O(\sigma)}= V_{O({\sigma})}(\iota_{\sigma}^{*}(\chi^{-m_{\sigma}}f)) =V_{O(\sigma)}(f(\sigma)),
\end{displaymath}
proving the first statement. The second statement follows from the
first together with the functoriality in \eqref{eq:5}, noting that
$\chi^{m_{\sigma}}(t) \ne 0$.
\end{proof}

Using this result, we show that the zero set of a generic twist of the
global section~$s_{f}$ avoids any given finite set of points.  Recall
that a condition on the algebraic points of $\mathbb{T}$ is
\emph{generic} if it holds on a nonempty open subset of
$\mathbb{T}_{\overline{K}}$.

\begin{lemma}\label{lemma: generically twisted of divisors avoid a finite set of points}
  Let $p_{1}, \dots, p_{r} \in X_{\overline{K}}$ be a finite family of
  (not necessarily closed) points. Then for a generic choice
  of~$t\in\mathbb{T}(\overline{K})$ we have that
  $p_i\notin V_{X}(t^*s_f)$ for all $i$.
\end{lemma}

\begin{proof}
  We reduce without loss of generality to a single point
  $p\in X_{\overline{K}}$, since the general case follows from this
  one by intersecting the corresponding genericity
  conditions. Moreover, taking any closed point in the closure
  $\overline{\{p\}}$ we can also suppose that
  ${p \in X(\overline{K})}$.

  Take $\sigma\in \Sigma$ such that $p\in O(\sigma)$. By \cref{lem:
    restriction of twist to orbits}, for
  $t\in \mathbb{T}(\overline{K})$ we have that
  $p\notin V_{X}(t^{*}s_{f})$ if and only if
  $p\notin V_{O(\sigma)}(\pi_{\sigma}(t)^{*} f(\sigma))$. This is
  equivalent to the fact that
  \begin{displaymath}
f(\sigma)(\pi_{\sigma}(t)\cdot p) \ne 0,   
  \end{displaymath}
  which in turn is equivalent to
  $t\notin \pi_{\sigma}^{-1}(V_{O(\sigma)}(p^{*} f(\sigma)))$.  Since
  the Laurent polynomial $f(\sigma)$ is nonzero, this latter condition
  is open and nonempty, and so generic.
\end{proof}

For $0\le r\le n$, the \emph{degree} of an $r$-cycle $Z$ of $\mathbb{T}_{\overline{K}}$ with respect
to a family $D_1, \dots, D_{r}$ of divisors on~$X$ is defined by
taking its closure, that is
\begin{displaymath}
\deg_{D_{1},\dots, D_{r}}(Z)\coloneqq\deg_{D_{1},\dots, D_{r}}(\overline{Z}),
\end{displaymath}
where the latter degree is defined by considering the base change of
the divisors with respect to~$\overline{K}$.

Recall that to each toric divisor $D$ on $X$ we can associate a
lattice polytope ${\Delta_{D_{i}}\subset M_{\mathbb{R}}}$. For a family
$D_{1},\dots, D_{n}$ of nef toric divisors on $X$ we have that
\begin{equation}
  \label{eq:6}
  \deg_{D_{1},\dots, D_{n}}(\mathbb{T}_{\overline{K}})=\MV_{M}(\Delta_{D_{1}},\dots, \Delta_{D_{n}}),
\end{equation}
where $\MV_{M}$ denotes the mixed volume function for the Haar measure
$\vol_{M}$ on $M_{\mathbb{R}}$ that gives covolume $1$ to the lattice
$M$ and acts on families of $n$ convex bodies of this vector space
\cite[Definition 2.7.14]{BPS}.

\subsection{The dimension of intersection cycles of twists}
\label{sec:dimens-inters-cycl}

Here we study the intersection cycles defined by the twists of several Laurent
polynomials.  We start by considering an integer $0\le k\le n$ and
a family of nonzero Laurent polynomials in~$K[M]$
\begin{displaymath}
\boldsymbol{f}= (f_1,\ldots,f_k).
\end{displaymath}
Fix a complete fan $\Sigma$ on $N_\mathbb{R}$ which is {compatible}
with the Newton polytopes of these Laurent polynomials and set $X$ for
the associated complete toric variety. This compatibility allows to
consider the family $s_{\boldsymbol{f}}=(s_{f_{1}}, \dots, s_{f_{k}})$
where $s_{f_{i}}$ is the global section of the nef toric line
bundle $L_{f_{i}}$ on~$X$ associated with $f_{i}$.

Let
$\mathbb{T}^k$ be the product of $k$-many copies of the
torus~$\mathbb{T}$, and for $\boldsymbol{t}\in\mathbb{T}(\overline{K})^{k}$ set
\begin{displaymath}
  \boldsymbol{t}^{*}\boldsymbol{f}=(t_{1}^{*}f_{1}, \dots, t_{k}^{*}f_{k})
  \and   \boldsymbol{t}^{*}s_{\boldsymbol{f}}=(t_{1}^{*}s_{f_{1}}, \dots, t_{k}^{*}s_{f_{k}})
\end{displaymath}
for the corresponding families of twists over $\overline{K}$
(\cref{def:4,,def:3}).  

The following is the main result of this section, realizing the
closure of the intersection cycle of the torus defined by a generic
twist of $\boldsymbol{f}$ as the intersection cycle of the toric
variety defined by the same twist of $s_{\boldsymbol{f}}$.

\begin{theorem}\label{thm: geometric main theorem}
  There is a proper closed subset
  $H\subset \mathbb{T}^{k}$ such that for
  $\bm{t} \in (\mathbb{T}^k\setminus H)(\overline{K})$ we have that
  $t_1^*s_{f_1},\ldots,t_k^*s_{f_k}$ meet properly and that
  $ \overline{Z_{\mathbb{T}}(\bm{t}^*\bm{f})} =
  Z_{X}(\boldsymbol{t}^{*} s_{\boldsymbol{f}}).  $
\end{theorem}

We first prove the next auxiliary result.

\begin{lemma}\label{lemma: exists fiber of good dimension}
  There is $\boldsymbol{t}\in\mathbb{T}(\overline{K})^{k}$ such that every
  irreducible component of the zero set
  $V_{X}(\boldsymbol{t}^{*}s_{\boldsymbol{f}})$ has codimension $k$
  and is not contained in $X\setminus X_{0}$.
\end{lemma}

\begin{proof} 
  We proceed by induction on~$k$.  The case~$k=0$ is tautological, and
  so we suppose that $k\ge 1$.
  The inductive step follows by applying \cref{lemma: generically
    twisted of divisors avoid a finite set of points} to the generic
  point of each of the components of
  $V_{X}(t_{1}^{*}s_{f_{1}}, \dots, t_{k-1}^{*}s_{f_{k-1}})$ and of
  the intersection of this zero set with $X\setminus X_{0}$. As a
  consequence, there is ${t_{k}\in \mathbb{T}(\overline{K})}$ such that
  $V_{X}(t_{k}^{*}s_{f_{k}})$ does not contain any of the components
  of $V_{X}(t_{1}^{*}s_{f_{1}}, \dots,
  t_{k-1}^{*}s_{f_{k-1}})$. Setting
  $\boldsymbol{t}=(t_{1},\dots, t_{k-1},t_{k})$ we have that every
  component of
\begin{displaymath}
V_{X}(\boldsymbol{t}^{*}s_{\boldsymbol{f}})=   V_{X}(t_{1}^{*}s_{f_{1}}, \dots, t_{k-1}^{*}s_{f_{k-1}})\cap V_{X}(t_{k}^{*}s_{f_{k}})
\end{displaymath}
has codimension $k$,
  as stated.

  Moreover, by the inductive hypothesis all the  components
  of the intersection of
  $V_{X}(t_{1}^{*}s_{f_{1}}, \dots, t_{k-1}^{*}s_{f_{k-1}})$ with
  $X\setminus X_{0}$ have codimension $k$. By construction of $t_{k}$,
  none of them is contained in
  $V_{X}(\boldsymbol{t}^{*}s_{\boldsymbol{f}})$, and so this zero set
  has no  component contained in $X\setminus X_{0}$.
\end{proof}

\begin{proof}[Proof of  \cref{thm: geometric main theorem}]
  We first show that the conditions in \cref{lemma: exists fiber of
    good dimension} are generic. To this end, for each $i$ denote by
  $\mu_{i}\colon \mathbb{T}^{k}\times X \to X$ the {$i$-th
    multiplication map} defined~by
  \begin{displaymath}
(t_{1},\dots, t_{k}, p)\longmapsto t_{i}\cdot p.   
  \end{displaymath}
  The pullback $\mu_{i}^{*}s_{f_{i}}$ is a global section of the line
  bundle $ \mu_{i}^{*}L_{f_{i}}$ on $\mathbb{T}^{k}\times X$, and we
  put 
  \begin{displaymath}
    \Omega=V_{\mathbb{T}^{k}\times X}(\mu_{1}^{*}s_{f_{1}}, \dots, \mu_{k}^{*}s_{f_{k}}). 
  \end{displaymath}
  This a closed set of $ \mathbb{T}^{k}\times X$ whose algebraic
  points are the pairs
  $(\boldsymbol{t},p) \in (\mathbb{T}^{k}\times
  X)(\overline{K})=\mathbb{T}(\overline{K})^{k}\times X(\overline{K})
  $ such that
  \begin{displaymath}
s_{f_{1}}(t_{1}\cdot p)=\cdots = s_{f_{k}}(t_{k}\cdot p)=0.
  \end{displaymath}
  Hence $\Omega$ is the incidence closed subset of
  the global sections $s_{f_{i}}$, $i=1,\dots, k$, with respect to the
  action of the torus on the toric variety.

For
  $\boldsymbol{t}\in \mathbb{T}(\overline{K})^{k}$ the fiber of the
  projection map $\pi\colon   \mathbb{T}^{k}\times X\to \mathbb{T}^{k}$ restricted to
  this closed subset identifies with the zero set of the family of
  twisted global sections $\boldsymbol{t}^{*}s_{\boldsymbol{f}}$, that
  is
  \begin{equation*}
    (\pi|_{\Omega})^{-1}(\boldsymbol{t})=\{\boldsymbol{t}\}\times V_{X}(\boldsymbol{t}^{*}s_{\boldsymbol{f}}).
  \end{equation*}
  
  Let $W$ be an irreducible component of $\Omega$.
  We have that the map $\pi$ is closed because $X$ is complete, and so
  the image $\pi(W)\subset \mathbb{T}^{k}$ is a closed subset.

  On the one hand, if $\pi(W)\ne \mathbb{T}^{k}$ then
  $\pi(W) \subset \mathbb{T}^{k}$ is a proper closed subset and so
  $ (\pi|_{W})^{-1}(\boldsymbol{t})=\emptyset$ for $\boldsymbol{t}$
  generic.  On the other hand, if $\pi(W)=\mathbb{T}^{k}$ then the
  restriction
  \begin{equation}
    \label{eq:9}
\pi|_{W}\colon W\longrightarrow \mathbb{T}^{k}
  \end{equation}
  is surjective.  As $\Omega$ is defined by $k$
  equations, dimension theory implies that the codimension of $W$
  cannot exceed~$k$. Moreover, by the first part of \cref{lemma:
    exists fiber of good dimension} the map in~\eqref{eq:9} has a
  fiber of dimension $n-k$, and by the theorem of dimension of fibers
  \cite[Exercise 3.22 at page 95]{Har} this quantity bounds above the
  relative dimension of $W$ over $\mathbb{T}^{k}$. Altogether this
  implies that
\[
n-k\leq\dim(W)-\dim(\mathbb{T}^k)\leq n-k.
\]
It follows that this relative dimension is equal to~$n-k$, and so the
theorem of dimension of fibers also gives that
$\dim((\pi|_{W})^{-1}(\boldsymbol{t}))= n-k$ for $\boldsymbol{t}$
generic.

Note that
\begin{displaymath}
 (\pi|_{\Omega})^{-1}(\boldsymbol{t})=
\bigcup_{W}(\pi|_{W})^{-1}(\boldsymbol{t}), 
\end{displaymath}
the union being over the irreducible components of $\Omega$. Hence
intersecting the genericity conditions corresponding to each $W$ we get
that $V_{X}(\boldsymbol{t}^{*}s_{\boldsymbol{f}})$ is either empty or
has pure codimension $k$ for $\boldsymbol{t}$ generic.

Applying the same reasoning to the intersection
$\Omega \cap (\mathbb{T}^{k}\times (X\setminus X_{0}))$ and using the
second part of \cref{lemma: exists fiber of good dimension} we deduce
that
$V_{X}(\boldsymbol{t}^{*}s_{\boldsymbol{f}}) \cap (X\setminus X_{0}) $
is either empty or has pure codimension $k+1$ in $X$ for a generic
choice of $\boldsymbol{t}$. We obtain that
$V_{X}(\boldsymbol{t}^{*}s_{\boldsymbol{f}})$ has no irreducible
component contained in $X\setminus X_{0}$ for this choice of
$\boldsymbol{t}$.

This proves that the conditions in \cref{lemma: exists fiber of good
  dimension} applied to the family 
$s_{f_{1}},\dots, s_{f_{k}}$ hold for $\boldsymbol{t}$
generic. Applying this to the global sections corresponding to each
index subset $I\subset \{1,\dots, k\}$ and intersecting the resulting
genericity conditions we get that a generic twist of the global
sections meets properly, as stated.

Summing up, for a generic choice of $\boldsymbol{t}$ the intersection
cycle $Z_{X}(\boldsymbol{t}^{*}s_{\boldsymbol{f}})$ can be defined and
has no irreducible component contained in $X\setminus X_{0}$. In
particular, it coincides with the closure of
$Z_{\mathbb{T}}(\boldsymbol{t}^{*}\boldsymbol{f})$ in $X$, as stated.
\end{proof}

We next use this result to compute the degree of the intersection
cycle of the torus defined by a generic twist of $\boldsymbol{f}$.

\begin{corollary}\label{cor: geometric degree of cycle}
  Let $D_{1},\dots, D_{n-k}$ be nef toric divisors on~$X$ and  $H\subset \mathbb{T}^{k}$ the proper closed subset in \cref{thm: geometric main theorem}. 
  Then for  $\bm{t}\in(\mathbb{T}^k\setminus H)(\overline{K})$ we have that
\[
\deg_{D_1,\ldots,D_{n-k}}(Z_{\mathbb{T}}(\bm{t}^*\bm{f}))
=
\MV_M(\Delta_{D_{1}},\ldots,\Delta_{D_{n-k}},\NP(f_1),\ldots,\NP(f_{k})).
\]
\end{corollary}

\begin{proof}
  By \cref{thm: geometric main theorem}, for a generic choice of
  $\bm{t}\in\mathbb{T}(\overline{K})^{k}$ we have that
\begin{displaymath}
\deg_{D_1,\ldots,D_{n-k}}(Z_{\mathbb{T}}(\bm{t}^*\bm{f}))
=  \deg_{D_1,\ldots,D_{n-k}}( \div(t_{1}^{*}s_{f_{1}}) \cdots  \div(t_{n-k}^{*}s_{f_{n-k}}) \cdot X).
\end{displaymath}
By B\'ezout's formula, this latter degree coincides with
$ \deg_{D_{1},\dots, D_{n-k}, D_{f_{1}}, \dots,
  D_{f_{k}}}(\mathbb{T}_{\overline{K}})$ where  $D_{f_{i}}$ is the nef
toric divisor on $X$  associated with  $f_{i}$ for each
$i$. The polytope of $D_{f_{i}}$ coincides with $\NP(f_{i})$ and so
the statement follows from \eqref{eq:6}.
\end{proof}

Following \cite{Bilu}, a sequence of algebraic points of a torus is
said to be \emph{strict} if it eventually avoids any fixed proper
algebraic subgroup. In this subsection we slightly relax this notion
to one that suits better to twists of systems of equations.

Consider the diagonal embedding
$\mathbb{T}\hookrightarrow \mathbb{T}^{k}$ defined by
$t\mapsto (t,\dots, t)$. The quotient $\mathbb{T}^{k}/ \mathbb{T}$ is
isomorphic to the torus $\mathbb{T}^{k-1}$ through  the map
\begin{equation}
  \label{eq:7}
  [(t_{1},t_{2},\dots, t_{k})]\longmapsto (t_{1}^{-1}t_{2},\dots, t_{1}^{-1}t_{k}). 
\end{equation}
Set $ \varpi\colon \mathbb{T}^{k}\to \mathbb{T}^{k}/ \mathbb{T}$ for
the quotient homomorphism.

\begin{definition}
  \label{def:1}
  A sequence $(\boldsymbol{t}_{\ell})_{\ell}$ in
  $ \mathbb{T}(\overline{K})^{k}$ is \emph{quasi-strict} if for every
  proper algebraic subgroup $G\subset \mathbb{T}^{k}/ \mathbb{T}$
  there is $\ell_{0}\in \mathbb{N}$ such that
  $\varpi(\boldsymbol{t}_{\ell})\notin G(\overline{K})$ for all
  $\ell \ge \ell_{0}$.
\end{definition}

Notice that any strict sequence in $\mathbb{T}(\overline{K})^{k}$ is
also quasi-strict.  In fact, a sequence
$(\boldsymbol{t}_{\ell})_{\ell}$ in $ \mathbb{T}(\overline{K})^{k}$ is
quasi-strict if and only if the sequence
\begin{math}
  ((t_{\ell,1}^{-1}t_{\ell,2}, \dots, t_{\ell,1}^{-1}t_{\ell,k}))_{\ell}
\end{math}
in $\mathbb{T}(\overline{K})^{k-1}$ is strict.

The next result is a direct consequence of \cref{thm: geometric main
  theorem} and \cref{cor: geometric degree of cycle} together with the
former toric Manin-Mumford conjecture \cite{Laurent}.
 
\begin{corollary}\label{cor: proper intersection for strict sequences}
  Let $(\bm{\omega}_\ell)_{\ell}$ be a quasi-strict sequence of
  torsion points of $\mathbb{T}(\overline{K})^{k}$.  Then for $\ell$ sufficiently
  large the global sections
  $\bm{\omega}_{\ell,1}^*s_{f_1},\ldots,\bm{\omega}_{\ell,k}^*s_{f_k}$
  meet properly and
\[
\overline{Z_{\mathbb{T}}(\bm{\omega}^*_\ell\bm{f})}
= 
Z_{X}(\boldsymbol{\omega}_{\ell}^{*} s_{\boldsymbol{f}}). 
\]
In particular, for a family $D_{1},\dots, D_{n-k}$ of nef toric divisors
on~$X$ we have that
\[
  \deg_{D_1,\ldots,D_{n-k}}(Z_{\mathbb{T}}(\bm{\omega}_\ell^*\bm{f})) =
  \MV_M(\Delta_{D_{1}},\ldots,\Delta_{D_{n-k}},\NP(f_1),\ldots,\NP(f_{k}))
\]
for $\ell$ sufficiently large.
\end{corollary}

\begin{proof}
  Let $H \subset \mathbb{T}^{k}$ be the minimal closed subset
  satisfying the conditions in \cref{thm: geometric main theorem}.
  One can check that these conditions are invariant under the diagonal
  action of $\mathbb{T}(\overline{K})$. This implies that
  $\varpi(H)\subset \mathbb{T}^{k}/\mathbb{T}$ is a proper closed
  subset, and so by Laurent's theorem \cite{Laurent}
    there is a finite union of proper algebraic subgroups of
    $\mathbb{T}^{k}/\mathbb{T}$ that contains all torsion points of
    $\varpi(H)$.

    The sequence $(\bm{\omega}_\ell)_{\ell}$ projects to a strict
    sequence of torsion points of
    $(\mathbb{T}^{k}/\mathbb{T})(\overline{K})$.  Hence this strict
    sequence eventually avoids $\varpi(H)$, and so
    $(\bm{\omega}_\ell)_{\ell}$ eventually avoids $H$.
    The statement then follows readily from
  \cref{thm: geometric main theorem} and \cref{cor: geometric degree
    of cycle}.
\end{proof}

\section{The height of intersection cycles of twists}
\label{sec:height-inters-cycl}

In this section we move to the arithmetic setting, approaching the
main content of the paper.  After recalling from \cite{BPS, Gualdi}
the notions and objects playing a role in the statement of \cref{conj:
  main conjecture} we present a strategy for its proof, which reduces
it to a local logarithmic equidistribution statement and a global
adelic one.  We also provide several reduction steps that can be
employed for the proof, relegating the burdensome details to
\cref{sec: appendix A}.

\subsection{Arithmetic toric constructions}
\label{sec:arithm-toric-constr}

Throughout this section we denote by $\mathbb{T}$ a split algebraic
torus over $\mathbb{K}$ of dimension $n\ge 0$ and $X$ a complete toric
variety compactifying $\mathbb{T}$. Recall that for each 
$v\in \mathfrak{M}$ we respectively denote by $\mathbb{T}_v$ and
$X_{v}$ the base change of the torus and the toric variety with
respect to the complete and algebraically closed field
$\mathbb{C}_{v}$, and we set $\mathbb{T}_{v}^{\an} $ and $X_{v}^{\an}$
for their analytification. The \emph{$v$-adic compact torus} is the
subset of $\mathbb{T}_{v}^{\an}$ defined as
\begin{equation}\label{eq: compact torus}
  \mathbb{S}_v=\big\{t\in\mathbb{T}_v^{\an}:|\chi^m(t)|_v=1 \text{ for all }m\in M\big\}.
\end{equation}
It is an analytic group acting on $X_{v}^{\an}$.

We also denote by
$\val_v\colon\mathbb{T}_v^{\an}\rightarrow N_\mathbb{R}$ the
\emph{$v$-adic valuation map} \cite[Section 4.1]{BPS}. Choosing a
splitting $\mathbb{T}\simeq \Gm^{n}$ identifies $N_{\mathbb{R}}$ with
$ \mathbb{R}^{n}$ and the set of rigid points of
$\mathbb{T}_{v}^{\an}$ with $(\mathbb{C}_{v}^{\times})^{n}$.  In these
coordinates, the $v$-adic valuation map writes down at the level of
these rigid points~as
\begin{displaymath}
  \val_{v}(t_{1},\dots, t_{n})=(-\log|t_{1}|_{v},\dots, -\log|t_{n}|_{v}).
\end{displaymath}

Let $\overline{D}$ be a \emph{semipositive toric (adelically) metrized
divisor} on $X$. It is the datum
\[
\overline{D}=\big(D,(\|\cdot\|_v)_{v\in\mathfrak{M}}\big),
\]
where $D$ is a toric divisor on~$X$ and each $\|\cdot\|_v$ is a
semipositive $\mathbb{S}_{v}$-invariant and continuous metric on the
analytic line bundle $\mathcal{O}(D)_{v}^{\an}$ that is induced by a
single integral model of the pair $(X,O(D))$ for all but finitely many
$v\in \mathfrak{M}$ \cite[Definition~4.9.1]{BPS}.

Let $s_{D}$ be the canonical rational section of the line bundle
$\mathcal{O}(D)$ associated with~$D$.  For each $v$ the function
$\log \|s_{D}\|_v \colon \mathbb{T}_v^{\an}\to \mathbb{R}$ factors
through the valuation map. In~fact,
\begin{displaymath}
\log \|s_{D}\|_v=\psi_{\overline{D},v}\circ \val_{v}  
\end{displaymath}
for a concave function
$\psi_{\overline{D},v}\colon N_{\mathbb{R}}\to \mathbb{R}$, called the
\emph{$v$-adic metric function} of $\overline{D}$. The
Legendre--Fenchel dual of $\psi_{\overline{D},v}$ is a continuous
concave function 
\begin{displaymath}
  \vartheta_{\overline{D},v}\colon \Delta_{D}\longrightarrow \mathbb{R},
\end{displaymath}
where $\Delta_{D} \subset M_{\mathbb{R}}$ is the polytope associated
with $D$. It is called the \emph{$v$-adic roof function} of
$\overline{D}$. The adelicity of $\overline{D}$ implies that
$\vartheta_{\overline{D},v}$ is the zero function on $\Delta_{D}$ for
all except finitely many $v\in \mathfrak{M}$.

For $0\le r\le n$, the \emph{height} of an $r$-cycle $Z$ of
$\mathbb{T}_{\overline{\mathbb{K}}}$ with respect to a family of
semipositive metrized divisors
$\overline{D}_{0},\dots, \overline{D}_{r}$ on $X$ is defined through
its closure as
\begin{displaymath}
  \h_{\overline{D}_{0},\dots, \overline{D}_{r}}(Z)=  \h_{\overline{D}_{0},\dots, \overline{D}_{r}}(\overline{Z}).
\end{displaymath}
The latter is defined by considering the base change of these metrized
divisors to any finite extension of $\mathbb{K}$ over which
$\overline{Z}$ is defined, and applying the recursive definition of
the height in \cite[Chapter 1]{BPS}.

The following result due to Burgos, Philippon and the second author
expresses the height of the torus with respect to a family of
semipositive toric metrized divisors in convex-analytic terms. It
involves the {mixed integral} operator $\MI_{M}$ with respect to the
Haar measure $\vol_{M}$ on $M_{\mathbb{R}}$ acting on families of
$n+1$ concave functions on convex bodies of $M_{\mathbb{R}}$, see for
instance \cite[Definition 2.7.16]{BPS}.

\begin{theorem}[{\cite[Theorem~5.2.5]{BPS}}]\label{thm: BPS main}
  Let $\overline{D}_{0},\ldots,\overline{D}_{n}$ be
  semipositive toric metrized divisors on $X$. Then
\[
\h_{\overline{D}_{0},\ldots,\overline{D}_{n}}( \mathbb{T}_{\overline{\mathbb{K}}})=\sum_{v\in\mathfrak{M}}n_v\MI_M(\vartheta_{\overline{D}_{0},v},\ldots,\vartheta_{\overline{D}_{n},v}).
\]
\end{theorem}

Let now $f\in\mathbb{K}[M]$ be a nonzero Laurent polynomial.  When
investigating the height of its Weil divisor it is necessary to consider an
adelic family of functions measuring the arithmetic complexity of this
Laurent polynomial.  To this end, for each $v\in\mathfrak{M}$ we
consider the \emph{$v$-adic Ronkin function} of $f$, which is a
concave function
\begin{displaymath}
  \rho_{f,v}\colon N_\mathbb{R}\longrightarrow \mathbb{R}  
\end{displaymath}
defined as the average of
$\log |f|_{v} \colon \mathbb{T}_v^{\an}\to \mathbb{R}\cup\{-\infty\}$
on the fibers of the $v$-adic valuation map \cite[Section 2]{Gualdi}.
Its Legendre--Fenchel dual $\rho_{f,v}^\vee$ is a continuous concave
function on the Newton polytope of~$f$.

\begin{theorem}[{\cite[Theorem~5.12]{Gualdi}}]\label{thm: Gualdi thesis main}
  Let $f\in\mathbb{K}[M]$ be a nonzero Laurent polynomial and
  $\overline{D}_{0},\ldots,\overline{D}_{n-1}$ semipositive toric
  metrized divisors on $X$.  Then
\[
  \h_{\overline{D}_{0},\ldots,\overline{D}_{n-1}}(Z_\mathbb{T}(f))=\sum_{v\in\mathfrak{M}}n_v\MI_M(\vartheta_{\overline{D}_{0},v},\ldots,\vartheta_{\overline{D}_{n-1},v},\rho_{f,v}^\vee).
\]
\end{theorem}

\subsection{An approach to \cref{conj: main conjecture}}
\label{sec:an-appr-crefc}

Our main concern in this article is the height of higher codimensional
complete intersection cycles in~$X$. To this end, for an integer
$0\le k\le n$ fix a family $\boldsymbol{f}=(f_{1},\dots, f_{n-k})$ of
nonzero Laurent polynomials in $\mathbb{K}[M]$ and a family
$\overline{D}_0,\ldots,\overline{D}_{n-k}$ of semipositive toric
metrized divisors on~$X$. \cref{conj: main conjecture} in the
introduction predicts that
\begin{equation*}
\lim_{\ell\to\infty}\h_{\overline{D}_0,\ldots,\overline{D}_{n-k}}(Z_{\mathbb{T}}(\bm{\omega}_\ell^*\bm{f}))
=
\sum_{v\in\mathfrak{M}}n_v\MI_M(\vartheta_{\overline{D}_{0},v},\ldots,\vartheta_{\overline{D}_{n-k},v},\rho_{f_1,v}^\vee,\ldots,\rho_{f_k,v}^\vee)
\end{equation*}
for any quasi-strict sequence $(\bm{\omega}_\ell)_\ell$ of torsion
points of~$\mathbb{T}(\overline{\mathbb{K}})^{k}$.

\begin{remark}
  \label{rem:3}
For $k=0$ this claim boils down to \cref{thm: BPS main}, while for
$k=1$ it coincides with \cref{thm: Gualdi thesis main} noticing that
\[
\h_{\overline{D}_0,\ldots,\overline{D}_{n-1}}(Z_{\mathbb{T}}(\omega^*f))
=
\h_{\overline{D}_0,\ldots,\overline{D}_{n-1}}(Z_{\mathbb{T}}(f))
\]
for any torsion point~$\omega\in\mathbb{T}(\overline{\mathbb{K}})$, which follows from the arithmetic projection formula and the definition of toric metrics.
\end{remark}

We now outline a strategy for proving \cref{conj: main conjecture},
arguing by induction on the number $k$ of Laurent polynomials.

As the
conjecture holds true for $k=0$ and $k=1$ (\cref{rem:3}), we just need
to prove the inductive step.
Let~$k\geq2$ and assume for convenience that the fan of $X$ is
compatible with the Newton polytopes of~$f_1,\ldots,f_k$. As we will
later see, this can always be done without loss of generality
(\cref{prop: reduction steps}).  Under this compatibility condition,
for each Laurent polynomial $f_i$ we consider the nef toric divisor
$D_{f_i}$ on $X$ and the global section $s_{f_i}$ of the toric line
bundle~$\mathcal{O}(D_{f_i})$.  Following \cite[Section~5]{Gualdi}, we
can enrich the divisor $D_{f_{i}}$ with its adelic family of $v$-adic
Ronkin metrics. The obtained pair
\begin{equation}\label{eq: Ronkin metric}
\overline{D}^{\Ron}_{f_{i}}=(D_{f_i},(\|\cdot\|_{\Ron,v})_{v\in\mathfrak{M}})
\end{equation}
is a semipositive toric metrized divisor on $X$ whose $v$-adic metric
function coincides with the Ronkin function $\rho_{f_{i},v}$.

Let~$H\subset\mathbb{T}^k$ be the proper closed subset given by
\cref{thm: geometric main theorem} and take a family of torsion points
$\bm{\omega}=(\omega_{1},\dots,
\omega_{k})\in\mathbb{T}(\overline{\mathbb{K}})^{k}$ not lying in
it. By this result, the twisted global sections
$\omega_1^*s_{f_1},\ldots,\omega_k^*s_{f_k}$ meet properly and
\[
\overline{Z_{\mathbb{T}}(\bm{\omega}^*\bm{f})}
=
Z_X(\bm{\omega}^*s_{\bm{f}})
=
\div(\omega_k^*s_{f_k})\cdot Z_X\big(\widehat{\boldsymbol{\omega}}^{*} s_{\widehat{\boldsymbol{f}}}\big)
\]
with $\widehat{\boldsymbol{f}}=(f_{1},\dots, f_{k-1})$ and
$\widehat{\boldsymbol{\omega}}=(\omega_{1},\dots, \omega_{k-1})$.
Applying \cref{thm: geometric main theorem} to the
family~$\widehat{\boldsymbol{f}}$, up to possibly enlarging $H$ we
also have that
$\overline{Z_{\mathbb{T}}(\widehat{\bm{\omega}}^*\widehat{\bm{f}})} =
Z_X(\widehat{\bm{\omega}}^*s_{\widehat{\bm{f}}})$. Hence
\[
\overline{Z_{\mathbb{T}}(\bm{\omega}^*\bm{f})}
=
\div(\omega_k^*s_{f_k})\cdot
\overline{Z_{\mathbb{T}}(\widehat{\bm{\omega}}^*\widehat{\bm{f}})}.
\]

This expression is particularly favourable to compute the height of
this cycle of $X_{\overline{\mathbb{K}}}$ using the recursive
definition of the height.  To do so, let us consider the finite
extension~$\mathbb{K}(\bm{\omega})/\mathbb{K}$ and denote by
$\mathfrak{M}({\bm{\omega}})$ its set of places.  Each
$w \in \mathfrak{M}(\boldsymbol{\omega})$ restricts to a place
$v\in \mathfrak{M}$, and its weight is
\begin{equation}\label{eq: weight of places in field extension}
  n_w=\frac{[\mathbb{K}(\bm{\omega})_w:\mathbb{K}_v]}{[\mathbb{K}(\bm{\omega}):\mathbb{K}]} \, n_v
\end{equation}
by the multiplicativity of degrees of finite extensions.  Applying the
arithmetic B\'ezout formula over the number
field~$\mathbb{K}(\bm{\omega})$ yields
\begin{equation}\label{eq: recursive definition of height in the general strategy}
\h_{\overline{D}_0,\ldots,\overline{D}_{n-k}}(Z_{\mathbb{T}}(\bm{\omega}^*\bm{f}))
=
\h_{\overline{D}_0,\ldots,\overline{D}_{n-k},\overline{D}^{\Ron}_{f_{k}}}({Z_{\mathbb{T}}(\widehat{\bm{\omega}}^*\widehat{\bm{f}})})
\ +
\sum_{w\in\mathfrak{M}(\bm{\omega})}n_wJ_w(\bm{\omega}),
\end{equation}
where    for each~$w\in\mathfrak{M}(\bm{\omega})$ we have set
\begin{displaymath}
J_w(\bm{\omega})
=
\int_{X_w^{\an}}\log\|\omega_k^*s_{f_k}\|_{\Ron,w}\ \chern(\overline{D}_{0,w})\wedge\ldots\wedge\chern(\overline{D}_{n-k,w})
\wedge\delta_{Z_X( \widehat{\boldsymbol{\omega}}^*s_{\widehat{\boldsymbol{f}}})_w^{\an}}.
\end{displaymath}
Here the integral is with respect to the $w$-adic mixed
Monge--Amp\`ere measure of the extension to
$\mathbb{K}(\boldsymbol{\omega})$ of the metrized divisors, restricted
to the analytic $(n-k+1)$-cycle defined by the first $k-1$ twisted
global sections.

We can write \eqref{eq: recursive definition of height in the general
  strategy} in a friendlier way by grouping together the local
$w$-adic terms according to their restriction to~$\mathbb{K}$ and
reformulating the resulting sums in term of Galois orbits.

For $t\in\mathbb{T}(\overline{\mathbb{K}})$ we denote by
$O(t) \subset \mathbb{T}(\overline{\mathbb{K}})$ its \emph{Galois
  orbit}, that is the orbit of this algebraic point  under the action of the absolute Galois
group of $\mathbb{K}$.
To pass from the algebraic to the analytic setting, we need to choose a $\mathbb{K}$-embedding
\[
\iota_v\colon\overline{\mathbb{K}}\longhookrightarrow\mathbb{C}_v.
\]
This choice induces a composite injective map
$\mathbb{T}(\overline{\mathbb{K}})\rightarrow\mathbb{T}_v(\mathbb{C}_v)\rightarrow\mathbb{T}_v^{\an}$
that we also denote by~$\iota_v$. The \emph{$v$-adic Galois orbit} of
$t\in\mathbb{T}(\overline{\mathbb{K}})$ is now defined as the finite
set
\[
O(t)_v=\iota_v(O(t)) \subset \mathbb{T}_{v}^{\an}.
\]
Since the different $\mathbb{K}$-embeddings of $\overline{\mathbb{K}}$
into $\mathbb{C}_v$ differ by an element
of~$\Gal(\overline{\mathbb{K}}/\mathbb{K})$, the set $O(t)_v$ does not
depend on the choice of the embedding~$\iota_v$.  Moreover, its
cardinality coincides with that of~$O(t)$.

We also consider the function
$I_v\colon(\mathbb{T}^k\setminus
H)(\mathbb{C}_v)\rightarrow\mathbb{R}$ defined for
${\boldsymbol{t}}=(t_{1},\dots, t_{k})$ as
\[
  I_v(\bm{t})=\int_{X_{v}^{\an}}\log\|t_{k}^*s_{f_k}\|_{\Ron,v}\ \chern(\overline{D}_{0,v})\wedge\ldots\wedge\chern(\overline{D}_{n-k,v})\wedge
\delta_{Z_X( \widehat{\boldsymbol{t}}^*s_{\widehat{\boldsymbol{f}}})_v^{\an}} 
\]
with $\widehat{\boldsymbol{t}}=(t_{1},\dots, t_{k-1})$.

\begin{claim}\label{lem: from places to Galois}
  We have that 
$\displaystyle{
\sum_{w\in\mathfrak{M}(\bm{\omega})}n_wJ_w(\bm{\omega})
=
\sum_{v\in\mathfrak{M}}\frac{n_{v}}{\# O(\boldsymbol{\omega})}\sum_{\boldsymbol{\eta}\in O(\boldsymbol{\omega})_{v}} I_{v}(\boldsymbol{\eta}).}$
\end{claim}

\begin{proof}
  Let $v\in\mathfrak{M}$ and denote by
  $\mathfrak{M}(\boldsymbol{\omega})_{v}$ the set of places of
  $\mathbb{K}(\bm{\omega})$ over $v$. The finite extension
  $\mathbb{K}(\bm{\omega})/\mathbb{K}$ is Galois and we denote by
  $G_{\boldsymbol{\omega}}$ its Galois group.  This finite group acts
  by composition on the set $\mathfrak{M}(\boldsymbol{\omega})_{v}$,
  and by \cite[Chapter II, Proposition 9.1]{Neukirch:ant} this action
  is transitive.

  In particular, the degree
  $[\mathbb{K}(\bm{\omega})_w:\mathbb{K}_v]$ is the same for all
  $w\in \mathfrak{M}(\boldsymbol{\omega})_{v}$, and by \cite[Chapter
  II, Corollary 8.4]{Neukirch:ant} these degrees sum up to
  $[\mathbb{K}(\bm{\omega}):\mathbb{K}]$.  Combining this with the
  orbit-stabilizer theorem we obtain that
  \begin{equation}
    \label{eq:12}
\frac{[\mathbb{K}(\bm{\omega})_w:\mathbb{K}_v]}{[\mathbb{K}(\bm{\omega}):\mathbb{K}]}=
\frac{1}{\# \mathfrak{M}(\boldsymbol{\omega})_{v}} = \frac{\#\stab(w)}{\#G_{\bm{\omega}}} \quad \text{ for any }
w\in \mathfrak{M}(\boldsymbol{\omega})_{v},
  \end{equation}
where $\stab(w)$ denotes the stabilizer of $w$ with respect to the
action of $G_{\boldsymbol{\omega}}$.

Choose the place $w_{v} \in \mathfrak{M}(\bm{\omega})$ induced by the
embedding
$\iota_{v} \colon \overline{\mathbb{K}}\hookrightarrow
\mathbb{C}_{v}$. Then~\eqref{eq:12} together with the transitivity of
the action gives that
\[
  \sum_{w\in\mathfrak{M}(\bm{\omega})_{v}}
\frac{[\mathbb{K}(\bm{\omega})_w:\mathbb{K}_v]}{[\mathbb{K}(\bm{\omega}):\mathbb{K}]} \, 
J_w(\bm{\omega})
=
\frac{1}{\# G_{\boldsymbol{\omega}}}\sum_{\sigma\in G_{\boldsymbol{\omega}}}  J_{\sigma (w_{v})}(\bm{\omega}).
\]
For each $\sigma\in G_{\boldsymbol{\omega}}$ we have that
$J_{\sigma (w_{v})}(\bm{\omega}) =I_{v}(\iota_{v}\circ
\sigma(\bm{\omega}))$, as it can be checked using the isomorphism
$X_{\sigma(w_{v})}^{\an}\simeq X_{v}^{\an}$ and the change of
variables formula.

On the other hand, the finite group $G_{\boldsymbol{\omega}}$ acts
freely and transitively on the Galois orbit of 
$\boldsymbol{\omega}$. Hence the assignment
$\sigma\mapsto \iota_{v}\circ \sigma(\boldsymbol{w})$ is a bijection
between $G_{\boldsymbol{\omega}}$ and the $v$-adic Galois orbit of
this point. Altogether we deduce that
\[
  \sum_{w\in\mathfrak{M}(\bm{\omega})_{v}}
\frac{[\mathbb{K}(\bm{\omega})_w:\mathbb{K}_v]}{[\mathbb{K}(\bm{\omega}):\mathbb{K}]} \, 
J_w(\bm{\omega})
=
\frac{1}{\# O(\boldsymbol{\omega})}\sum_{\boldsymbol{\eta}\in O(\boldsymbol{\omega})_{v}} I_{v}(\boldsymbol{\eta}).
\]
The claim follows multiplying this equality by $n_{v}$ and summing
over all $v\in \mathfrak{M}$, together with the relation between the
weights of $\mathbb{K}$ and those of $\mathbb{K}(\boldsymbol{\omega})$
in \eqref{eq: weight of places in field extension}.
\end{proof}

Combining this claim with
\eqref{eq: recursive definition of height in the general strategy} we
deduce the following recursive formula.

\begin{proposition}\label{prop: induction in general strategy}
  In this setting, for any torsion point
  $\bm{\omega}\in\mathbb{T}(\overline{\mathbb{K}})^{k} \setminus
  H(\overline{\mathbb{K}})$ we have~that
\begin{displaymath}
\h_{\overline{D}_0,\ldots,\overline{D}_{n-k}}(Z_{\mathbb{T}}(\bm{\omega}^*\bm{f}))
=
\h_{\overline{D}_0,\ldots,\overline{D}_{n-k},\overline{D}^{\Ron}_{f_{k}}}(Z_{\mathbb{T}}( \widehat{\boldsymbol{\omega}}^*{\widehat{\boldsymbol{f}}}))
+
\sum_{v\in\mathfrak{M}}\frac{n_{v}}{\# O(\boldsymbol{\omega})}\sum_{\boldsymbol{\eta}\in O(\boldsymbol{\omega})_{v}} I_{v}(\boldsymbol{\eta}).
\end{displaymath}
\end{proposition}

Let us now see how this formula could lead to the proof of the
inductive step of \cref{conj: main conjecture}.  Let
$(\bm{\omega}_\ell)_\ell$ be a quasi-strict sequence of torsion points
of~$\mathbb{T}(\overline{\mathbb{K}})^{k}$.  Setting
$\widehat{\boldsymbol{\omega}}_{\ell}= (\omega_{\ell,1},\dots,
\omega_{\ell,k-1})$ for each $\ell$, the sequence
$(\widehat{\boldsymbol{\omega}}_{\ell})_{\ell}$ of torsion points of
$\mathbb{T}(\overline{\mathbb{K}})^{k-1}$ is also quasi-strict.  Using
the inductive hypothesis, the first summand in the right-hand side of
the formula in \cref{prop: induction in general strategy} converges to
the desired limit as $\ell\to \infty$.  Thus the proof of the
conjecture  is reduced to showing that the second summand in this formula
asymptotically vanishes, namely
\begin{equation}
  \label{eq:14}
\lim_{\ell\to \infty}\sum_{v\in\mathfrak{M}}\frac{n_{v}}{\# O(\boldsymbol{\omega_{\ell}})}\sum_{\boldsymbol{\eta}\in O(\boldsymbol{\omega}_\ell)_{v}} I_{v}(\boldsymbol{\eta})=0.
\end{equation}

When approaching this problem, the following reduction steps might be
useful.

\begin{proposition}\label{prop: reduction steps}
  When proving \cref{conj: main conjecture} for a given $0\le k \le n$
  it is enough to suppose that:
\begin{enumerate}[leftmargin=*]
\item\label{item: reduction of Laurent polynomials} the Laurent
  polynomials $f_1,\ldots,f_k$ have coefficients in the ring of
  integers of~$\mathbb{K}$, their supports contain the lattice
  point~$0\in M$, and (up to possibly replacing $\mathbb{K}$ by a
  finite extension) they are absolutely irreducible, i.e. they are
  irreducible as elements of $\overline{\mathbb{K}}[M]$
\item\label{item: reduction of toric variety} the toric variety $X$ is
  smooth and projective, and its fan is compatible with the Newton polytopes of the
  Laurent polynomials
\item\label{item: reduction of geometric divisors} the toric divisors $D_0,\ldots,D_{n-k}$
  are very ample
\item\label{item: reduction of metrized divisors} the semipositive
  toric metrics of $\overline{D}_{0},\ldots,\overline{D}_{n-k}$
  are smooth at Archimedean
  places and algebraic at non-Archimedean ones
\item\label{item: reduction of sequence of torsion points} the torsion
  points in the sequence are of the form
  $\bm{\omega}_\ell=(1,\omega_{\ell,2},\ldots,\omega_{\ell,k})$ with
  $1$ the neutral element of~$\mathbb{T}(\overline{\mathbb{K}})$ and
  $((\omega_{\ell,2},\ldots,\omega_{\ell,k}))_\ell$ a strict sequence
  in~$\mathbb{T}(\overline{\mathbb{K}})^{k-1}$.
\end{enumerate}
\end{proposition}

Proving \cref{prop: reduction steps} would result weighty here, and
therefore we postpone this to \cref{sec: appendix A}.

\section{An auxiliary function}
\label{sec:an-auxil-funct-1}

With the aim of proving \cref{thm: main theorem} and in view of the
reduction steps in \cref{prop: reduction steps} we now place ourselves
in the following setting.  Let ${f,g\in \mathbb{K}[M]}$ be two
absolutely irreducible Laurent polynomials such that both of their
supports contain the lattice point $0\in M$.  We write them as
\[
f=\sum_m\alpha_m\chi^m\quad\text{and}\quad g=\sum_m\beta_m\chi^m
\]
with $\alpha_m,\beta_m\in \mathbb{K}$ that are zero for all but
finitely many~$m$.  Let also $X$ be a smooth projective toric variety
compactifying $\mathbb{T}$ with fan compatible with the Newton
polytopes of $f$ and $g$. Furthermore we denote by
$\overline{D}_0,\ldots,\overline{D}_{n-2}$ a family of semipositive
toric metrized divisors on~$X$ with very ample underlying divisors,
and with smooth metrics at the Archimedean places and algebraic
metrics at the non-Archimedean ones.

Recall from \cref{sec:geom-toric-constr} that $D_{f}$ denotes the nef
toric divisor on $X$ associated with the Newton polytope of $f$ and
$s_{f}$ the global section of $\mathcal{O}(D_{f})$ associated with this
Laurent polynomial. We set $Z=Z_{X}(s_{f})$ for the corresponding
hypersurface of $X$.

Fix $v\in \mathfrak{M}$.  To treat the $v$-adic summand in
\eqref{eq:14} in the present $2$-codimensional case it will be
convenient to consider the function
$F_v\colon\mathbb{T}(\mathbb{C}_v)\to\mathbb{R}\cup\{-\infty\}$
defined by
\begin{equation}\label{eq: auxiliary function}
F_v(t)=\int_{X_{v}^{\an}}\log|t^*g|_v\ \chern(\overline{D}_{0,v})\wedge\ldots\wedge\chern(\overline{D}_{n-2,v})\wedge\delta_{Z_v^{\an}}.
\end{equation}
This integral is computed with respect to the $v$-adic mixed
Monge--Amp\`ere measure of these metrized divisors on the
analytification of the hypersurface defined by~$f$.  We devote this
section to the study of the regularity of this auxiliary function.

\subsection{The goal of this section}
\label{sec:goal-this-section}

We first introduce a class of functions on $\mathbb{T}_v^{\an}$ with
controlled behaviour along closed algebraic subsets, similarly as
those considered by Chambert-Loir and Thuillier along Weil divisors
\cite{CLT}.

\begin{definition}\label{def: log singularities}
  Let $H\subset \mathbb{T}_{v}$ be a proper closed subset and
  $h_1,\ldots,h_s \in \mathbb{C}_{v}[M]$ a system of Laurent
  polynomials defining it. A function
  $\varphi\colon \mathbb{T}_v^\an\to\mathbb{R}\cup\{-\infty\}$ is said
  to have \emph{at most logarithmic singularities along $H^{\an}$} if
  the two following conditions are met:
\begin{enumerate}[leftmargin=*]
\item \label{item:1} the restriction of $\varphi$ to
  $\mathbb{T}_v^{\an}\setminus H^{\an}$ is a continuous function with
  real values,
\item \label{item:2} for each point of $H^{\an}$ there exist an open neighbourhood
  $U \subset \mathbb{T}_v^{\an}$ and positive real numbers
  $c_{1},c_{2}$ such that
\[
  |\varphi| \le -c_{1}\,\log\max_{j=1,\ldots,s}|h_j|_{v} +c_{2} \quad
  \text{ on } U.
\]
\end{enumerate}
\end{definition}

\begin{remark}
  \label{rem:2}
  Let $ H' \subset \mathbb{T}_{v}$ be a closed  subset containing $H$, and 
  $g_{1},\dots, g_{r}$  a system of Laurent polynomials defining
  it. By Hilbert's Nullstellensatz, there is an integer
  $\kappa_{1} \ge 1$ such that
  $g_{i}^{\kappa_{1}} \in (h_{1},\dots, h_{s})$ for every
  $i$. Therefore for any open subset $U\subset \mathbb{T}_{v}^{\an}$
  with compact closure there is a positive real number $\kappa_{2}$
  such that
  \begin{displaymath}
-    \log \max_{j=1,\dots, s}  |h_{j}|_{v} \le -\kappa_{1}\, \log \max_{i=1,\dots, r}|g_{i}|_{v} +\kappa_{2} \quad \text{ on } U.
\end{displaymath}
Hence if
$\varphi$ has at most logarithmic singularities along~$H^{\an}$ then it
also has at most logarithmic singularities along~$(H')^{\an}$.

Applying this observation to the case when $H'=H$, it implies moreover that
\cref{def: log singularities} does not depend on the choice of the
defining system of Laurent polynomials.
\end{remark}

To study the singularities of the auxiliary function we introduce the
proper closed subset of the torus defined as
\begin{equation}
  \label{eq: set of bad
    intersection}
  \Upsilon=     V_{\mathbb{T}}(\{\alpha_{m}\beta_{m'+m_{0}}\chi^{m'+m_{0}}-\alpha_{m'}\beta_{m+m_{0}}\chi^{m+m_{0}}\}_{m,m'\in M}) 
\end{equation}
if there is $ m_{0}\in M $ such that $\supp(g)=\supp(f)+m_{0}$, and as
$\Upsilon= \emptyset$ otherwise.

\begin{lemma}
  \label{lem:7}
  For $t\in \mathbb{T}(\mathbb{C}_{v})$ we have that
  $F_{v}(t)=-\infty$  if and only if $t\in \Upsilon(\mathbb{C}_{v})$.
  \end{lemma}

\begin{proof}
  First suppose that there is $ m_{0}\in M $ such that
  $\supp(g)=\supp(f)+m_{0}$. Then for
  $t\in \mathbb{T}(\mathbb{C}_{v})$ we have that
  $t\in \Upsilon(\mathbb{C}_{v})$ if and only if there is
  $\lambda \in \mathbb{C}_{v}^{\times}$ such that
  $\beta_{m+m_{0}} \chi^{m+m_{0}}(t)=\lambda \alpha_{m}$ for all $m$,
  or equivalently that
  \begin{displaymath}
      t^{*}g= \lambda \chi^{m_{0}} f.   
  \end{displaymath}
  Otherwise $f$ and $t^{*}g$ are coprime for all
  $t \in \mathbb{T}(\mathbb{C}_{v})$ because both Laurent polynomials
  are absolutely irreducible and their supports do not coincide modulo
  a translation.

  We conclude that $t\in \Upsilon(\mathbb{C}_{v})$ if and only if $f$
  and $t^{*}g$ coincide up to a monomial factor. When this condition
  holds we have that $|t^{*}g \, (x)|_{v}=0$ for all
  $x\in Z_{v}^{\an}$ and so $F_{v}(t)=-\infty$. Otherwise $t^{*}g$ is
  a nonzero rational function on the hypersurface $Z$ and therefore
  $F_{v}(t)\in \mathbb{R}$ because the $v$-adic mixed Monge--Amp\`ere
  measure integrates functions with at most logarithmic singularities.
\end{proof}

The following is our main result here. Its proof is rather long and
technical, and will occupy us for the rest of the section.

\begin{theorem}\label{thm: auxiliary function has log singularities}
The function $F_{v}$ extends uniquely to a
  function $ \mathbb{T}_v^{\an}\rightarrow \mathbb{R}\cup\{-\infty\} $
  with at most logarithmic singularities along~$\Upsilon_v^{\an}$ and
  taking the value $-\infty$ on this analytic subvariety.
\end{theorem}

\subsection{Archimedean continuity}
\label{sec:cont-auxil-funct}

We start by showing that when $v$ is Archimedean, the auxiliary
function is continuous outside its singularity locus. The proof is an
application of Stoll's theorem on the continuity of the fiber
integral~\cite{Stoll}.

\begin{proposition}
  \label{prop:1}
  Let $v$ be an Archimedean place of $\mathbb{K}$. Then the
  restriction of $F_{v}$ to
  $\mathbb{T}_{v}^{\an}\setminus \Upsilon_{v}^{\an}$ is a continuous
  function with real values.
\end{proposition}
\begin{proof}
  After fixing an isometry we identify $\mathbb{C}_{v}$ with the
  field of complex numbers and consequently we drop the index $v$ from
  the notation throughout. In particular, the $v$-adic
  analytifications of $\mathbb{T}_{v}$ and $\Upsilon_{v}$ are respectively identified with 
  $\mathbb{T}(\mathbb{C})$ and $\Upsilon(\mathbb{C})$.
  Thus we can write the auxiliary function as
  \begin{equation}\label{eq: Archimedean auxiliary function}
    F(t)=\int_{Z(\mathbb{C})}\log|g(t\cdot x)|\ d \lambda \quad \text{ for } t\in \mathbb{T}(\mathbb{C})
  \end{equation}
  where $\lambda$ denotes the $(n-1,n-1)$-form
  $\bigwedge_{i=0}^{n-2}c_{1}(\overline{D}_{i}|_{Z})$.  Up to
  desingularizing we can assume that $Z$ is a smooth complete variety
  over $\mathbb{C}$.

  Consider the finite open covering of $X$ by the affine toric
  varieties $X_{\sigma}$, $\sigma\in \Sigma$, as in~\eqref{eq:10}.
  Choose a family of compact semialgebraic subsets
  ${K_{\sigma}\subset (X_{\sigma}\cap Z)(\mathbb{C})}$,
  $\sigma\in \Sigma$, such that
  \begin{displaymath}
 \bigcup_{\sigma\in \Sigma}K_{\sigma}=Z(\mathbb{C}).
  \end{displaymath}
  Such a covering can be obtained, for instance, as the restriction
  to $Z(\mathbb{C})$ of the Batyrev-Tschinkel decomposition of
  $X(\mathbb{C})$ \cite[Section 3.2]{Maillot}.  By the
  inclusion-exclusion formula we have 
    \begin{equation}
      \label{eq:17}
    F(t)=\sum_{P\subset \Sigma} (-1)^{\#P-1}\int_{ K_{P}} \log|g(t\cdot x)| \, d\lambda
  \end{equation}
  with $K_{P}=\bigcap_{\sigma\in P}K_{\sigma}$ for $P\subset \Sigma$.

  Fix $P$. Choose then any $\sigma\in P$ and write the restriction of
  the rational function $g$ to the corresponding affine toric variety
  as a quotient of nonzero regular functions
  \begin{displaymath}
    g|_{X_{\sigma}}=\frac{h_{1}}{h_{2}}
  \end{displaymath}
  that are coprime as elements of $\mathbb{C}[M]$. By \cite[Theorem
  4.9]{Stoll}, for each $j$ the function
  \begin{equation}
    \label{eq:15}
    t\longmapsto \int_{K_{P}} \log|h_{j}(t\cdot x)| \, d\lambda
  \end{equation}
  is continuous at any point
  $t_{0}\in \mathbb{T}(\mathbb{C})\setminus \Upsilon(\mathbb{C})$. We
  now check this claim by placing ourselves in the notation and
  terminology of \emph{loc. cit.}. To this end consider first the
  smooth complex manifold
  \begin{displaymath}
 M= (X_{\sigma}\cap Z)(\mathbb{C})\times \mathbb{T}(\mathbb{C}) 
  \end{displaymath}
  equipped with the differential form of bidegree $(n-1,n-1)$ obtained
  as the pullback of $\lambda$ with respect to the projection onto the
  first factor $M\to (X_{\sigma}\cap Z)(\mathbb{C})$.  Consider also
  the $(n-1)$-fibering obtained as the projection onto the second
  factor
  \begin{math}
  M \rightarrow  \mathbb{T}(\mathbb{C}).
  \end{math}
Furthermore  choose an open subset $B \subset \mathbb{T}(\mathbb{C})$
  with compact closure containing the point~$t_{0}$ and set
  \begin{displaymath}
    G=\interior(K_{P}) \times B \subset M,
  \end{displaymath}
  where $\interior(K_{P})$ denotes the interior of this compact
  semialgebraic subset.

  Notice that the boundary of $G$ restricted to the fiber of $t_{0}$
  is of measure zero with respect to the chosen
  $(n-1,n-1)$-form. Moreover the holomorphic function
  $M\to \mathbb{C}$ defined by $(x,t)\mapsto h_{j}(t\cdot x)$ is not
  identically zero on the fiber over $t_{0}$ because
  $t_{0}\notin \Upsilon(\mathbb{C})$.
  Then the hypotheses of Stoll's theorem are satisfied, and the
  continuity at $t_{0}$ of the integral in \eqref{eq:15} follows from
  this result together with the computation of the multiplicities of the
  fibering given by \cite[Lemma~5.2]{Stoll:mhm}.

  Finally the continuity on
  $ \mathbb{T}(\mathbb{C}) \setminus \Upsilon(\mathbb{C})$ in
  \eqref{eq:15} implies that of the $P$-term in the
  inclusion-exclusion formula in \eqref{eq:17}, and in turn that of
  $F$.
\end{proof}

\subsection{Non-Archimedean continuity}
\label{sec:non-arch-cont}

When $v$ is non-Archimedean, the function $F_v$ is only defined at the
rigid points of~$\mathbb{T}_{v}^{\an}$.  Here we show that it can be
uniquely extended to the whole of the $v$-adic analytic torus with
singularities along~$\Upsilon_{v}^{\an}$. Its proof relies on a number
of technical results and constructions from non-Archimedean formal and
analytic geometry.

\begin{proposition}
\label{prop:2}
Let $v$ be a non-Archimedean place of $\mathbb{K}$. Then $F_{v}$
extends uniquely to a function
$\mathbb{T}_{v}^{\an} \to \mathbb{R}\cup\{-\infty\}$
which is continuous with real values on
$\mathbb{T}_{v}^{\an}\setminus \Upsilon_{v}^{\an}$ and takes the value
$-\infty$ on $\Upsilon_{v}^{\an}$.
\end{proposition}

\begin{proof}
If the claimed extension exists then it is unique,
  by the density of $\mathbb{T}(\mathbb{C}_{v})$ in
  $\mathbb{T}_{v}^{\an}\setminus \Upsilon_{v}^{\an}$. We thus focus on the
  proof of its existence.

  We start by providing a more explicit expression for the function
  $F_v$ from which the definition of the extension will be more
  natural.  We denote by $\iota\colon Z\hookrightarrow X$ the
  inclusion of the hypersurface $Z$ into the toric variety $X$. Note
  that we can consider the pullback with respect to $\iota$ of any
  toric divisor on $X$ because $Z$ is not contained in the boundary
  $X\setminus X_{0}$.
  
  Since the metrics in the $v$-adic metrized divisors
  $\overline{D}_{0,v},\ldots,\overline{D}_{n-2,v}$ are algebraic, by passing
  to the formal completion along the special fiber they are also 
  formal  in the sense of \cite[Section~7]{G0}.
  Thus, the metrics of the pullbacks
  $\iota_v^*\overline{D}_{0,v},\ldots,\iota_v^*\overline{D}_{n-2,v}$
  are formal. It then follows from \cite[Proposition~3.11]{G2} and
  \cite[Proposition~1.11]{G0} that there exist a distinguished formal
  analytic variety $\mathfrak{Z}$ over~$\mathbb{C}_v$ with generic
  fiber $Z_v^{\an}$ and reduced special
  fiber~$\widetilde{\mathfrak{Z}}$, and for each~$i=0,\ldots,n-2$ a
  pair $(\mathfrak{L}_i, e_{i})$ consisting of a formal analytic line
  bundle on $\mathfrak{Z}$ and a positive integer inducing the metric
  of~$\iota_{v}^{*}\overline{D}_{i,v}$. In this setting the $v$-adic
  Monge--Amp\`ere measure in the definition of $F_v$ is the discrete
  measure on $X_{v}^{\an}$ given by
  \begin{multline}
 \label{eq: form of non-Archimedean Monge-Ampère}
\chern(\overline{D}_{0,v})\wedge\ldots\wedge\chern(\overline{D}_{n-2,v})\wedge\delta_{Z_v^{\an}}
=
\iota_{v,*}^{\an}(\chern(\iota_v^*\overline{D}_{0,v})\wedge\ldots\wedge\chern(\iota_v^*\overline{D}_{n-2,v}))
\\=
\sum_{V\in \widetilde{\mathfrak{Z}}^{(0)}}\frac{\deg_{\widetilde{\mathfrak{L}}_0,\ldots,\widetilde{\mathfrak{L}}_{n-2}}(V)}{e_{0}\cdots e_{n-2}}\, \delta_{\iota_v^{\an}(\xi_V)},
\end{multline}
where $\widetilde{\mathfrak{L}}_{i}$ is the line bundle on
$\widetilde{\mathfrak{Z}}$ induced by $\mathfrak{L}_{i}$, the index
set $\widetilde{\mathfrak{Z}}^{(0)}$ is the collection of irreducible
components of~$\widetilde{\mathfrak{Z}}$, and for each
$V\in \widetilde{\mathfrak{Z}}^{(0)}$ we denote by $\xi_V$ the unique
point of $Z_v^{\an}$ which is sent to the generic point $\eta_{V}$ of
$V$ by the reduction map.  Hence the auxiliary function takes
the form
\begin{equation}
  \label{eq:18}
    F_v(t) =
  \sum_{V\in \widetilde{\mathfrak{Z}}^{(0)}}\frac{\deg_{\widetilde{\mathfrak{L}}_0,\ldots,\widetilde{\mathfrak{L}}_{n-2}}(V)}{e_{0}\cdots e_{n-2}} \, 
  \log|t^*g \, (\xi_V)|_{v} \quad \text{ for } t\in
  \mathbb{T}(\mathbb{C}_{v}),
\end{equation}
where the Laurent polynomial $t^*g$ is seen as an element
of~$\mathbb{C}_v[M]/(f)$.

For what follows we need a more explicit description of the analytic
points in the support of the measure.  Since $\mathfrak{Z}$ is a
distinguished formal analytic variety with reduced special fiber, by
definition it is locally isomorphic to the formal spectrum of a
distinguished $\mathbb{C}_v$-affinoid algebra~$\mathscr{A}$ whose
reduction~$\wA$ is a reduced
$\widetilde{\mathbb{C}}_{v}$-algebra. Therefore for
each~$V\in \widetilde{\mathfrak{Z}}^{(0)}$ there is a distinguished
$\mathbb{C}_{v}$-affinoid algebra $\mathscr{A}_{V}$ such that
$\eta_{V}\in \Spec(\wAV)$. Following the proof of
\cite[Proposition~2.4.4]{Ber}, up to localizing we can assume that
$\wAV$ is a domain, in which case $\xi_V$ can be described as the
point of the Berkovich spectrum
$\mathcal{M}(\mathscr{A}_{V}) \subset Z_{v}^{\an}$ corresponding to
the sup-seminorm on~$\mathscr{A}_{V}$.  Precisely, this point is the
multiplicative seminorm on $\mathscr{A}_{V}$ defined as
\begin{equation}
  \label{eq:19}
  |a(\xi_V)|_{v}=\sup_{x\in\mathcal{M}(\mathscr{A}_V)}|a(x)|_v\quad\text{for }a\in\mathscr{A}_V.
\end{equation}

We want to apply this norm to the Laurent polynomial $t^{*}g $ for
$t\in \mathbb{T}(\mathbb{C}_{v})$. By \cite[6.4.3/Theorem~1 and
6.2.1/Proposition~4]{BGR}, for each $V$ the corresponding
$\mathbb{C}_{v}$-affinoid algebra $\mathscr{A}_V$ is reduced and
therefore the seminorm in \eqref{eq:19} is actually a norm. This
latter fact implies that $\xi_{V}$ avoids every proper analytic subset
of $\mathcal{M}(\mathscr{A}_{V})$, and in particular the
analytification of the boundary $Z \setminus X_{0}$. Hence up to
localizing again we can assume without loss of generality that
$\mathbb{C}_{v}[M]/(f)\subset \mathscr{A}_{V}$, and so
\begin{equation}
  \label{eq:20}
  |t^{*}g \, (\xi_{V})|_{v}=\sup_{x\in \mathcal{M}(\mathscr{A}_{V})} |t^{*}g \, (x)|_{v}.
\end{equation}

Let now $t\in\mathbb{T}_v^{\an}$ be an arbitrary analytic point
and~$\mathscr{H}(t)$ its complete residue field, which is a complete
valued field extension of~$\mathbb{C}_v$.  For each
$V\in \widetilde{\mathfrak{Z}}^{(0)}$ we plan to define
$|t^{*}g \, (\xi_{V})|_{v}$ extending the expression in \eqref{eq:20}.

To this end, recall that $\mathscr{A}_{V}$ is a distinguished
$\mathbb{C}_{v}$-affinoid algebra with integral reduction
$\wAV$. Since $\widetilde{\mathbb{C}}_{v}$ is algebraically closed,
the tensor product
\begin{math} 
\wAV \otimes_{\widetilde{\mathbb{C}}_{v}}\wHt
\end{math}
is an integral $\wHt$-algebra
\cite[V.17.5/Corollaire~2]{Bou81}.
Hence thanks to \cite[Satz~6.4]{Bosch_69} the completed tensor product
$\mathscr{A}_{V}\widehat{\otimes}_{\mathbb{C}_v}\mathscr{H}(t)$ is a
distinguished~$\mathscr{H}(t)$-algebra.
Set  then
\begin{displaymath}
  \xi_{V,t} \in
  \mathcal{M}(\mathscr{A}_{V}\widehat{\otimes}_{\mathbb{C}_{v}}\mathscr{H}(t))  
\end{displaymath}
for the sup-seminorm on this completed tensor product. As before, the
algebra
$\mathscr{A}_{V}\widehat{\otimes}_{\mathbb{C}_{v}} \mathscr{H}(t)$ is
reduced and so this multiplicative seminorm is actually a norm.
Moreover it coincides with the completed tensor product norm, as shown
in the proof of \cite[Proposition~5.2.5]{Ber}.

We have that
\begin{math}
  t^*g\in\mathscr{H}(t)[M]/(f) \simeq \mathbb{C}_{v}[M]/(f) \otimes_{\mathbb{C}_{v}} \mathscr{H}(t)
  \subset \mathscr{A}_V\widehat{\otimes}_{\mathbb{C}_v}\mathscr{H}(t),
\end{math}
and so we can consider the norm for this twist. By the previous
discussion it can be expressed~as
\begin{equation}
\label{eq:21}
|t^{*}g \, (\xi_{V,t})|_{v}=\sup_{x} |t^{*}g \, (x)|_{v}
=|t^{*}g |_{\mathscr{A}_{V}\widehat{\otimes}_{\mathbb{C}_{v}}\mathscr{H}(t)},
\end{equation}
where $x$ ranges over the Berkovich spectrum 
$\mathcal{M}(\mathscr{A}_{V}\widehat{\otimes}_{\mathbb{C}_{v}}\mathscr{H}(t))$.

Finally, the  sought extension for $F_{v}$ is defined as
\begin{equation}\label{eq: extension of non-Archimedean auxiliary map}
  F_v(t)
  =
  \sum_{V\in \widetilde{\mathfrak{Z}}^{(0)}}\frac{\deg_{\widetilde{\mathfrak{L}}_0,\ldots,\widetilde{\mathfrak{L}}_{n-2}}(V)}{e_{0}\cdots e_{n-2}} \, \log|t^*g \, (\xi_{V,t})|.
\end{equation}
When~$t\in\mathbb{T}(\mathbb{C}_v)$ we have that
$\mathscr{H}(t) =\mathbb{C}_v$ and so the above expression agrees with
that for the auxiliary function in \eqref{eq:18}.

Moreover the defined extension takes the value $-\infty$ if and only
if there is at least one $V$ for which
$|t^*g\,(\xi_{V,t})|_{v}=0$. Since $\xi_{V,t}$ is a norm, this is
equivalent to the fact that $t^*g=0 \in \mathscr{H}(t)[M]/(f)$. As both $f$
and $g$ are absolutely irreducible, this vanishing occurs if and only
if $t^*g$ and $f$ coincide up to a nonzero scalar in  $\mathscr{H}(t)$,
which translates into the condition that
\[
\alpha_m\beta_{m^\prime}\chi^{m^\prime}(t)-\alpha_{m^\prime}\beta_m\chi^m(t)=0\in \mathscr{H}(t)\quad\text{ for all }m,m^\prime\in M.
\]
This is equivalent to the fact that the ideal of $\Upsilon_{v}$ is contained
in the kernel of the multiplicative seminorm corresponding to the analytic point $t$,
and in turn to the fact that $t\in \Upsilon_{v}^{\an}$.

To conclude, we are left to show that the function $F_v$ is continuous
on~$\mathbb{T}^{\an}\setminus \Upsilon_v^{\an}$.  Let
$t_{0}\in \mathbb{T}_{v}^{\an}$ and choose a $\mathbb{C}_{v}$-affinoid
algebra $\mathscr{B}$ such that the Berkovich spectrum
$\mathcal{M}(\mathscr{B})$ is a neighbourhood of $t_{0}$. Then
\cite[Lemma~5.2.6]{Ber} applied to the element
$\sum_m\beta_m\chi^m\otimes\chi^m \in
\mathscr{A}_{V}\widehat{\otimes}_{\mathbb{C}_{v}} \mathscr{B}$ gives
that the function on $\mathcal{M}(\mathscr{B})$ defined as
\[
t
\longmapsto |t^{*}g |_{\mathscr{A}_{V}\widehat{\otimes}_{\mathbb{C}_{v}}\mathscr{H}(t)}
\]
is continuous. Because of \eqref{eq:21}
and the fact that $|t^*g\,(\xi_{V,t})|_{v}>0$ for
$t\in \mathbb{T}_{v}^{\an}\setminus \Upsilon_{v}^{\an}$ for all $V$ this
implies that $F_v$ is continuous with real values on this set.
\end{proof}

\begin{remark}
The extension of the
  auxiliary function can be also written as
\[
F_v(t)=\int_{X_{v}^{\an}}\log|g(t\ast x)|\  \chern(\overline{D}_{0,v})\wedge\ldots\wedge\chern(\overline{D}_{n-2,v})\wedge\delta_{Z_v^{\an}},
\]
where $t\ast x\in X^{\an}$ is the \emph{peaked product} between
$t\in\mathbb{T}^{\an}$ and $x\in X^{\an}$ as defined in
\cite[\S5.2]{Ber}, see also \cite[Definition 4.2.2]{BPS}. This follows
from the previous proof, precisely from \eqref{eq:21} and \eqref{eq:
  extension of non-Archimedean auxiliary map}.  This is the
non-Archimedean analogue of the expression for $F_v$ in \eqref{eq:
  Archimedean auxiliary function}.
\end{remark}

\subsection{A logarithmic lower bound}
\label{sec: logarithmic lower bound}

Here we prove the following lower bound for the $v$-adic auxiliary
function.

\begin{proposition}\label{prop: lower log bound for Fv}
Let $v\in \mathfrak{M}$. There is a system of Laurent
  polynomials $h_{1},\dots, h_{s}\in \mathbb{C}_{v}[M]$ defining
  $\Upsilon_{v}$ and a positive real number $c$ such that
\[
  F_v\geq \log\max_{j=1,\dots, s} |h_{j}|_{v} -c
   \quad \text{ on }  \mathbb{T}(\mathbb{C}_{v}).
\]
\end{proposition}

The main idea of the proof is to compare the value of $F_v$ at $t$
with the valuation of $t^*g$ at the $0$-cycle in $Z_v$ obtained by
cutting the latter with suitably generic global sections.  The Poisson
formula for the sparse resultant from \cite{DS:pfsr} will then allow
to see that such a valuation depends polynomially on $t$ in a way that
is related with the vanishing of the Laurent polynomials defining
$\Upsilon_v$.
  
  To turn this strategy into practice we need a series of side results. The first gives a
simple criterion for the proper intersection of a family of global
sections in the ample case.

\begin{lemma}
  \label{lem:6}
  Let $Y$ be a complete variety over $\mathbb{C}_{v}$ of dimension
  $e$, and for $i=1,\dots, k$ with $1\le k\le e+1$ let $s_{i}$ be a
  global section of an ample line bundle $L_{i}$ on $Y$. Then
\begin{equation}
  \label{eq:25}
 \dim(V_{Y}(s_{1},\dots, s_{k}))=e-k
\end{equation}
if and only if $s_{1},\dots, s_{k}$ meet properly.
\end{lemma}

\begin{proof}
  If $s_{1},\dots, s_{k}$ meet properly then the
  condition \eqref{eq:25} clearly holds.

Conversely assume that this condition is satisfied. If
$s_{1},\dots,s_{k}$ do not meet properly then there is
$I\subset \{1,\dots, k\}$ such that
  \begin{displaymath}
  \dim(V_{Y}(\{s_{i}\}_{i\in I})>e-\#I.  
  \end{displaymath}
  Suppose that $I$ is the maximal index subset satisfying this
  inequality and take an irreducible component $C$ of this zero set
  with
  \begin{math}
\dim(C)>e-\#I.  
  \end{math}
  Take also an index ${i_{0}\in \{1,\dots, k\} \setminus I}$, which is
  possible because \eqref{eq:25} ensures that $I\ne \{1,\dots, k\}$.

  This irreducible component has positive dimension because
  $\#I \le k-1 \le e$, and cannot be contained in the hypersurface
  $V_{Y}(s_{i_{0}})$ because this would contradict the maximality
  of~$I$. Hence this hypersurface cuts~$C$, and so
  $C\cdot \div(s_{i_{0}})$ is a cycle of positive degree with respect
  to the ample line bundle~$L_{i_{0}}$.  In particular
  $C\cap V_{Y}(s_{i_{0}}) \ne \emptyset$, and by Krull's
  Hauptidealsatz this intersection has dimension $\dim(C)-1$. Hence
  \begin{displaymath}
    \dim(V_{Y}(\{s_{i}\}_{i\in I\cup\{i_{0}\}})) \ge \dim(V_{Y}(\{s_{i}\}_{i\in I}))-1
>e-\#(I \cup\{i_{0}\}),
  \end{displaymath}
  contradicting again the maximality of $I$ and thus proving the statement. 
\end{proof}

We also need the next formula from \cite[Lemma 8.17(ii)]{BE21}.  It
can be seen as a higher dimensional metric version of the Weil
reciprocity law as explained in \cite[Remark~9.7]{GualdiSombra}, and a
self-contained proof for complex smooth projective curves can be found
in \cite[Proposition~9.6]{GualdiSombra}.

\begin{proposition}\label{metric Weil reciprocity law}
  Let $Y$ be a complete variety over $\mathbb{C}_{v}$ of dimension $e$
  and
  $\overline{L}_{0}=(L_0,\|\cdot\|_0),
  \ldots,\overline{L}_{e}=(L_e,\|\cdot\|_e)$ semipositive metrized
  line bundles on $Y$.  Let
  $s_{e-1} \in \Gamma(Y,L_{e-1})$
  and $s_{e}\in \Gamma(Y,L_{e})$ be global sections that meet properly. Then
  \begin{multline*}
\hspace{-3mm}    \int_{Y^{\an}}\hspace{-2mm}\log\|s_{e-1}\|_{e-1}\, \big(\chern(\overline{L}_0)\wedge\ldots\wedge\chern(\overline{L}_{e-2})\wedge
    \chern(\overline{L}_{e})\,-\,\chern(\overline{L}_0)\wedge\ldots\wedge\chern(\overline{L}_{e-2})\wedge\delta_{Z_{Y}(s_{e})^{\an}}\big) \\
    =\int_{Y^{\an}}\hspace{-2mm}\log\|s_{e}\|_{e}\, \big(\chern(\overline{L}_0)\wedge\ldots\wedge\chern(\overline{L}_{e-2})\wedge
    \chern(\overline{L}_{e-1})\,-\,\chern(\overline{L}_0)\wedge\ldots\wedge\chern(\overline{L}_{e-2})\wedge\delta_{Z_{Y}(s_{e-1})^{\an}}\big). 
\end{multline*}
\end{proposition}

We next give a lower bound for the auxiliary function at a point in
terms of the evaluation of the corresponding twist of $g$ at a certain
$0$-cycle of $X_{v}$. The evaluation of a Laurent polynomial
$q\in \mathbb{C}_{v}[M]$ at a $0$-cycle $W=\sum_{i\in I}k_{i}x_{i}$ of
$X_{v}$ is the scalar defined as
\begin{displaymath}
q  (W)=\prod_{i\in I} q (x_{i})^{k_{i}}
  \in \mathbb{C}_{v}.
\end{displaymath}

Recall from \cref{sec:geom-toric-constr} that $D_{f}$ and $D_{g}$ are
the nef toric divisors associated with the Newton polytopes of $f$ and
$g$ respectively, and $s_{f} \in \Gamma(X,\mathcal{O}(D_{f})) $ and
$s_{g} \in \Gamma(X,\mathcal{O}(D_{g})) $ the global sections induced
by these Laurent polynomials.  Moreover $s_{D_{g}}$ denotes the
canonical rational section of $\mathcal{O}(D_{g})$, which in our
current setting is also a global section because of the assumption
that $0\in \supp(g)$.

\begin{lemma}
  \label{lem:2}
  Let $\sigma_{i}\in \Gamma(X_{v},\mathcal{O}(D_{i}))$, $i=0,\dots, n-2$,
  such that the global sections $\sigma_{0},\dots, \sigma_{n-2},s_{f}$
  meet properly and their common zero set is contained in the
  principal open subset $X_{v,0}$. Then there is $c>0$ such that
  \begin{displaymath}
    F_{v}(t)\ge \log |t^{*}g\,  (Z_{X_{v}}(\sigma_{0},\dots, \sigma_{n-2},s_{f}))|_{v} -c \quad \text{ for } t\in \mathbb{T} (\mathbb{C}_{v}) .
  \end{displaymath}
\end{lemma}

\begin{proof}
  Set $W=Z_{X_{v}}(\sigma_{0},\dots, \sigma_{n-2},s_{f})$ for short. This
  is an effective $0$-cycle of $X_{v}$ whose support $|W|$ is contained in
  $Z_{v} \cap X_{v,0}$. Its degree coincides with
  $\deg_{D_{0},\dots, D_{n-2}}(Z)$, and so it is positive because
  these divisors are ample. In particular $W\ne 0$.

  Let $t\in \mathbb{T} (\mathbb{C}_{v})$.  We suppose without loss of
  generality that $t^{*}g \,  (W) \ne 0$ because otherwise the right-hand
  side of the sought inequality  equals $-\infty$ and so it  is
  trivially verified.  Then the twist $t^{*}s_{g}=t^{*}g \, s_{D_{g}}$
  has no zeros on $|W|$  because of this assumption and
  the fact that the zero set of $s_{D_{g}}$ is contained in the
  boundary $X_{v}\setminus X_{v,0}$, that is
  \begin{equation}
    \label{eq:30}
|W|\cap V_{X_{v}}(t^{*} s_{g})  = V_{X_{v}}(\sigma_{0},\dots, \sigma_{n-2},s_{f},t^{*} s_{g})=\emptyset.
\end{equation}

This implies that $t\notin \Upsilon(\mathbb{C}_{v})$ because the
support $|W|$ is contained in $ Z_{v}$ and nonempty. Hence the global
sections $s_{f}$ and $t^{*}s_{g}$ meet properly. We can then 
apply \cref{lem:6} to each irreducible component $C$ of  $V_{X_{v}}(s_{f},t^{*}s_{g})$ 
and  the ample line bundles
$\mathcal{O}(D_{0})|_{C}, \dots, \mathcal{O}(D_{n-2})|_{C}$ equipped
with global sections $\sigma_{0}|_{C},\dots, \sigma_{n-2}|_{C}$, to
deduce that $\sigma_{0},\dots, \sigma_{n-2}$ meet properly on this
zero set.  In turn, this implies that for $i=0,\dots, n-2$ the global
sections $\sigma_{i}$ and $t^{*}s_{g}$ meet properly on
$V_{X_{v}}(\sigma_{i+1},\dots, \sigma_{n-2},s_{f})$.

Similarly, the fact that 
$V_{X_{v}}(\sigma_{0},\dots, \sigma_{n-2},s_{f}, s_{D_{g}})= \emptyset$
implies that for each $i$
the global sections $\sigma_{i}$ and $s_{D_{g}}$ meet properly on
$V_{X_{v}}(\sigma_{i+1},\dots, \sigma_{n-2},s_{f})$.

Now take any semipositive metric $\|\cdot\|$ on the analytic line
  bundle $\mathcal{O}(D_{g})_{v}^{\an}$ and consider the metrized
  divisor $ \overline{D}_{g,v}=(D_{g,v},\|\cdot\|)$ on $X_{v}$.  Since 
  $t^{*}s_{g}= t^{*}g \, s_{D_{g}} $ we have that
  $\|t^{*}s_{g}(x)\|=|t^{*} g(x)|_{v}\, \|s_{D_{g}}(x)\|$ for
all  $x\in X_{v,0}^{\an}$ and so
\begin{multline}
    \label{eq:23}
        F_{v}(t)= \int_{X_{v}^{\an}} \log\|t^{*}s_{g}\|
    \ \chern(\overline{D}_{0,v})\wedge\ldots\wedge\chern(\overline{D}_{n-2,v})
    \wedge\delta_{Z_v^{\an}} \\ -\int_{X_{v}^{\an}} \log\|s_{D_{g}}\| 
    \ \chern(\overline{D}_{0,v})\wedge\ldots\wedge\chern(\overline{D}_{n-2,v})
    \wedge\delta_{Z_v^{\an}}.
  \end{multline}
  By the previous discussion we can apply twice \cref{metric Weil
    reciprocity law} to the restriction to the hypersurface
  $Z_{v}=Z_{X_{v}}(s_{f})$ of the metrized line bundles of
  $\overline{D}_{0,v},\dots, \overline{D}_{n-2,v},
  \overline{D}_{g,v}$, first to the global sections $ \sigma_{n-2}$
  and $ t^{*}s_{g}$, and secondly to $\sigma_{n-2} $ and $ s_{D_{g}}$.
  Subtracting the resulting formulae and taking into account
  \eqref{eq:23} we obtain
  \begin{align*}
    F_{v}(t)= &\int_{X_{v}^{\an}} \log|t^{*}g|_{v} 
    \ \chern(\overline{D}_{0,v})\wedge\ldots\wedge\chern(\overline{D}_{n-3,v})
    \wedge\delta_{Z_{X_{v}}(\sigma_{n-2},s_f)^{\an}} \\
&    +\int_{X_{v}^{\an}} \log\|\sigma_{n-2}\|_{n-2}     \  \chern(\overline{D}_{0,v})\wedge\ldots\wedge\chern(\overline{D}_{n-3,v})
    \wedge\delta_{Z_{X_{v}}(s_f,s_{D_{g}})^{\an}}
\\ &-\int_{X_{v}^{\an}} \log\|\sigma_{n-2}\|_{n-2}     \ \chern(\overline{D}_{0,v})\wedge\ldots\wedge\chern(\overline{D}_{n-3,v})
\wedge\delta_{Z_{X_{v}}(s_f,t^{*}s_{g})^{\an}}.
  \end{align*}
  Iterating this procedure we obtain
  \begin{equation}
    \label{eq:22}
F_{v}(t)=\int_{X_{v}^{\an}} \log|t^{*}g|_{v} 
    \ \delta_{W^{\an}} 
    + \sum_{i=0}^{n-2} (S_{i} +T_{i})
    = \log|t^{*}g\, (W) |_{v} 
    + \sum_{i=0}^{n-2} (S_{i} +T_{i})
  \end{equation}
  with
    \begin{align*}
S_{i}& =    \int_{X_{v}^{\an}} \log\|\sigma_{i}\|_{i}     \ \chern(\overline{D}_{0,v})\wedge\ldots\wedge\chern(\overline{D}_{i-1,v})
    \wedge\delta_{Z_{X_{v}}(\sigma_{i+1}, \dots, \sigma_{n-2},s_f, s_{D_{g}})^{\an}},\\      
T_{i}& =    -\int_{X_{v}^{\an}} \log\|\sigma_{i}\|_{i}     \ \chern(\overline{D}_{0,v})\wedge\ldots\wedge\chern(\overline{D}_{i-1,v})
    \wedge\delta_{Z_{X_{v}}(\sigma_{i+1}, \dots, \sigma_{n-2},s_f, t^{*}s_{g})^{\an}}.
    \end{align*}
    The first of these quantities is independent of $t$ whereas the
    second can be bounded below in terms of the sup-norm of the global
    section $\sigma_{i}$ as
    \begin{align*}
      T_{i} &\ge -\log\|\sigma_{i}\|_{i,\sup} \,
\int_{X_{v}^{\an}}    \chern(\overline{D}_{0,v})\wedge\ldots\wedge\chern(\overline{D}_{i-1,v})
\wedge\delta_{Z_{X_{v}}(\sigma_{i+1}, \dots, \sigma_{n-2},s_f, t^{*}s_{g})^{\an}}
\\ &=-\log\|\sigma_{i}\|_{i,\sup} \, (D_{0}
      \cdots D_{i-1}\cdot D_{i+1} \cdots D_{n-2} \cdot D_{f} \cdot
      D_{g}\cdot X).
    \end{align*}
    The statement follows then from \eqref{eq:22} together with this
    lower bound.
  \end{proof}

  We also need the following upper bound for the coefficients of a
  Laurent polynomial.
  
  \begin{lemma}
    \label{lem:4}
    For $r\ge 1$ let $\mathcal{A}\subset \mathbb{Z}^{r}$ be a finite subset 
and 
    $U\subset \Gm^{r}$   open and nonempty. Then there is a
    finite subset $F\subset U(\mathbb{C}_{v})$ such that for every
    Laurent polynomial $p\in \mathbb{C}_{v}[\mathbb{Z}^{r}]$ of the
    form $p=\sum_{a\in \mathcal{A}} p_{a}\chi^{a}$ we have
    \begin{displaymath}
      \max_{a\in \mathcal{A}}|p_{a}|_{v}\le \max_{u\in F}|p(u)|_{v}.
    \end{displaymath}
  \end{lemma}

  \begin{proof}
Let $k$ be a positive integer. Then for each
    $m\in \mathbb{Z}^{r}$ we have
    \begin{displaymath}
\frac{1}{k^{r}}      \sum_{\zeta\in \mu_{k}^{r}}\chi^{m}( \zeta)=
      \begin{cases}
1        & \text{ if } m\in k\mathbb{Z}^{r}, \\
        0 & \text{ otherwise,}
      \end{cases}
    \end{displaymath}
    where $\mu_{k}$ denotes the set of $k$-roots of unity of
    $ \mathbb{C}_{v}$.
 
    Take $k$ large enough so that the reduction map
    $\mathbb{Z}^{r}\to \mathbb{Z}^{r}/k\mathbb{Z}^{r}$ is injective on~$\mathcal{A}$. Then the latter observation implies that for any $\eta\in \Gm^{r}(\mathbb{C}_{v})$ and
    each $a\in \mathcal{A}$ we have
    \begin{displaymath}
p_{a}=\frac{1}{k^{r}} \sum_{\zeta \in \mu_{k}^{r}} \chi^{-a}(\eta  \cdot \zeta) \, p (\eta \cdot \zeta).
    \end{displaymath}
    Choosing $\eta=(\eta_1,\ldots,\eta_r)$ with
    $|\eta_1|_v=\ldots=|\eta_r|_v=1$ and assuming furthermore that
 $|k|_{v}=1$   if $v$ is non-Archimedean, we deduce
  from this identity that
        \begin{displaymath}
      \max_{a\in \mathcal{A}}|p_{a}|_{v}\le \max_{\zeta \in \mu_{k}^{r}} |p(\eta \cdot\zeta)|_{v}.
    \end{displaymath}
    The result follows then by taking $F= \eta \cdot \mu_{k}^{r}$ for
    any choice of $\eta$ lying both in the nonempty open subset
    $\bigcap_{\zeta \in \mu_{k}^{r}} \zeta \cdot U (\mathbb{C}_{v})$
    and in the dense
    $\{\eta\in\Gm^{r}(\mathbb{C}_{v}):|\eta_1|_v=\dots=|\eta_r|_v=1\}$.
  \end{proof}

  \begin{proof}[Proof of \cref{prop: lower log bound for Fv}]
  For $i=0,\dots, n-2$ denote by $\Delta_{i} \subset M_{\mathbb{R}}$
  the lattice polytope of the toric divisor~$D_{i}$ and set
  \begin{displaymath}
 p_{i}= \sum_{m\in \Delta_{i}\cap M}\chi^{m} \in \mathbb{C}_{v}[M],
  \end{displaymath}
  which is a Laurent polynomial whose Newton polytope coincides with~$\Delta_{i}$. Hence the corresponding nef toric divisor
  on $X_v$ agrees with the base change $D_{i,v}$ and we can consider the
  global section $s_{p_{i}} \in \Gamma(X_{v},\mathcal{O}(D_{i,v}))$.
  
  By \cref{thm: geometric main theorem} there exists a nonempty open
  subset $U\subset \mathbb{T}^{n-1}$ such that for all
  $(u_{0},\dots, u_{n-2}) \in U(\mathbb{C}_{v})$ the global sections
  $u_{0}^{*}s_{p_{0}}, \dots, u_{n-2}^{*}s_{p_{n-2}}, s_{f}$ meet
  properly and their zero set is contained in $X_{0,v}$.  Indeed, the
  direct application of the theorem ensures these properties for a
  generic twist of $s_{p_{0}}, \dots, s_{p_{n-2}}, s_{f}$, but they
  also hold without twisting $s_{f}$ because of their invariance under
  the diagonal action of $\mathbb{T}$ on $\mathbb{T}^{n}$.

  In this situation, the zero set in $X_{v}$ of
  $u_{0}^{*}s_{p_{0}}, \dots, u_{n-2}^{*}s_{p_{n-2}}, s_{f}$ agrees
  with the zero set in $\mathbb{T}_{v}$ of the Laurent polynomials
  $u_{0}^{*}{p_{0}}, \dots, u_{n-2}^{*}{p_{n-2}}, {f} $ and so
  \begin{displaymath}
Z_{X_{v}}(u_{0}^{*}s_{p_{0}}, \dots, u_{n-2}^{*}s_{p_{n-2}}, s_{f})= Z_{\mathbb{T}_{v}}(u_{0}^{*}p_{0}, \dots, u_{n-2}^{*} p_{n-2}, {f}) .
  \end{displaymath}
  Then for $t\in \mathbb{T}(\mathbb{C}_{v})$ the Poisson formula for
  the sparse resultant \cite[Theorem~1.1]{DS:pfsr} shows that the
  evaluation of $t^{*}g$ at this effective $0$-cycle can be written as
  the quotient of the resultant of
  $u_{0}^{*}p_{0}, \dots, u_{n-2}^{*}p_{n-2}, {f}, t^{*}g$ by the
  product of the resultants of a finite number of initial parts of
  $u_{0}^{*}{p_{0}}, \dots, u_{n-2}^{*}{p_{n-2}}, {f}$. Since all
  these resultants are polynomials in the coefficients of the
  corresponding family of Laurent polynomials, this implies that there
  are $P\in \mathbb{C}_{v}[M\oplus M^{\oplus (n-1)}]$ and
  $Q\in \mathbb{C}_{v}[M^{\oplus (n-1)}] \setminus \{0\}$ with
  \begin{equation}
    \label{eq:26}
  t^{*} g \, ( Z_{X_{v}}(u_{0}^{*}s_{p_{0}}, \dots, u_{n-2}^{*}s_{p_{n-2}}, s_{f}))=\frac{P(t,u_{0},\dots, u_{n-2})}{Q(u_{0},\dots, u_{n-2})}. 
\end{equation}
Combining with \cref{lem:2}, this implies that for each
$(u_{0},\dots, u_{n-2}) \in U(\mathbb{C}_{v})$ there is $c>0$ such that
\begin{equation}
  \label{eq:24}
      F_{v}(t)\ge \log |P(t,u_{0},\dots, u_{n-2})|_{v} - c
      \quad \text{ for all } t\in \mathbb{T} (\mathbb{C}_{v}).
    \end{equation}

    Now consider $P$ as a Laurent polynomial in
    $ \mathbb{C}_{v}[M] [ M^{\oplus (n-1)}]$ and write it as
\begin{displaymath}
  P=\sum_{a\in \mathcal{A}}P_{a} \, \chi^{a}
\end{displaymath}
for a finite subset $\mathcal{A} \subset M^{\oplus (n-1)}$ and
$P_{a}\in \mathbb{C}_{v}[M]$ for each~$a \in \mathcal{A}$.

We claim that the family $P_{a}\in \mathbb{C}_{v}[M]$,
$a\in \mathcal{A}$, defines the closed subset~$\Upsilon_{v} \subset \mathbb{T}_{v}$.  Indeed it is enough to show
this for the $\mathbb{C}_{v}$-points of the torus. For
$t\in \mathbb{T}(\mathbb{C}_{v})$ we have that $P_{a}(t)=0$ for all
$a$ if and only if $P(t,u_{0},\dots, u_{n-2})=0$ for all
$(u_{0},\dots, u_{n-2})\in U(\mathbb{C}_{v})$. By the identity in
\eqref{eq:26} this is equivalent to the fact that $t^{*}g $ vanishes
on the support of the $0$-cycle therein for all such $  (u_{0},\dots, u_{n-2})$, which is the same~as
    \begin{equation}
      \label{eq:37}
V_{X}( s_{f}, t^{*}s_{g}) \cap V_{X}(u_{0}^{*}s_{p_{0}}, \dots, u_{n-2}^{*}s_{p_{n-2}})  \ne \emptyset.
    \end{equation}

    On the other hand $t\in \Upsilon(\mathbb{C}_{v})$ if and only if
    $\codim(V_{X}(s_{f}, t^{*}s_{g})) \le 1$. Since
    $ u_{i}^{*}s_{p_{i}}$, $i=0,\dots, n-2$, are $n-1$ generic global
    sections of ample line bundles, this latter condition is
    equivalent to that in \eqref{eq:37}. Hence $P_{a}(t)=0$ for all
    $a$ if and only if $t\in \Upsilon(\mathbb{C}_{v})$, proving this
    claim.

    Finally, by \cref{lem:4} there is a finite subset
    $F\subset U(\mathbb{C}_{v})$ such that
    \begin{displaymath}
 \max_{(u_{0},\dots, u_{n-2})\in F}|P(t,u_{0},\dots, u_{n-2})|_{v} \ge        \max_{a\in
      \mathcal{A}}|P_{a}(t)|_{v}  \quad \text{ for all } t\in \mathbb{T} (\mathbb{C}_{v}).
    \end{displaymath}
    The statement then follows by taking the maximum of  the lower bounds in
    \eqref{eq:24} for $(u_{0},\dots, u_{n-2})\in F$.
\end{proof}

\subsection{Proof of \cref{thm: auxiliary function has log singularities}}
\label{sec:proof-theorem}

When $v$ is Archimedean we have that
$\mathbb{T}_v^{\an}=\mathbb{T}_v(\mathbb{C}_v)$, and  so the statement
is an immediate consequence of \cref{prop:1,prop: lower log bound for Fv}.

When $v$ is non-Archimedean, \cref{prop:2} shows that there is a
unique extension
$ \mathbb{T}_{v}^{\an} \to \mathbb{R}\cup \{-\infty\}$ that
is continuous with real values outside $ \Upsilon_v^{\an}$ and takes
the value $-\infty$ therein. The result then follows from \cref{prop:
  lower log bound for Fv} together with the density of
$\mathbb{T}_v(\mathbb{C}_v)$ in~$\mathbb{T}_v^{\an}$.

\section{A local logarithmic equidistribution}\label{sec: local log equidistribution}

We now begin building up the proof of \cref{thm: main theorem} following
the strategy proposed in \cref{sec:an-appr-crefc}. Our main result
here shows that in the two-codimensional situation each $v$-adic
summand in \eqref{eq:14} vanishes asymptotically (\cref{thm: local
  equidistribution for main theorem}).

Throughout this section we denote by $\mathbb{T}$ a split algebraic
torus over the number field~$\mathbb{K}$, of dimension $n$ and with
character lattice~$M$.

\subsection{Logarithmic equidistribution of torsion points}\label{sec: logarithmic equidistribution}

In this subsection we recall the results concerning the distribution
of Galois orbits of torsion points in analytic tori that will play a
key role for our present local treatment.

When $v$ is Archimedean, the analytic torus $\mathbb{T}_v^{\an}$
identifies with~$\mathbb{T}_v(\mathbb{C}_v)$. We have that
$\mathbb{S}_v$ is a compact subgroup of this topological group, and we
denote by $\nu_{v}$ its Haar probability measure.

The following result is a direct consequence of a theorem of Dimitrov
and Habegger~\cite{DH}.  For a closed subset
$H\subset \mathbb{T}_{v}$, its analytification $H^{\an}$ identifies
with $H(\mathbb{C}_{v})$.  Following \textit{loc. cit.}, we then say
that $H$ is \emph{essentially atoral} if the closure
of~$H^{\an}\cap \mathbb{S}_{v}$ in~$\mathbb{T}_{v}$ is a finite union
of subvarieties of codimension at least $2$ and proper torsion cosets,
that is translates of subtori by torsion points.

\begin{theorem}[Archimedean logarithmic equidistribution of torsion
  points]\label{thm: archimedean logarithmic equidistribution}
  Let $v$ be an Archimedean place of $\mathbb{K}$. Let  $H \subset \mathbb{T}_{v}$ be an
  essentially atoral closed algebraic subset and
  $\varphi\colon\mathbb{T}_v(\mathbb{C}_v)\to\mathbb{R}\cup\{-\infty\}$
  a function with at most logarithmic singularities
  along~$H(\mathbb{C}_v)$.  Then for any strict sequence
  $(\omega_\ell)_\ell$ of torsion points of
  $\mathbb{T}(\overline{\mathbb{K}})$ we have~that
\[
\lim_{\ell\to+\infty}\ \frac{1}{\#O(\omega_\ell)}\sum_{y\in O(\omega_\ell)_{v}}\varphi(y)=\int_{\mathbb{S}_v}\varphi\ d\nu_v.
\]
\end{theorem}
\begin{proof}
  After fixing an isometry between $(\mathbb{C}_v,|\cdot|_v)$
  and~$(\mathbb{C},|\cdot|)$ and an isomorphism between $\mathbb{T}$
  and $\Gm^n$ we can identify $\mathbb{T}_v(\mathbb{C}_v)$
  with~$(\mathbb{C}^{\times})^n$.  Under this identification, the
  compact torus $\mathbb{S}_v$ corresponds to the polycircle
\[
(S^1)^n=\{(z_1,\ldots,z_n)\in(\mathbb{C}^{\times})^n: |z_1|=\dots=|z_n|=1\}
\]
and $\nu_v$ to its Haar probability measure.

Let $h_1,\ldots,h_s \in \mathbb{C}[x_{1}^{\pm1}, \dots, x_{n}^{\pm1}]$ be a system of Laurent
polynomials defining~$H$ and consider the Laurent polynomial  defined by 
\begin{displaymath}
  h(x_{1},\dots, x_{n})=\sum_{i=1}^{s} \overline{h_{i}}(x_{1}^{-1},\dots, x_{n}^{-1}) \, h_{i}(x_{1},\dots, x_{n}),
\end{displaymath}
where $\overline{h_{i}}$ denotes the complex conjugate of $h_{i}$.
For all $ z=(z_{1},\dots, z_{n})\in (S^{1})^{n}$ we have that
$z_{i}\overline{z_{i}}=1$ for each $i$, and so
\begin{displaymath}
  h(z)=\sum_{i=1}^{s} \overline{h_{i}}(\overline{z}) \, h_{i}(z) =\sum_{i=1}^{s} | h_{i}(z)|^{2}.
\end{displaymath}
This readily implies that
$H(\mathbb{C}) \cap (S^{1})^{n} = V_{\Gm^{n}}(h)(\mathbb{C}) \cap
(S^{1})^{n}$.  As the closed subset $H$ is essentially atoral, the
Laurent polynomial $h$ is essentially atoral in the sense
of~\cite{DH}.

Since $V_{\Gm^{n}}(h)$ contains $H$, by \cref{rem:2} we have that
$\varphi$ has at most logarithmic singularities along the
analytification of this hypersurface, and so there are positive real numbers $c_{1},c_{2}$ such
that
\begin{displaymath}
  \varphi(z)\ge c_{1}\log|h(z)|-c_{2}  \quad \text{ for all } z\in (S^{1})^{n}
\end{displaymath}
thanks to the compactness of~$(S^{1})^{n}$.  Applying the
measure-theoretical lemma \cite[Lemma~6.3]{CLT} together with Bilu's
equidistribution theorem \cite{Bilu} we can now reduce the proof of
the theorem to the case $\varphi=\log|h|$. This boils down to a
particular case of \cite[Corollary 8.9]{DH}, as we next explain.

To do so, fix an embedding
$\iota\colon \overline{\mathbb{K}}\hookrightarrow \mathbb{C}$ and for
each $\ell$ consider the Galois group
${G_\ell=\Gal(\mathbb{K}(\omega_\ell)/\mathbb{K})}$. In the present
situation, we have to prove that
\begin{equation}\label{eq: integrals in DH proof}
  \lim_{\ell\to \infty}\frac{1}{\#G_\ell}\sum_{\sigma\in G_\ell}\log|h(\iota\circ \sigma(\omega_\ell))|
  =\m( \iota(h)),
\end{equation}
 where $\m(\iota(h))$ denotes the Mahler measure of the complex Laurent polynomial~$\iota(h)$.
 
 Set $\Gamma_b=(\mathbb{Z}/b\mathbb{Z})^{\times}$ for
 $b\in\mathbb{N}$.  With this notation, for each $\ell$ the group
 $\Gamma_{\ord(\omega_\ell)}$ agrees with the Galois group of the
 cyclotomic extension~$\mathbb{Q}(\omega_\ell)/\mathbb{Q}$.  Notice
 that $G_\ell$ is isomorphic to the Galois group
 of~$\mathbb{Q}(\omega_\ell)/(\mathbb{K}\cap\mathbb{Q}(\omega_\ell))$
 \cite[Chapter VI, Theorem~1.12]{Langbook}, and hence $G_{\ell}$ can
 be identified with a subgroup of~$\Gamma_{\ord(\omega_\ell)}$.  On
 the other hand, the conductor $\mathfrak{f}_{G_\ell}$ is the minimal
 positive integer $b\, |\ord(\omega_\ell)$ such that $G_\ell$ contains
 the kernel of the reduction
 map~$\Gamma_{\ord(\omega_\ell)}\to\Gamma_b$.

 To apply \cite[Corollary~8.9]{DH} we need to show that both
 $[\Gamma_{\ord(\omega_\ell)}:G_\ell]$ and $ \mathfrak{f}_{G_\ell}$
 are uniformly bounded above. When this is case, this result ensures
 that the discrepancy between the Mahler measure $\m(\iota(h))$ and
 its $\ell$-th approximant in \eqref{eq: integrals in DH proof} is
 bounded above by
 \begin{equation}
   \label{eq:13}
   {c}\, {\delta(\omega_{\ell})^{-\kappa}}
 \end{equation}
 for all $\ell$ such that the strictness degree~$\delta(\omega_\ell)$
 is sufficiently large, where both $c,\kappa$ are positive real
 numbers depending only on $h$.

 As $\mathbb{K}$ is a number field, it has finitely many subfields.
 Therefore there is a finite subset $S\subset \overline{\mathbb{K}}$
 of roots of unity such that for every $\ell$ there is
 $\eta(\ell)\in S$ with
 \begin{equation*}
\mathbb{K}\cap\mathbb{Q}(\omega_\ell)=\mathbb{K}\cap\mathbb{Q}(\eta(\ell) ) \and \mathbb{Q}(\eta(\ell))\subseteq\mathbb{Q}(\omega_\ell). 
\end{equation*}
Setting  for short $d_{\ell}=\ord(\omega_{\ell})$ and
$e_{\ell}=\ord(\eta(\ell))$, there is a  diagram
\[
\xymatrix@=0.6em{
&&\mathbb{Q}(\omega_\ell)\ar@{-}[lld]_{G_{\ell}}\ar@{-}[rd]\ar@{-}[dd]_{\Gamma_{d_{\ell}}}&
\\ \mathbb{K} \cap \mathbb{Q}(\omega_{\ell})\ar@{-}[rrd]& & &\mathbb{Q}(\eta(\ell))\ar@{-}[ld]^{\Gamma_{e_{\ell}}}\\
&&\mathbb{Q}&
}
\]
for each $\ell$. This readily implies that
\begin{displaymath}
  [\Gamma_{d_{\ell}}:G_\ell]=[\mathbb{K}\cap\mathbb{Q}(\omega_\ell):\mathbb{Q}]\leq[\mathbb{K}:\mathbb{Q}].
\end{displaymath}
On the other hand, the Galois group of the extension
$\mathbb{Q}(\omega_\ell)/\mathbb{Q}(\eta(\ell))$ is isomorphic to the
kernel of the reduction map~$\Gamma_{d_{\ell}}\to\Gamma_{e_{\ell}}$
and it is a subgroup of~$G_\ell$, implying
that~$\mathfrak{f}_{G_\ell}\leq e_{\ell}$.  Therefore
\begin{displaymath}
  \mathfrak{f}_{G_\ell}\leq\max_{\eta\in S}\ord(\eta) \quad \text{ for all } \ell,
\end{displaymath}
and so it is also uniformly bounded above.

As the sequence $(\omega_\ell)_\ell$ is strict, its strictness degree
diverges and  the required convergence follows from \eqref{eq:13}.
\end{proof}

\begin{remark}
  It is expected that the technical condition on $H$ in \cref{thm:
    archimedean logarithmic equidistribution} is not necessary for the
  conclusion of the theorem, as discussed in  \cite[Conjecture~1.3]{DH}.
\end{remark}

When $v$ is non-Archimedean, recall  that the  \emph{Gauss
  point} $\zeta_{v}$ of~$\mathbb{T}_v^{\an}$ is the point of this Berkovich analytic space
corresponding to the multiplicative seminorm on $\mathbb{C}_v[M]$
defined by
\[
\|f\|_{\zeta_{v}} =\max_m|c_m|_v \quad \text{ for }  f= \sum_{m} c_m\chi^m\in\mathbb{C}_v[M].
\]
It is a point of $\mathbb{S}_v$, and the Dirac delta measure at it is
the non-Archimedean analogue of the Haar probability measure on the
compact torus from the Archimedean case.

The following result is a direct consequence of the theorem of Tate
and Voloch on linear forms in $p$-adic roots of unity
\cite{Tate_Voloch}.

\begin{theorem}[non-Archimedean logarithmic equidistribution of torsion points]\label{thm: nonarchimedean logarithmic equidistribution}
  Let $v$ be a non-Archimedean place of~$\mathbb{K}$ and
  $\varphi \colon \mathbb{T}_{v}^{\an}\to \mathbb{R}\cup \{-\infty\}$
  a function with at most logarithmic singularities along $H^{\an}$
  for a closed subset $H\subset \mathbb{T}_v$.  Then for any strict
  sequence $(\omega_\ell)_\ell$ of torsion points of
  $\mathbb{T}(\overline{\mathbb{K}})$ we have that
\[
\lim_{\ell\to+\infty}\ \frac{1}{\#O(\omega_\ell)}\sum_{y\in O(\omega_\ell)_v}\varphi(y)=\varphi(\zeta_{v}).
\]
\end{theorem}

\begin{proof}
  As in the proof of \cref{thm: archimedean logarithmic
    equidistribution} we can assume that~$\mathbb{T}=\Gm^n$.
      Consider a nonzero Laurent polynomial
  $h\in\mathbb{C}_v[x_1^{\pm1},\ldots,x_n^{\pm1}]$ vanishing on~$H$.
  Then $\varphi$ has logarithmic singularities along
  $V_{\Gm^{n}}(h)^{\an}$ by \cref{rem:2}, and so it is sufficient to
  show the result for $\varphi=\log|h|_{v}$, similarly as for the
  proof of \cref{thm: archimedean logarithmic equidistribution}. This
  amounts to show~that
\begin{equation}\label{eq: claim in proof of nonarchimedean log equidistribution}
\lim_{\ell\to+\infty}\ \frac{1}{\#O(\omega_\ell)}\sum_{y\in O(\omega_\ell)_v}\log|h(y)|_{v}=\log\|h\|_{\zeta_{v}}.
\end{equation}

By Tate--Voloch's theorem \cite[Theorem~2]{Tate_Voloch}, there is a
constant $c>0$ such that for any family
$\eta_1,\ldots,\eta_n\in\mathbb{C}_{v}$ of roots of unity
with~$h(\eta_1,\ldots,\eta_n)\neq0$ we have that
$|h(\eta_1,\ldots,\eta_n)|_{v}\geq c$. In particular, for such a
family of roots of unity we have that
\begin{equation}\label{eq: inequality in proof of nonarchimedean log equidistribution}
c\leq|h(\eta_1,\ldots,\eta_n)|_{v}\leq\|h\|_{\zeta_{v}}.
\end{equation}

Let $\rho\colon \mathbb{T}_{v}^{\an}\to \mathbb{R}$ be a continuous
function with compact support and value $1$ on~$\mathbb{S}_{v}$. Its
existence is ensured by \cite[Corollaire~(3.3.4)]{CLD} since
$\mathbb{S}_{v}$ is compact.  We now consider the function
$\psi\colon \mathbb{T}_v^{\an}\to\mathbb{R}$ defined as
\[
\psi(x)=\rho(x) \log\max(|h|_x,c).
\]
It is a continuous and compactly supported function, and so it can be
extended to a continuous function on the analytic projective space.
Applying Chambert-Loir's non-Archimedean equidistribution theorem as
in \cite[Exemple 3.2]{C-L}
we thus obtain
\[
\lim_{\ell\to+\infty}\ \frac{1}{\#O(\omega_\ell)}\sum_{y\in O(\omega_\ell)_v}\psi(y)=\psi(\zeta_{v}).
\]
Since the sequence $(\omega_{\ell})_{\ell}$ is strict, for $\ell$
large enough the Galois orbit $O(\omega_{\ell})$ eventually avoids the
zero set of~$h$ and so for each $y\in O(\omega_{\ell})_{v}$ we have
that $\psi(y)=\log|h(y)|_{v}$.  Finally
$\psi(\zeta_{v})=\log\|h\|_{\zeta_{v}}$ because of~\eqref{eq:
  inequality in proof of nonarchimedean log equidistribution}, proving
\eqref{eq: claim in proof of nonarchimedean log equidistribution}.
\end{proof}

\subsection{The local vanishing}
\label{sec:local-vanishing}

We now come back to the situation of the present paper, and address the
main objective of this section.  We show that, in the
$2$-codimensional situation and under suitable hypotheses, each of the
local error terms in the recursive expression for the height in
\cref{prop: induction in general strategy} converges to zero for
strict sequences of torsion points.

We set ourselves in the same hypotheses and notations of
\cref{sec:an-auxil-funct-1}. In particular, $f,g\in \mathbb{K}[M]$ are
two absolutely irreducible Laurent polynomials whose supports
contain the origin lattice point, $X$ is a smooth projective toric
variety compactifying~$\mathbb{T}$ whose fan is compatible with the
Newton polytopes of $f$ and~$g$, and
$\overline{D}_0,\ldots,\overline{D}_{n-2}$ is a family of semipositive
toric metrized divisors on~$X$ with very ample underlying divisors and
with smooth metrics at the Archimedean places and algebraic metrics at
the non-Archimedean ones.

For all $v\in\mathfrak{M}$, similarly as in \cref{sec:an-appr-crefc}, we consider the function $I_v\colon\mathbb{T}(\mathbb{C}_v)\to\mathbb{R}\cup\{-\infty\}$ defined by
\[
I_v(t)=\int_{X_{v}^{\an}}\log\|t^*s_{g}\|_{\Ron,v}\ \chern(\overline{D}_{0,v})\wedge\ldots\wedge\chern(\overline{D}_{n-2,v})\wedge
\delta_{Z_X(s_{f})_v^{\an}},
\]
where $\|\cdot\|_{\Ron,v}$ stands for the $v$-adic Ronkin metric on
$D_g$ as in~\eqref{eq: Ronkin metric}.  This function takes the value
$-\infty$ precisely on the $\mathbb{C}_v$-points of the proper closed
subset $\Upsilon$ of $\mathbb{T}$ defined in  \eqref{eq: set of bad
  intersection}.

By combining the local logarithmic equidistribution theorems for
torsion points from \cref{sec: logarithmic equidistribution}, the
definition of the Ronkin function and \cref{thm: auxiliary function
  has log singularities}, we obtain the following 
asymptotic vanishing for~$I_v$.

\begin{theorem}\label{thm: local equidistribution for main theorem}
In the previous hypotheses and notations, assume moreover that if $v$ is Archimedean the closed subset $\Upsilon_v\subset\mathbb{T}_v$ is essentially atoral.
Then, for all strict sequence $(\omega_\ell)_\ell$ of torsion points of $\mathbb{T}(\overline{\mathbb{K}})$ we have
\[
\lim_{\ell\to \infty}\frac{1}{\# O(\omega_{\ell})}\sum_{\eta\in O(\omega_\ell)_{v}} I_{v}(\eta)=0.
\]
\end{theorem}
\begin{proof}
Recall from \eqref{eq: twist of section} that for all $t\in\mathbb{T}(\mathbb{C}_v)$ we have the equality $t^*s_g=(t^*g)\,s_{D_g}$ of rational sections of the toric line bundle~$\mathcal{O}(D_g)$.
Together with the definition of the $v$-adic Ronkin metric on $D_g$ this relation yields
\begin{equation}\label{eq: decomposition of Ronkin}
\log\|t^*s_g\|_{\Ron,v}
=
\log|t^*g|+\rho_{g,v}\circ\val_v
\quad
\text{on }X_{0,v}^{\an},
\end{equation}
where $\rho_{g,v}$ is the $v$-adic Ronkin function of~$g$.
The measure
\[
\chern(\overline{D}_{0,v})\wedge\ldots\wedge\chern(\overline{D}_{n-2,v})\wedge
\delta_{Z_X(s_{f})_v^{\an}}
\]
has zero mass on the analytic boundary $X_v^{\an}\setminus X_{0,v}^{\an}$, and to lighten the notation we write $\mu$ for its restriction to~$X_{0,v}^{\an}$.
It follows then from \eqref{eq: decomposition of Ronkin} that
\begin{equation}
  \label{eq:41}
I_v(t)
=
F_v(t)
+
\int_{X_{0,v}^{\an}}(\rho_{g,v}\circ\val_v)\,d\mu  
\end{equation}
for all~$t\in\mathbb{T}(\mathbb{C}_v)$, with $F_v$ being the $v$-adic auxiliary function from~\eqref{eq: auxiliary function}.
In particular, for a strict sequence $(\omega_\ell)_\ell$ of torsion points of~$\mathbb{T}(\overline{\mathbb{K}})$ we obtain that
\begin{equation}\label{eq: main eq in local vanishing}
\lim_{\ell\to \infty}\frac{1}{\# O(\omega_{\ell})}\sum_{\eta\in O(\omega_\ell)_{v}} I_{v}(\eta)
=
\lim_{\ell\to \infty}\frac{1}{\# O(\omega_{\ell})}\sum_{\eta\in O(\omega_\ell)_{v}}F_v(\eta)
+
\int_{X_{0,v}^{\an}}(\rho_{g,v}\circ\val_v)\,d\mu.
\end{equation}

To conclude the proof it is hence enough to show that the right hand side of \eqref{eq: main eq in local vanishing} vanishes.
We do this by distinguishing the Archimedean and non-Archimedean cases.

Suppose first that $v$ is Archimedean.
In this case, the valuation map $\val_v\colon X_{0,v}^{\an}\to N_{\mathbb{R}}$ is a group homomorphism, and for each $x\in X_{0,v}^{\an}$ the fiber of $\val_v(x)$ agrees with the coset of $x$ under the preimage of~$0$, that is with the set~$\mathbb{S}_v\cdot x$.
Thus, denoting by $\nu_v$ the Haar probability measure on~$\mathbb{S}_v$, the definition of the Ronkin function and Fubini--Tonelli's theorem yield
\begin{equation}\label{eq: Archimedean value of Ronkin integral}
\begin{split}
\int_{X_{0,v}^{\an}}(\rho_{g,v}\circ\val_v)\,d\mu
&=
\int_{X_{0,v}^{\an}}\int_{\val_v^{-1}(\val_v(x))}-\log|g|_v\,d\nu_v\,d\mu(x)
\\&=
\int_{X_{0,v}^{\an}}\int_{\mathbb{S}_v}-\log|g(\theta\cdot x)|_v\,d\nu_v(\theta)\,d\mu(x)
\\&=
\int_{\mathbb{S}_v}\int_{X_{0,v}^{\an}}-\log|(\theta^*g)(x)|_v\,d\mu(x)\,d\nu_v(\theta)
\\&=
-\int_{\mathbb{S}_v}F_v\,d\nu_v.
\end{split}
\end{equation}
On the other hand, thanks to \cref{thm: auxiliary function has log singularities} the auxiliary function $F_v$ is a function on $\mathbb{T}_v^{\an}$ with at most logarithmic singularities along~$\Upsilon_v(\mathbb{C}_v)$, and $\Upsilon_v$ is an essentially atoral closed algebraic subset of $\mathbb{T}_v$ because of the hypotheses.
Therefore, \cref{thm: archimedean logarithmic equidistribution} together with \eqref{eq: main eq in local vanishing} and \eqref{eq: Archimedean value of Ronkin integral} implies that
\[
\lim_{\ell\to \infty}\ \frac{1}{\# O(\omega_{\ell})}\sum_{\eta\in O(\omega_\ell)_{v}} I_{v}(\eta)=0.
\]

Let us now suppose that $v$ is non-Archimedean.
Again by \cref{thm: auxiliary function has log singularities}, the auxiliary function extends to a function $F_v$ on $\mathbb{T}_v^{\an}$ with at most logarithmic singularities along~$\Upsilon_v^{\an}$.
Therefore, applying \cref{thm: nonarchimedean logarithmic equidistribution} to \eqref{eq: main eq in local vanishing} yields
\begin{equation}\label{eq: equation non-archimedean in local vanishing}
\lim_{\ell\to \infty}\frac{1}{\# O(\omega_{\ell})}\sum_{\eta\in O(\omega_\ell)_{v}} I_{v}(\eta)
=
F_v(\zeta_v)+\int_{X_{0,v}^{\an}}(\rho_{g,v}\circ\val_v)\,d\mu,
\end{equation}
where $\zeta_v$ is the Gauss point of~$\mathbb{T}_v^{\an}$.  To
compute this sum, we recall from the proof of \cref{prop:2} the
expression for the non-Archimedean Monge--Amp\`ere measure $\mu$ in
\eqref{eq: form of non-Archimedean Monge-Ampère} and the definition of
the extended auxiliary function in~\eqref{eq: extension of
  non-Archimedean auxiliary map}.  Using those, we obtain that the
right hand side in \eqref{eq: equation non-archimedean in local
  vanishing} is equal to a weighted sum of
\[
\log|\zeta_v^*g\, (\xi_{V,\zeta_v}))|_v+(\rho_{g,v}\circ\val_v)(\xi_V),
\]
where the sum ranges over the set of irreducible components of the special fiber of a suitable distinguished formal analytic model of~$Z_X(s_f)_v$, with $\xi_V$ and $\xi_{V,\zeta_v}$ as in the cited proof.

We claim that each of the previous terms is zero, which will be enough to conclude. 
On the one hand, let $\mathscr{A}_V$ be as in the proof of \cref{prop:2} and $\mathscr{B}$ be a completion of $\mathbb{C}_v[M]$ with respect to the Gauss norm.
They are $\mathbb{C}_v$-affinoid algebras satisfying $\xi_V\in\mathcal{M}(\mathscr{A}_V)$ and $\zeta_v\in\mathcal{M}(\mathscr{B})$ respectively.
Because of \cite[Lemme~3.1]{Poineau_angelique}, both canonical morphisms $\mathscr{A}_V\widehat{\otimes}_{\mathbb{C}_{v}}\mathscr{B}\to\mathscr{A}_V\widehat{\otimes}_{\mathbb{C}_{v}}\mathscr{H}(\zeta_v)$ and $\mathscr{A}_V\widehat{\otimes}_{\mathbb{C}_{v}}\mathscr{B}\to\mathscr{H}(\xi_V)\widehat{\otimes}_{\mathbb{C}_{v}}\mathscr{B}$ are isometries.
Therefore, writing~$g=\sum_m\beta_m\chi^m$ and using the definition in \eqref{eq:21}, we have that
\begin{equation*}
\begin{split}
|\zeta_v^*g\, (\xi_{V,\zeta_v})|_v
&=
\Big|\sum_m\beta_m\chi^m\otimes\chi^m(\zeta_v)\Big|_{\mathscr{A}_{V}\widehat{\otimes}_{\mathbb{C}_{v}}\mathscr{H}(\zeta_v)}
=
\Big|\sum_m\beta_m\chi^m\otimes\chi^m\Big|_{\mathscr{A}_{V}\widehat{\otimes}_{\mathbb{C}_{v}}\mathscr{B}}
\\&=
\Big|\sum_m\beta_m\chi^m(\xi_V)\otimes\chi^m\Big|_{\mathscr{H}(\xi_V)\widehat{\otimes}_{\mathbb{C}_{v}}\mathscr{B}}.
\end{split}
\end{equation*}
As $\zeta_v$ is the Gauss norm on~$\mathscr{B}$, the completed tensor product norm on $\mathscr{H}(\xi_V)\widehat{\otimes}_{\mathbb{C}_{v}}\mathscr{B}$ agrees with the Gauss norm on the corresponding algebra, see for instance \cite[Appendix B/Proposition 7]{Bosch_lectures}, so that we get
\begin{equation}\label{eq: non-Archimedean equation 1 in proof of local vanishing}
|\zeta_v^*g\, (\xi_{V,\zeta_v})|_v
=
\max_m|\beta_m|_v|\chi^m(\xi_V)|_v.
\end{equation}
On the other hand, the expression for the non-Archimedean Ronkin function from \cite[Remark~2.8]{Gualdi} implies that
\begin{multline}\label{eq: non-Archimedean equation 2 in proof of local vanishing}
(\rho_{g,v}\circ\val_v)(\xi_V)
=
\min_{m}(\langle m,\val_v(\xi_V)\rangle-\log|\beta_m|_v)
\\=
\min_{m}(-\log|\chi^m(\xi_V)|_v-\log|\beta_m|_v)
=
-\log\max_{m}(|\beta_m|_v|\chi^m(\xi_V)|_v).
\end{multline}
Thus, putting together \eqref{eq: non-Archimedean equation 1 in proof of local vanishing} and \eqref{eq: non-Archimedean equation 2 in proof of local vanishing} concludes the proof.
\end{proof}

\section{The adelic vanishing}\label{sec: adelic equidistribution}

We continue in the $2$-codimensional situation and keep the setting at
the beginning of \cref{sec:an-auxil-funct-1}.  Our aim here is to show
that for a strict sequence of torsion points the sum of the correcting
integrals over all places outside a specific finite subset converges
to zero.
To this end, we consider the following conditions on a place $v$
of~$\mathbb{K}$:
\begin{enumerate}[leftmargin=*]
\item \label{item:6} the restriction of $v$ to $\mathbb{Q}$ is a
  $p$-adic absolute value for a prime number $p\ne 2$,
\item \label{item:7} all nonzero coefficients of $f$ and $g$ have
$v$-adic absolute value equal to~$1$,
\item \label{item:8} for $i=0,\ldots,n-2$ the $v$-adic metric of
  $\overline{D}_{i}$ is canonical.
\end{enumerate}

Whenever these conditions are met, $v$ is a non-Archimedean place
and~$f,g\in\mathbb{C}_v^\circ[M]$.  This allows to consider the following additional conditions on~$v$:

\begin{enumerate}[resume,leftmargin=*]
\item \label{item:15} the value group of $v$ coincides with that of
  its restriction to $\mathbb{Q}$, namely
  $|\mathbb{K}^{\times}|_v =p^{\mathbb{Z}}$,
\item \label{item:9} the reductions
  $\widetilde{f}, \widetilde{g}\in \widetilde{\mathbb{C}}_{v}[M]$ are
 irreducible Laurent polynomials.
\end{enumerate}

We denote by $\mathfrak{S} \subset \mathfrak{M}$ the subset of places
for which at least one of the conditions~\eqref{item:6}-\eqref{item:9}
fails to hold.

\begin{lemma}\label{lemma: S is finite} The set $\mathfrak{S}$ is
finite.
\end{lemma}

\begin{proof} This is equivalent to the fact that each of the above
  conditions fails for at most finitely many places.  This is clear
  for \eqref{item:6}, \eqref{item:7} and \eqref{item:8} because the
  number field $\mathbb{K}$ has a finite number of places above each
  place of $\mathbb{Q}$, every $\alpha\in\mathbb{K}^\times$ has
  unitary $v$-adic absolute value for almost all $v$, and for each $i$
  the non-Archimedean metrics of $\overline{D}_i$ are defined by the
  canonical integral model of the line bundle $\mathcal{O}(D_{i})$ for
  almost all $v$.
  
  For the remaining we assume that the first three conditions are met.
  To the non-Archimedean place $v$ it corresponds a prime ideal
  $\mathfrak{p}$ of the ring of integers~$\mathcal{O}_{\mathbb{K}}$
  with ramification equal to the index of the value group
  $|\mathbb{Q}^{\times}|_{v}$ inside $|\mathbb{K}^{\times}|_{v}$.
  Thus the finiteness for \eqref{item:15} follows from the well-known
  fact that only finitely many rational primes ramify in a given
  number field \cite[Proposition~I.8.4]{Neukirch:ant}.  Furthermore
  the algebraic closure of the residue
  field~$\mathcal{O}_{\mathbb{K}}/\mathfrak{p}$ coincides with
  $\widetilde{\mathbb{C}}_v$. Hence the finiteness for~\eqref{item:9}
  is given by the classical result of Ostrowski stating that the
  reductions of an absolutely irreducible polynomial over $\mathbb{K}$
  are absolutely irreducible for all but a finite number of prime
  ideals of $\mathcal{O}_{\mathbb{K}}$ \cite[Hilfssatz at the end of
  page~296]{Ostr2}.
\end{proof}

As in \cref{sec:local-vanishing}, for each $v\in\mathfrak{M}$ we denote by
$I_v\colon\mathbb{T}(\mathbb{C}_v)\to\mathbb{R}\cup\{-\infty\}$ the
function defined by
\begin{equation*}
  I_v(t)=\int_{X_{v}^{\an}}\log\|t^*s_{g}\|_{\Ron,v}\
\chern(\overline{D}_{0,v})\wedge\ldots\wedge\chern(\overline{D}_{n-2,v})\wedge
\delta_{Z_v^{\an}}
\end{equation*}
with $\|\cdot\|_{\Ron,v}$ the $v$-adic Ronkin metric on the
divisor~$D_g$.  We devote the rest of this section to the proof of the
following result.

\begin{theorem}\label{thm: adelic equidistribution}
  For all strict sequence $(\omega_\ell)_\ell$ of torsion points of
  $\mathbb{T}(\overline{\mathbb{K}})$ we have
\begin{equation}
  \label{eq:16}
\lim_{\ell\to \infty}\sum_{v\in\mathfrak{M}\setminus \mathfrak{S}} \frac{n_{v}}{\# O(\omega_{\ell})} \hspace{-1.5mm} \sum_{\eta\in O(\omega_\ell)_{v}} I_{v}(\eta)=0.
\end{equation}
\end{theorem}

\subsection{Computing the integral}
\label{sec:an-expl-expr}

We start by giving an explicit expression for the function $I_{v}$ for
$v\in\mathfrak{M}\setminus\mathfrak{S}$ in terms of the coefficients
of the involved Laurent polynomials. To this end, write again
\[
f=\sum_{m}\alpha_{m}\chi^m \and g=\sum_{m}\beta_m\chi^m
\]
with $\alpha_{m},\beta_{m}\in \mathbb{K}$ that are zero except for
finitely many $m$.
Recall that $\mathbb{S}_v\subset\mathbb{T}_v^{\an}$ is the $v$-adic compact torus of $\mathbb{T}$ as in \eqref{eq: compact torus}.

\begin{proposition}\label{prop: explicit equalities for I}
  Let~$v\in\mathfrak{M}\setminus\mathfrak{S}$
  and~$t \in \mathbb{T}(\mathbb{C}_{v})\cap \mathbb{S}_{v}$.  If there
  is $m_0\in M$ such that~$\supp(g)=\supp(f)+m_0$ then
\[
  I_v(t) = \deg_{D_0,\ldots,D_{n-2}}(Z) \ \log
\max_{m,m'\in M} |\alpha_{m-m_{0}}\beta_{m'}\chi^{m'}(t)-\alpha_{m'-m_{0}}\beta_{m}\chi^{m}(t)|_{v}.
\]
Otherwise $I_v(t)=0$.
\end{proposition}

To prove this statement we need the following construction for a
fixed~$v\in \mathfrak{M}\setminus \mathfrak{S}$.  By the condition
\eqref{item:6} this place is non-Archimedean, and we can consider the
associated $\mathbb{C}_v$-affinoid algebra
$\mathbb{C}_{v}\langle M\rangle$ of \emph{strictly convergent Laurent
  series} \cite[Section 6.1.4]{BGR}. As a set, it consists of the
Laurent series $\sum_{m} \gamma_{m}\chi^{m} $ such that for every
$\varepsilon >0$ we have that $|\gamma_{m}|_{v} < \varepsilon$ for all
but a finite number of~$m \in M$.  We equip this set with the standard
addition and multiplication, and with the corresponding Gauss
norm~$\|\cdot\|_{\zeta_{v}}$.

We then consider the distinguished $\mathbb{C}_{v}$-affinoid algebra
\begin{equation}
  \label{eq:40}
\mathscr{A}= \mathbb{C}_{v}\langle M\rangle /(f)
\end{equation}
and denote by $\|\cdot \|$ its sup-seminorm. This latter is the
quotient of the Gauss norm, and so it is defined as
\begin{equation}
  \label{eq:53}
  \|q+ (f)\|=\inf_{h \in \mathbb{C}_{v}\langle M\rangle}\|q-h\, f\|_{\zeta_{v}} \quad \text{ for } q\in \mathbb{C}_{v}\langle M\rangle.
\end{equation}
The reduction of $\mathscr{A}$ agrees with 
\begin{math}
 \wA=\widetilde{\mathbb{C}}_{v}[M]/(\widetilde{f}),
\end{math}
and so it is an integral~$\widetilde{\mathbb{C}}_{v}$-algebra because
of the condition~\eqref{item:9}.

\begin{lemma}
  \label{lem:5}
  Let $q =\sum_{m} \kappa_{m}\chi^{m}\in \mathbb{C}_{v}[ M] $ with
  $|\kappa_{m}|_{v}=1$ for all $m \in \supp(q)$ such that its
  reduction is either invertible or irreducible.  If there is
  $m_{0}\in M$ such that $\supp(q)=\supp(f)+m_{0}$ then
  \begin{equation}
    \label{eq:27}
  \|q+(f)\|= \max_{m,m'\in M} |\alpha_{m-m_{0}}\kappa_{m'}-\alpha_{m'-m_{0}}\kappa_{m}|_{v}.
\end{equation}
Otherwise $  \|q+(f)\|= 1$.
\end{lemma}
\begin{proof}
  By the hypotheses and the condition \eqref{item:7} we have that
  $\|f\|_{\zeta_{v}}=\|q\|_{\zeta_{v}}=1$ and so
  $\|q+(f)\|\le \|q\|_{\zeta_{v}}=1$. On the other hand, let
  $ h=\sum_{m}\gamma_{m} \chi^{m}\in \mathbb{C}_{v}\langle M\rangle$ and
  suppose that
  \begin{math}
\|q-hf\|_{\zeta_{v}}<1.  
  \end{math}
  The ultrametric inequality implies that~$\|h\|_{\zeta_{v}}=1$, and so
  we can consider the reduction of $q-hf$ to obtain that
  \begin{equation}
    \label{eq:52}
    \widetilde{q}=\widetilde{h}\widetilde{f} \in \widetilde{\mathbb{C}}_{v}[M]. 
  \end{equation}
  When $\widetilde{q}$ is invertible this equality cannot hold because
  $\widetilde{f}$ is irreducible, and so in this case ${\|q+(f)\|=1}$.

  Hence for the rest of this proof we can restrict to the case when
  $\widetilde{q}$ is irreducible. Then \eqref{eq:52} implies that
  $\widetilde{h}$ is a monomial or equivalently that
  $ h= \gamma_{m_{0}} \chi^{m_{0}}+h'$ with $|\gamma_{m_{0}}|_{v}=1$
  and $ h'=\sum_{m\ne m_{0}}\gamma_{m} \chi^{m}$ satisfying
  $ \|h'\|_{\zeta_{v}}<1$.  We deduce that
  \begin{math}
    \supp(q)=\supp(\widetilde{q})= \supp(\widetilde{f})+m_{0} =\supp(f)+m_{0}, 
  \end{math}
  completing the proof of the second statement.

  For the first suppose that $\supp(q)=\supp(f)+m_{0}$ for
  some~$m_{0}$.  We have that $\|q+(f)\|=\|\chi^{-m_0}q+(f)\|$ and so
  we can also suppose without loss of generality that $m_{0}=0$.  With
  this normalization consider the finite subset
  \begin{displaymath}
\mathcal{C}=\supp(f)=\supp(q) \subset M.    
  \end{displaymath}
  It verifies that~$0\in \mathcal{C}$ and that
  $\mathcal{C}\setminus \{0\} \ne \emptyset$ because $f$ is absolutely
  irreducible.  Then the right hand side of \eqref{eq:27} can be
  written as
  \begin{multline}
    \label{eq:36}
    \max_{m,m'\in M} |\alpha_{m}\kappa_{m'}-\alpha_{m'}\kappa_{m}|_{v}
    =       \max_{m,m'\in \mathcal{C}} |\alpha_{m}\kappa_{m'}-\alpha_{m'}\kappa_{m}|_{v}
\\    =    \max_{m,m'\in \mathcal{C}} \Big|\frac{\kappa_{m}}{\alpha_{m}}-\frac{\kappa_{m'}}{\alpha_{m'}}\Big|_{v}=
    \max_{m\in \mathcal{C}\setminus \{0\}} \Big|\frac{\kappa_{m}}{\alpha_{m}}-\frac{\kappa_{0}}{\alpha_{0}}\Big|_{v},
  \end{multline}
  with the last equality coming from the ultrametric property.
  Denoting this quantity by $\rho$ we then have
  \begin{displaymath}
    \Big\| q- \frac{\kappa_{0}}{\alpha_{0}} f\Big\|_{\zeta_{v}} =\rho,
  \end{displaymath}
  and so the left hand side of \eqref{eq:27} is less or equal to the
  right hand side.

  If $\rho=0$ the proof is complete, so that we assume now
  that~$\rho>0$.  To prove the equality in this case, suppose that
  there exists $h\in \mathbb{C}_{v}\langle M\rangle $ such that
  $ \|q-hf\|_{\zeta_{v}}<\rho$. Since $\rho\le 1$ we can consider
  again the reduction of the Laurent series $q-hf$ as in
  \eqref{eq:52}, and using the fact that the supports of
  $\widetilde{f}$ and $\widetilde{q}$ coincide we deduce that
  $h=\gamma_{0}+h'$ with $|\gamma_{0}|_{v}=1$ and $h'$ a Laurent
  series with no constant term satisfying $\|h'\|_{\zeta_{v}}<1$.

  It follows from \eqref{eq:36} that
  $\|q-\gamma_{0} f\|_{\zeta_{v}}\ge \rho$, and since
  \begin{equation}
    \label{eq:38}
    \|q-\gamma_{0}f - h'f\|_{\zeta_{v}}=\|q-hf\|_{\zeta_{v}}<\rho
  \end{equation}
  this implies that
  $\|h'\|_{\zeta_{v}}=\|h'f\|_{\zeta_{v}}=\|q-\gamma_{0}f\|_{\zeta_{v}}$. Choose
  then $\tau\in \mathbb{C}_{v}^{\times}$ with
  $|\tau|_{v}=\|h'\|_{\zeta_v}$, so that setting $h''=\tau^{-1}h'$ and
  $q'=\tau^{-1} (q-\gamma_{0}f)$ we have that
  $\|h''\|_{\zeta_{v}}=\|q'\|_{\zeta_{v}}=1$. By \eqref{eq:38} we have
  $\|q'-h''f\|_{\zeta_v}<1$ and so the reductions of these Laurent
  polynomials satisfy
  \begin{displaymath}
    \widetilde{q'}= \widetilde{h''} \widetilde{f}. 
  \end{displaymath}
  Since
  $\supp(\widetilde{q'})\subset\supp(q')\subset\supp(f)=\supp(\widetilde{f})$
  this implies that $\widetilde{h''}$ is a nonzero constant, which is
  not possible because $h'$ has no constant term.
  \end{proof}

\begin{proof}[Proof of \cref{prop: explicit equalities for
    I}]
  Let $\mathcal{X}$ be the canonical model of the toric variety $X$
  and $\mathcal{Z} \subset \mathcal{X}$ the closure of the
  hypersurface $Z$.  By the condition \eqref{item:8}, for each $i$ the
  metric of $\overline{D}_{i,v}$ is the algebraic metric induced by
  the canonical model of the line bundle $\mathcal{O}(D_{i,v})$
  \cite[Definition 3.6.3]{BPS}.  This is a line bundle on
  $\mathcal{X}$, and we denote by $\mathcal{L}_{i}$ its pullback to
  $\mathcal{Z}$.

  These integral models induce a distinguished formal analytic variety
  $\mathfrak{Z}$ over $\mathbb{C}_{v}$ with generic fiber
  $Z_{v}^{\an}$ and special fiber
  $ \widetilde{\mathfrak{Z}}=\widetilde{\mathcal{Z}}$. The latter is
  an integral scheme over $\widetilde{\mathbb{C}}_{v}$ by the
  condition \eqref{item:9}.  For each $i$ we also obtain a formal
  analytic line bundle $\mathfrak{L}_i$ on $\mathfrak{Z}$. Its
  reduction is a line bundle on $\widetilde{\mathfrak{Z}}$ which
  coincides with the reduction $\widetilde{\mathcal{L}}_{i}$ of
  $\mathcal{L}_{i}$.

  Now let $t\in \mathbb{T}(\mathbb{C}_{v}) \cap \mathbb{S}_{v}$ and
  recall from \eqref{eq:41} that
\begin{equation}
  \label{eq:54}
I_v(t)
=
F_v(t)
+
\int_{X_{0,v}^{\an}}(\rho_{g,v}\circ\val_v)\, \chern(\overline{D}_{0,v})\wedge\ldots\wedge\chern(\overline{D}_{n-2,v})\wedge
\delta_{Z_v^{\an}}
\end{equation}
where $F_{v}$ is the auxiliary function in \eqref{eq: auxiliary
  function}, $\rho_{g,v}$ the $v$-adic Ronkin function of $g$, and
$\val_{v}$ the $v$-adic valuation map.  As in \eqref{eq: form of
  non-Archimedean Monge-Ampère}, the above Monge-Ampère measure is
equal to
\begin{math}
  \deg_{\widetilde{\mathfrak{L}}_0,\ldots,\widetilde{\mathfrak{L}}_{n-2}}(  \widetilde{\mathfrak{Z}}) \, \delta_{\xi_{\widetilde{\mathfrak{Z}}}},
\end{math}
where $\xi_{\widetilde{\mathfrak{Z}}}$ is the sup-seminorm of the the
distinguished $\mathbb{C}_{v}$-affinoid algebra $\mathscr{A}$ in
\eqref{eq:40}.  In our setting
\begin{displaymath}
\deg_{\widetilde{\mathfrak{L}}_0,\ldots,\widetilde{\mathfrak{L}}_{n-2}}(  \widetilde{\mathfrak{Z}}) =   \deg_{\widetilde{\mathcal{L}}_0,\ldots,\widetilde{\mathcal{L}}_{n-2}}(  \widetilde{\mathcal{Z}})=\deg_{D_0,\ldots,D_{n-2}}(Z)
\end{displaymath}
because of our previous considerations and the fact that the degree of
the fibers of a proper map is constant
\cite[Example~20.3.3]{Fulton_Intersection}.
  
Thus by the formula in \eqref{eq: non-Archimedean equation 2 in proof
  of local vanishing} and \cref{lem:5}, the second term in
\eqref{eq:54} is equal to
\begin{displaymath}
  -\deg_{D_0,\ldots,D_{n-2}}(Z) \, \log\max_{m}|\chi^{m}(\xi_{\widetilde{\mathfrak{Z}}})|_{v}=
  -\deg_{D_0,\ldots,D_{n-2}}(Z) \, \log \max_{m}  \|\chi^{m} +(f)\|=0.
\end{displaymath}

On the other hand, by \eqref{eq:18} we also have
\begin{displaymath}
    F_v(t) =
\deg_{D_0,\ldots,D_{n-2}}(Z) \, 
\,   \log|t^*g \, (\xi_{  \widetilde{\mathfrak{Z}}})|_{v}=
\deg_{D_0,\ldots,D_{n-2}}(Z) \, 
\,   \log \|t^{*}g +(f) \|.
\end{displaymath}
The statement then follows from \cref{lem:5}.
\end{proof}

\subsection{Bounds for the integral}
\label{sec:bounds-integral}

Fix $v\in \mathfrak{M}\setminus\mathfrak{S}$ and let $\mathcal{A} \subset M$ be a nonempty
finite subset together with a family of scalars
$\gamma_{m}\in \mathbb{K}$ with $|\gamma_{m}|_{v}=1$,
$m \in \mathcal{A}$.  For a torsion point
$\omega\in \mathbb{T}(\overline{\mathbb{K}})$ we set
\begin{equation}
  \label{eq:43}
  K_{v}(\omega)=\frac{1}{\# O(\omega)} \sum_{\eta\in O(\omega)_{v}} \log \max_{m\in\mathcal{A}} |\chi^{m}(\eta)-\gamma_{m}|_{v}.
\end{equation}
This quantity arises naturally when computing the mean of the function
$I_{v}$ over the $v$-adic Galois orbit of $\omega$, as we will see in
\cref{sec:proof-theorem-1}. In this subsection we give the bounds needed
for the proof of our adelic vanishing theorem.

To this end, we denote by $H\colon \mathbb{Z}^{\mathcal{A}}\to M$ the
linear map defined by
\begin{equation}\label{eq: linear map from support}
 H(a)= \sum_{m\in \mathcal{A}} a_{m}m,
\end{equation}
and by
$\phi\colon \mathbb{T} \to \Gm^{\mathcal{A}}$ the corresponding monomial
map, defined by
\begin{displaymath}
  \phi(t)= (\chi^{m}(t))_{m\in \mathcal{A}}.
\end{displaymath}
For
$a=(a_{m})_{m\in \mathcal{A}} \in \mathbb{Z}^{\mathcal{A}}$ and
$u\in
\Gm^{\mathcal{A}}(\overline{\mathbb{K}})=(\overline{\mathbb{K}}^{\times})^{\mathcal{A}}$
we set
\begin{math}
u^{a}=\prod_{m\in \mathcal{A}} u_{m}^{a_{m}} \in \overline{\mathbb{K}}^{\times}.  
\end{math}
With this notation we have that
\begin{displaymath}
\chi^{H(a)}(t)=\phi(t)^{a}.   
\end{displaymath}

Denote by $p_{v}$ the unique prime number such that the restriction of
$v$ to $\mathbb{Q}$ is the $p_{v}$-adic place.  We consider separately
the case when this prime number divides the order of~$\phi(\omega)$ and when it
does not.

\begin{proposition}
  \label{prop:6}
  If $p_{v} \mid \ord(\phi(\omega))$ then $\displaystyle{\frac{-\log(p_{v})}{\# O(\phi(\omega))} \le     K_{v}(\omega)\le 0.}$
\end{proposition}

\begin{proposition}
\label{prop:7}
  Suppose that $p_{v}\nmid  \ord(\phi(\omega))$ and let  $a \in \mathbb{Z}^{\mathcal{A}}$. 
  \begin{enumerate}[leftmargin=*]
  \item \label{item:13} If $\chi^{H(a)}(\omega)= 1$ then
    \begin{displaymath}
      \frac{\log |\gamma^{a}-1|_{v}}{\#
        O(\phi(\omega))} \le K_{v}(\omega) \le 0.
    \end{displaymath}
  \item \label{item:12} If   $\chi^{H(a)}(\omega)\ne 1$ and $\gamma^{a}=1$ then
    $K_{v}(\omega)=0$.
  \end{enumerate}
\end{proposition}

The proofs of these results rely on the following reduction properties of roots
of unity.
For a prime number~$p$, let $\mathbb{C}_p$ be the algebraically closed field of $p$-adic numbers, equipped with the standard $p$-adic absolute value~$|\cdot|_p$.
Its residue field is identified with an algebraic closure of the finite field of $p$ elements, and we denote by
\[
\red\colon \mathbb{C}_{p}^{\circ}\to \overline{\mathbb{F}}_{p}
\]
the corresponding reduction map.
Also, we write $\varphi$ for the Euler totient function.

\begin{lemma}
  \label{lem:8}
Let $\xi$ be a root of unity in $\mathbb{C}_p$ of order~$d= p^{e}\, b$, with $e\ge 0$ and
  $b\in \mathbb{Z}$ such that $p\nmid b$.
  \begin{enumerate}[leftmargin=*]
  \item \label{item:16} The element $\red(\xi)$ is a root of unity in $\overline{\mathbb{F}}_p$ of order~$b$, and the reduction map is $\varphi(p^{e})$-to-$1$ on the set of roots of unity of order $d$ in~$\mathbb{C}_p$.
  \item \label{item:18} For each root of unity $\rho\in\mathbb{C}_p$ of order $b$ we have that
    \begin{displaymath}
      |\xi-\rho|_{p}=
      \begin{cases}
        p^{-1/\varphi(p^{e})} & \text{ if } \red(\xi)=\red(\rho), \\
        1 & \text{ otherwise.}
      \end{cases}
    \end{displaymath}
  \end{enumerate}
\end{lemma}

\begin{proof}
  For \eqref{item:16}, consider first the case in which~$e=0$.
  In this situation, the polynomial $x^b-1$ over~$\mathbb{F}_p$ is separable since~$p\nmid b$, and then $\red(\xi)$ is a root of a single factor in
  \[
   x^{b}-1 = \prod_{c\mid b} \red(\Phi_c)\in\mathbb{F}_p[x],
  \]
  where $\red(\Phi_c)$ denotes the reduction to $\mathbb{F}_p$ of the cyclotomic polynomial~$\Phi_{c}\in\mathbb{Z}_p[x]$.
  Since $\Phi_b(\xi)=0$ we deduce that $\red(\xi)$ is not a root of $\red(\Phi_{c})$ for all
  $c\neq b$ and so $\ord(\red(\xi))=b$, as stated.
  Moreover, the separable polynomials $\Phi_b$ and $\red(\Phi_{b})$ have as many roots, in $\mathbb{C}_p$ and in $\overline{\mathbb{F}}_p$ respectively, as their common degree.
  Since by Hensel's lemma the reduction map between the sets of their zeros is surjective, we conclude that the map is $1$-to-$1$ in this case.

  When $e\ge 1$ the Frobenius homomorphism $\Frob\colon\overline{\mathbb{F}}_{p}\to     \overline{\mathbb{F}}_{p}$ mapping $x$ to $x^{p^{e}}$ is an automorphism, and so the above implies that
  \begin{displaymath}
    \ord(\red(\xi))=\ord(\Frob(\red(\xi)))=\ord(\red(\xi^{p^{e}}))=b.
  \end{displaymath}
The statement about the degree of the reduction map follows from a cardinality count and from the surjectivity of the morphism $x\mapsto x^{p^e}$ between the set of roots of unity of order $d$ and the ones of order $b$ in~$\mathbb{C}_p$.

  For \eqref{item:18}, the case in which $b=1$ is given by  \cite[Lemma~2.5]{GualdiSombra}, that is
  \begin{displaymath}
    |\xi-1|_{p}=p^{-1/\varphi(p^{e})}.
  \end{displaymath}
  When $b$ is arbitrary, we then have that
  \begin{displaymath}
    \prod_{\rho^{b}=1}|\xi-\rho|_{p}=|\xi^{b}-1|_{p}=p^{-1/\varphi(p^{e})}.
  \end{displaymath}
  Since $|\xi-\rho|_{p}=1$ whenever $\red(\xi)\neq\red(\rho)$, and since the equality $\red(\rho)=\red(\xi)$ is realized for a single root of unity $\rho$ of order $b$ in $\mathbb{C}_p$ because of~\eqref{item:16}, the
  claim follows.
\end{proof}

Let us now focus on the proof of \cref{prop:6,prop:7}.
Notice that, for each $\eta\in O(\omega)_{v}$ and~$a\in\mathbb{Z}^{\mathcal{A}}$, simple algebraic manipulations allow to write
\[
\chi^{ H(a)}(\eta)-\gamma^{a}=\phi(\eta)^a-\gamma^a
=\sum_{m\in \mathcal{A}} \varepsilon_{m} \,  (\chi^{m}(\eta)-\gamma_{m}),
\]
with each $\varepsilon_{m}$ being the product between a monomial in $(\chi^m(\eta))_m$ and one in $(\gamma_m)_m$.
In particular, we have that $|\varepsilon_m|_v=1$ for all $m\in\mathcal{A}$ and so
\begin{equation}
  \label{eq:48}
  | \chi^{ H(a)}(\eta)- \gamma^{a}|_{v}\le \max_{m\in \mathcal{A}} |  \chi^{m}(\eta)-\gamma_{m}|_{v},
\end{equation}
which turns out to be useful to lower bound the quantity $K_v(\omega)$ in~\eqref{eq:43}.

\begin{proof}[Proof of \cref{prop:6}]
  The second inequality is clear from the fact that
  $ |\chi^{m}(\eta)|_{v}=|\gamma_{m}|_{v}=1$ for all~$m$, and so we
  focus on the proof of the first one.
To this end, let~$d=\ord(\phi(\omega))$; since $p_v$ divides $d$ by hypothesis, we can write $d=p_v^eb$ with $e\geq1$ and $b$ an integer such that~$p_v\nmid b$.
Choose also
  $a\in \mathbb{Z}^{\mathcal{A}}$ such that the root of unity
  \begin{math}
\chi^{H(a)}(\omega)=\phi(\omega)^a
  \end{math}
  has order equal to~$d$; this is possible by writing
\[
  \phi(\omega)=(\zeta^{c_{1}},\dots, \zeta^{c_{r}})
  \]
for a root of unity $\zeta$ of order $d$ in $\overline{\mathbb{K}}$ and for $0\le c_{i}\le d-1$ with~$\gcd(c_1,\ldots,c_r,d)=1$.

  Identifying $(\mathbb{C}_v,|\cdot|_v)$ with~$(\mathbb{C}_{p_v},|\cdot|_{p_v})$, the character $\chi^{ H(a)}$ thus induces a surjective map from the
  $v$-adic Galois orbit $O(\omega)_{v}$ to the set $\mu_d^{\circ}$ of roots of unity
  in $\mathbb{C}_{p_v}$ of order~$d$.
  This map is $\#O(\omega)/\varphi(d)$-to-$1$ and so, together with \eqref{eq:48}, it yields
  \begin{equation}\label{eq: K in divisible case}
  K_v(\omega)
   \geq
    \frac{1}{\# O(\omega)} \sum_{\eta\in O(\omega)_{v}} \log| \chi^{ H(a)}(\eta)- \gamma^{a}|_{v}
    =
    \frac{1}{\varphi(d)} \sum_{\xi\in\mu_d^\circ} \log|\xi- \gamma^{a}|_{p_v}.
  \end{equation}

We can now suppose that there is some $\xi_0\in\mu_d^\circ$ for which
  $ |\xi_0-\gamma^{a}|_{p_v}<1$, because otherwise the statement holds
  trivially.
  Under this assumption $\red(\gamma^a)=\red(\xi_0)$ and so \cref{lem:8}\eqref{item:16} ensures that $\gamma^a$ has the same reduction as a root of unity~$\rho$ of order $b$ in~$\mathbb{C}_{p_v}$.
  It follows for instance from \cite[Proposition~16
  at page~77]{Serre:local} that $\rho$ lives in a
  unramified extension of~$\mathbb{K}_{v}$. Combined with the assumption \eqref{item:15} on~$v$, this yields
\[
    |\rho-\gamma^{a}|_{p_v}\leq p_v^{-1}<p_v^{-1/\varphi(p_v^{e})}<1,
\]
where the second inequality follows from the fact that $e\geq1$ and that $p_v>2$ by the assumption \eqref{item:6} on~$v$.
Therefore, using \cref{lem:8}\eqref{item:18} we obtain that
  \begin{equation}\label{eq: distance in divisible case}
    |\xi-\gamma^a|_{p_v}=|\xi-\rho+\rho-\gamma^a|_{p_v}= 
    \begin{cases}
        p_v^{-1/\varphi(p_v^{e})} & \text{ if } \red(\xi)=\red(\rho), \\
        1 & \text{ otherwise.}
      \end{cases}
      \end{equation}
      
  Thanks to \cref{lem:8}\eqref{item:16} there are exactly $\varphi(p^e)$ roots of unity $\xi\in\mu_d^\circ$ satisfying $\red(\xi)=\red(\rho)$.
  Hence, we conclude from \eqref{eq: K in divisible case} and \eqref{eq: distance in divisible case} that
  \[
   K_v(\omega)
   \geq
    \frac{\varphi(p_v^e)}{\varphi(d)} \log p_v^{-1/\varphi(p_v^{e})}
    =
    \frac{-\log(p_v)}{\varphi(d)}
    =\frac{-\log(p_{v})}{\# O(\phi(\omega))},
  \]
  as desired.
  \end{proof}

\begin{proof}[Proof of \cref{prop:7}]
Again, the inequality $K_v(\omega)\leq0$ is clear from the fact that
  $ |\chi^{m}(\eta)|_{v}=|\gamma_{m}|_{v}=1$ for all~$m$, and so we focus on proving lower bounds for~$K_v(\omega)$.
  In doing this, we call $d=\ord(\phi(\omega))$ and we identify $(\mathbb{C}_v,|\cdot|_v)$ with~$(\mathbb{C}_{p_v},|\cdot|_{p_v})$.

  To prove \eqref{item:13} we start by remarking that for each $\eta\in O(\omega)_v$ the inequality
  \begin{equation}\label{eq: inequality in indivisibility}
  \max_{m \in \mathcal{A}}|\chi^{m}(\eta)-\gamma_{m}|_{p_v}\neq1
  \end{equation}
  holds if and only if $\red(\chi^{m}(\eta))=\red(\gamma_{m})$ for all~$m\in\mathcal{A}$.
  We can suppose that \eqref{eq: inequality in indivisibility} is realized for a certain $\eta_0$, otherwise the sought bound is obviously true.
  We claim that in this case \eqref{eq: inequality in indivisibility} happens for exactly $\# O(\omega)/\# O(\phi(\omega))$ choices of $\eta\in O(\omega)_v$.
  Indeed, if $\eta$ satisfies \eqref{eq: inequality in indivisibility} it must also fulfil
  \[
  \red(\chi^m(\eta))=\red(\gamma_m)=\red(\chi^m(\eta_0)) \quad\text{ for all }m\in\mathcal{A}.
  \]
  But each $\chi^m(\eta)$ is a coordinate of $\phi(\eta)$ and hence it is a root of unity in $\mathbb{C}_{p_v}$ of order dividing~$d$.
  In particular, its order is coprime with $p_v$ by hypothesis and hence it must happen that $\chi^m(\eta)=\chi^m(\eta_0)$ because of \cref{lem:8}\eqref{item:16}; thus,~$\phi(\eta)=\phi(\eta_0)$.
  We deduce that $\eta$ and $\eta_0$ are conjugate by an element of the Galois group of the extension $\mathbb{K}(\omega)/\mathbb{K}(\phi(\omega))$, and there are precisely $\# O(\omega)/\# O(\phi(\omega))$ such conjugates.

For all of these $\eta\in O(\omega)_v$ we have by \eqref{eq:48} and by hypothesis on $a$ that 
  \begin{displaymath}
    |1-\gamma^{a}|_{v} \le \max_{m \in \mathcal{A}}|\chi^{m}(\eta)-\gamma_{m}|_{v},
  \end{displaymath}
  which implies the desired lower bound on~$K_v(\omega)$.
  
  To prove \eqref{item:12}, we have by the equality $\gamma^a=1$ and by \eqref{eq:48} that
  \begin{displaymath}
    K_v(\omega)\geq\frac{1}{\#O(\omega)}\sum_{\eta\in O(\omega)_v}\log|\chi^{ H(a)}(\eta)-1|_{p_v}.
  \end{displaymath}
  For any given~$\eta\in O(\omega)_v$, by hypotheses $\chi^{H(a)}(\eta)=\phi(\eta)^a$ is a root of unity in $\mathbb{C}_{p_v}$ different from $1$, and with order coprime with~$p_v$.
  Therefore, by \cref{lem:8}\eqref{item:16} its reduction has the same order, and in particular $\red(\chi^{H(a)}(\eta))\neq\red(1)$. It follows that $|\chi^{H(a)}(\eta)-1|_{p_v}=1$, proving the statement.
  \end{proof}

  We will also need to control the size of a relation of
  multiplicative dependence for a torsion point of a torus.

\begin{proposition}
  \label{prop:5}
 There is
  $a\in \mathbb{Z}^{\mathcal{A}}$ with $H(a)\ne 0$ such that
  \begin{displaymath}
    \chi^{H(a)}(\omega)=1 \and \max_{m\in \mathcal{A}}|a_{m}|\le c \, \ord(\phi(\omega))^{1/\rank(H(\mathbb{Z}^{\mathcal{A}}))}
  \end{displaymath}
  for a constant $c>0$ independent of  $\omega$.
\end{proposition}

\begin{proof}
  Set $d=\ord(\phi(\omega))$. For every
  $a\in \mathbb{Z}^{\mathcal{A}}$ we have that
  $\chi^{H(a)}(\omega)=\phi(\omega)^{a}$ is a $d$-th root of
  unity. Hence we can consider the homomorphism
  $\mathbb{Z}^{\mathcal{A}}/\Ker(H)\to \mu_{d}$ defined~as
  \begin{equation}
    \label{eq:57}
    a\longmapsto \chi^{H(a)}(\omega).
  \end{equation}
  It is surjective, and so its kernel is subgroup of
  $\mathbb{Z}^{\mathcal{A}}/\Ker(H)$ of index~$d$.

  Now set $r=\rank(H(\mathbb{Z}^{\mathcal{A}}) )$. We have that
  $\mathbb{Z}^{\mathcal{A}}/\Ker(H) \simeq
  H(\mathbb{Z}^{\mathcal{A}})$, and so we can choose an isomorphism
  \begin{equation}
    \label{eq:58}
  \mathbb{Z}^{\mathcal{A}}/\Ker(H) \simeq \mathbb{Z}^{r}.  
  \end{equation}
  This identifies the kernel of the homomorphism in \eqref{eq:57} with
  a subgroup $\Lambda \subset \mathbb{Z}^{r}$.

  Denote by $\|\cdot\|$ the max-norm on $\mathbb{Z}^{r}$. Since
  $\Lambda$ is a subgroup of index $d$, by Minkowski's first theorem
  there is $m_{0}\in \Lambda \setminus \{0\}$ such that
  \begin{math}
    \|m_{0}\|\le d^{1/r}.
  \end{math}
  Considering the linear map
  $ \mathbb{Z}^{\mathcal{A}}\to \mathbb{Z}^{r}$ induced by the
  isomorphism in \eqref{eq:58} and the Smith normal form of its
  associated matrix, we deduce that there is
  $a\in \mathbb{Z}^{\mathcal{A}}$  with  $H(a)=m_{0}$ satisfying
  \begin{displaymath}
\max_{m\in \mathcal{A}} |a_{m}|\le c\, \|m_{0}\| \le c\, d^{1/r}
\end{displaymath}
for a constant $c>0$ depending only on the coefficients of this
matrix, as desired. 
\end{proof}

\subsection{Proof of  \cref{thm: adelic equidistribution}}
\label{sec:proof-theorem-1}

By \cref{prop: explicit equalities for I} we can reduce to the case
when the supports of $f$ and $g$ agree up to a translation because
otherwise $I_v(\eta)=0$ for all
$v\in\mathfrak{M}\setminus\mathfrak{S}$ and
$\eta\in O(\omega_\ell)_{v}$, and so the double sum in \eqref{eq:16}
vanishes for all $\ell$. Assuming that this is the case, up to
multiplying by a suitable monomial we can also suppose that
$\supp(f)=\supp(g)$. Recall that this support contains the lattice
point $0\in M$.

Consider then the finite subset
$ \mathcal{A}=\supp(f) \setminus \{0\}=\supp(g) \setminus \{0\}
\subset M$, which is nonempty because of the hypothesis that $f$ and
$g$ are absolutely irreducible. Set also 
\begin{displaymath}
  \gamma_{m}=\frac{\alpha_{m} \beta_{0}}{\alpha_{0}\beta_{m}} \in \mathbb{C}_{v}  \quad \text{ for } m\in \mathcal{A}.
\end{displaymath}

Let $\ell\ge 1$.  In our present situation, for each
$v\in \mathfrak{M}\setminus \mathfrak{S}$ and
$\eta\in O(\omega_{\ell})_{v}$ we have
\begin{equation}
  \label{eq:50}
  \max_{m,m'\in M} |\alpha_{m}\beta_{m'}\chi^{m'}(\eta)-\alpha_{m'}\beta_{m}\chi^{m}(\eta)|_{v} = \max_{m \in \mathcal{A}}|\chi^{m}(\eta) -\gamma_{m}|_{v},
\end{equation}
as shown in \eqref{eq:36}.  Hence using the notation in \eqref{eq:43}
and \cref{prop: explicit equalities for I} we can write
\begin{equation}
  \label{eq:44}
  \sum_{v\in\mathfrak{M}\setminus \mathfrak{S}}\frac{n_{v}}{\# O(\omega_{\ell})} \hspace{-1.5mm} \sum_{\eta\in O(\omega_\ell)_{v}} I_{v}(\eta)=
  \deg_{D_{0},\dots, D_{n-2}}(Z)
  \sum_{v\in\mathfrak{M}\setminus \mathfrak{S}}n_{v} K_{v}(\omega_{\ell}).
\end{equation}
Thus to prove the statement it is enough to show the asymptotic
vanishing of the sum in the right-hand side of this equality.

Recall the linear map $H$ in \eqref{eq: linear map from support} and denote by $r$ the rank of the subgroup $H(\mathbb{Z}^{\mathcal{A}})$ of~$M$.
When $r=1$, the fact that $f$ and $g$ are absolutely irreducible implies that
$f=\alpha_{m}\chi^{m}+\alpha_{0}$ and $g=\beta_{m}\chi^{m}+\beta_{0}$ with   $m\in M$
primitive. Hence in this case  $\mathcal{A}=\{m\}$ and so 
\begin{multline*}
  \sum_{v\in\mathfrak{M}\setminus \mathfrak{S}}n_{v} K_{v}(\omega_{\ell}) =      \sum_{v\in\mathfrak{M}\setminus \mathfrak{S}}
  \frac{n_{v}}{\# O(\omega_{\ell})} \sum_{\eta\in O(\omega_{\ell})_{v}} \log|\chi^{m}(\eta)-{\gamma_{m}}|_{v}\\=
  \sum_{v\in \mathfrak{S}}
  \frac{-n_{v}}{\# O(\omega_{\ell})} \sum_{\eta\in O(\omega_{\ell})_{v}} \log|\chi^{m}(\eta)-{\gamma_{m}}|_{v},
\end{multline*}
where the second equality follows from the product formula. Since
$\mathfrak{S}$ is finite, this sum converges to $0$ as
$\ell\to \infty$ as a consequence of the local vanishing (\cref{thm:
  local equidistribution for main theorem}).

Hence from now on we suppose that $r\ge 2$.  Write
$d_{\ell}=\ord(\phi(\omega_{\ell}))$, and denote by $\mathcal{P}_{\ell}$ the
finite set of prime divisors of this integer and by $\mathfrak{P}_{\ell} $
for the finite subset of $\mathfrak{M}$ of places extending the $p$-adic
place of $\mathbb{Q}$ for some $p\in \mathcal{P}_{\ell}$.
We then split the sum in the
right hand side of \eqref{eq:44} as
\begin{equation}
  \label{eq:45}
  \sum_{v\in\mathfrak{M}\setminus \mathfrak{S}}n_{v} K_{v}(\omega_{\ell}) = S_{\ell}+T_{\ell}
\end{equation}
where $S_{\ell}$ denotes the sum over the places in
$\mathfrak{P}_{\ell}\setminus \mathfrak{S}$, and $T_{\ell}$ the
complementary sum.

For the first sum, we deduce from \cref{prop:6} that
\begin{equation}
  \label{eq:55}
  0\ge   S_{\ell}  = \sum_{v\in \mathfrak{P}_{\ell} \setminus \mathfrak{S}} n_{v} K_{v}(\omega_{\ell})
  \ge \sum_{v\in \mathfrak{P}_{\ell} \setminus \mathfrak{S}}  \frac{-n_{v} \log(p_{v})}{\# O(\phi(\omega_{\ell}))} \ge
  \frac{-1}{\# O(\phi(\omega_{\ell}))} \sum_{p\in \mathcal{P}_{\ell}}\log(p), 
\end{equation}
where the last inequality comes from the fact that
$\sum_{v\mid v_{0}}n_{v}=1$ for each place $v_{0}$ of~$\mathbb{Q}$.
Furthermore we have
\begin{equation}
  \label{eq:51}
\# O(\phi(\omega_{\ell}))= [\mathbb{K}(\phi(\omega_{\ell})): \mathbb{K}] =
\frac{[\mathbb{K}(\phi(\omega_{\ell})): \mathbb{Q}]}{[\mathbb{K}: \mathbb{Q}]} \ge 
\frac{[\mathbb{Q}(\phi(\omega_{\ell})): \mathbb{Q}]}{[\mathbb{K}: \mathbb{Q}]} =
\frac{\varphi(d_{\ell})}{[\mathbb{K}: \mathbb{Q}]}, 
\end{equation}
where $\varphi$ denotes the Euler totient function. We conclude that
\begin{equation}
  \label{eq:29}
    0\ge   S_{\ell} \ge -[\mathbb{K}:\mathbb{Q}] \, \frac{\log(d_{\ell})}{
  \varphi(d_{\ell})}.
\end{equation}

For the second sum, suppose first that there is
$a\in \mathbb{Z}^{\mathcal{A}}$ such that $H(a)\ne 0$ and $\gamma^{a}=1$.  Since the sequence $(\omega_{\ell})_{\ell}$ is strict,
there is $\ell_{0}\ge 1$ such that $\chi^{H(a)}(\omega_{\ell})\ne 1$
for $\ell\ge \ell_{0}$, and so by \cref{prop:7}\eqref{item:12} we have
that $T_{ \ell}=0$ for all such $\ell$'s.

Otherwise suppose that for all $a\in \mathbb{Z}^{\mathcal{A}}$ with
$H(a)\ne 0$ we have that $\gamma^{a}\ne 1$.  Applying \cref{prop:5} to the
torsion point $\omega_{\ell} $ we deduce that there is
$a_{\ell}\in \mathbb{Z}^{\mathcal{A}}$ with $H(a_{\ell})\ne 0$ such
that
\begin{equation}
  \label{eq:42}
\chi^{H(a_{\ell})}(\omega_{\ell})=1 \and
\|a_{\ell} \|\le c\, d_{\ell}^{1/r}  
\end{equation}
for a constant $c>0$
independent of $\ell$.  Then by \cref{prop:7}\eqref{item:13} 
\begin{multline}
  \label{eq:46}
  0\ge   T_{\ell}= \hspace{-3mm} \sum_{v\in \mathfrak{M} \setminus ( \mathfrak{S} \cup\mathfrak{P}_{\ell})} n_{v} K_{v}(\omega_{\ell})\ge \frac{1}{\# O(\phi(\omega_{\ell}))}
 \hspace{-1mm}  \sum_{v\in  \mathfrak{M}\setminus (\mathfrak{S}\cup \mathfrak{P}_{\ell})} \hspace{-3mm}n_{v} \log|\gamma^{a_{\ell}}-1|_{v}
\\  \ge \frac{-1}{\# O(\phi(\omega_{\ell}))}\sum_{v \mid \infty} n_{v} \log|\gamma^{a_{\ell}}-1|_{v} ,
\end{multline}
where the last inequality is a consequence of the product formula and
the fact that $|\gamma^{a_{\ell}}-1|_{v}\le 1$ whenever $v$ is
non-Archimedean. It follows from \eqref{eq:42} and \eqref{eq:46} that
\begin{equation}
  \label{eq:56}
  0\ge T_{\ell} \ge \frac{-c'\, d_{\ell}^{1/r} }{\# O(\phi(\omega_{\ell}))}
  \ge -c'\, [\mathbb{K}:\mathbb{Q}]\,  \frac{d_{\ell}^{1/r} }{\varphi(d_{\ell})}
\end{equation}
for a constant $c'>0$ independent of $\ell$, where the last inequality is ensured by \eqref{eq:51}.

To conclude, notice that the sequence of torsion points $(\phi(\omega_{\ell}))_{\ell}$ is strict in the torus~$\phi(\mathbb{T})$, which has dimension at least~$1$, and therefore $\lim_{\ell\to \infty}d_{\ell}= \infty$. 
The result now follows from \eqref{eq:29} and \eqref{eq:56} by taking
$\ell\to \infty$, using the hypothesis that $r\ge 2$ and the fact that
for any $\varepsilon >0$ we have
\begin{displaymath}
  \lim_{\ell\to \infty }\frac{d_\ell^{1-\varepsilon}}{\varphi(d_\ell)} =0 .   
\end{displaymath}

\section{Main results}\label{sec: main results}

We now put together the main results from the previous sections to prove \cref{thm: main theorem}.
In doing this, we need to treat separately the case in which the involved Laurent polynomials are binomials, due to the annoying hypothesis in the Archimedean logarithmic equidistribution of torsion points.
As an application, we show how the statement of the main theorem affirmatively answers previous questions by the authors.

As usual, throughout the whole section $\mathbb{K}$ denotes a number field, and $\mathbb{T}$ a split algebraic torus over $\mathbb{K}$ of dimension $n\geq2$ with character lattice~$M$.
Consider also a complete toric variety $X$ which compactifies~$\mathbb{T}$, and a family $\overline{D}_0,\ldots,\overline{D}_{n-2}$ of semipositive toric metrized divisors on~$X$, with corresponding collection $(\vartheta_{\overline{D}_i,v})_{v\in\mathfrak{M}}$ of $v$-adic roof functions, for $i=0,\ldots,n-2$.

\subsection{Proof of \cref{thm: main theorem}}

We start by remarking that when the Laurent polynomials involved in \cref{conj: main conjecture} have a large locus of bad intersection, its statement can be easily verified.

\begin{lemma}\label{lem: dimension of Upsilon}
Let $f,g$ be two absolutely irreducible Laurent polynomials in $\mathbb{K}[M]$ for which the closed subset $\Upsilon\subset\mathbb{T}$ from \eqref{eq: set of bad intersection} has codimension at most~$1$.
Then $f$ and $g$ are either monomials or binomials with the same support up to translation, and in this case \cref{conj: main conjecture} holds true.
\end{lemma}
\begin{proof}
It is clear that $f$ and $g$ must have the same support up to translation, as otherwise $\Upsilon$ is by definition the empty set.
We can then restrict to the case in which $\supp(g)=\supp(f)+\{m_0\}$ for some~$m_0\in M$, and denote by $\Lambda$ the linear span of the set $\supp(g)-\supp(g)=\{m-m^\prime: m,m^\prime\in\supp(g)\}$ in~$M_{\mathbb{R}}$.
We claim that
\begin{equation}\label{eq: dimension of Upsilon}
\codim(\Upsilon)\geq \dim(\Lambda),
\end{equation}
which will be enough to prove the first statement because of the absolute irreducibility of $f$ and~$g$.

To show \eqref{eq: dimension of Upsilon} notice that, denoting~$r=\dim(\Lambda)$, there exist $m,m_1,\ldots,m_r\in\supp(g)$ such that the lattice points $m_1-m,\ldots,m_r-m$ are linearly independent elements of~$M$.
We can then find a basis $e_1,\ldots,e_n$ of~$M$, and $\lambda_1,\ldots,\lambda_r\in\mathbb{Z}$ for which
\[
\lambda_i e_i=\sum_{j=1}^na_{ij}(m_j-m)
\]
for all $i=1,\ldots,r$, with each $a_{ij}\in\mathbb{Z}$.
We then have that $\Upsilon$ is contained in the zero set
\[
V_{\mathbb{T}}(\chi^{\lambda_ie_i}-\gamma_i:  i=1,\ldots,r)
\]
for an appropriate choice of~$\gamma_1,\ldots,\gamma_r\in\mathbb{K}$.
Under the identification between $\mathbb{T}$ and $\Gm^n$ given by the choice of the basis $\{e_1,\ldots,e_n\}$, the $\overline{\mathbb{K}}$-points of this last set agree with
\[
\big\{t\in (\overline{\mathbb{K}}^\times)^n: t_1^{\lambda_1}=\gamma_1,\ldots,t_r^{\lambda_r}=\gamma_r\big\},
\]
and then it has codimension precisely~$r$.

Let us now prove the second statement.
When $f$ and $g$ are monomials, \cref{prop: conjecture for monomials} ensures that the conjecture holds.
Suppose then that $f$ and $g$ are binomials with the same support up to translation.
After multiplying each of them by a suitable monomial, \cref{prop: conjecture for monomials} and \cref{rem: additivity of conjecture wrt products} allow to reduce the proof to the case in which $f=\chi^m-\alpha_0$ and~$g=\chi^m-\beta_0$.
For all $\omega\in\mathbb{T}(\overline{\mathbb{K}})$ the Laurent polynomials $\omega_1^*f,\omega_2^*g$ meet properly if and only if
\[
\alpha_0\chi^m(\omega_2)\neq\beta_0\chi^m(\omega_1),
\]
in which case~$Z_{\mathbb{T}}(\omega_1^*f,\omega_2^*g)=0$.  Therefore,
the left hand side of \cref{conj: main conjecture} vanishes for $f$
and~$g$.  The right hand one is also zero, as can be shown by noticing
that the $v$-adic Ronkin functions of $f$ and $g$ are explicitly
computable piecewise affine functions \cite[Example~2.12]{Gualdi}, and
by the properties of mixed integrals.
\end{proof}

We are finally ready to prove the main result of the article, that is the following special case of \cref{conj: main conjecture}.

\begin{theorem}[\cref{thm: main theorem}]\label{thm: main theorem in section}
Let $f,g$ be two nonzero Laurent polynomials in~$\mathbb{K}[M]$.
Then, for all quasi-strict sequence $(\bm{\omega}_\ell)_\ell$ of torsion points in~$\mathbb{T}(\overline{\mathbb{K}})^2$ we have that
\[
\lim_{\ell\to\infty}\h_{\overline{D}_0,\ldots,\overline{D}_{n-2}}(Z_{\mathbb{T}}(\bm{\omega}_{\ell,1}^*f,\bm{\omega}_{\ell,2}^*g))
=
\sum_{v\in\mathfrak{M}}n_v\MI_M(\vartheta_{\overline{D}_0,v},\ldots,\vartheta_{\overline{D}_{n-2},v},\rho_{f,v}^\vee,\rho_{g,v}^\vee).
\]
\end{theorem}
\begin{proof}
Thanks to \cref{prop: reduction steps} it is enough to consider the case in which $f,g$ are
absolutely irreducible Laurent polynomials with coefficients in~$O_{\mathbb{K}}$, and whose Newton polytopes contain the lattice point~$0\in M$.
By the same result, we can also assume that $X$ is smooth and projective, with fan compatible with the Newton polytopes of $f$ and~$g$, that $\overline{D}_0,\ldots,\overline{D}_{n-2}$ are very ample toric divisors on $X$ equipped with smooth toric metrics at the Archimedean places and algebraic toric metrics at the non-Archimedean ones, and that $(\bm{\omega}_\ell)_\ell$ is of the form $(1,\omega_{\ell})_\ell$ for a strict sequence $(\omega_\ell)_\ell$ of torsion points in~$\mathbb{T}(\overline{\mathbb{K}})$.

Under these assumptions, the Laurent polynomials $f$ and $g$ define nef toric divisors $D_f$ and $D_g$ on~$X$, and sections $s_f\in\Gamma(X,\mathcal{O}(D_f))$ and $s_g\in\Gamma(X,\mathcal{O}(D_g))$, respectively.

When the closed subset $\Upsilon\subset\mathbb{T}$ from \eqref{eq: set of bad intersection} has codimension at most~$1$, the validity of the conjecture follows from \cref{lem: dimension of Upsilon}.
Therefore, we can assume for the remaining of the proof that $\Upsilon$ has codimension at least $2$ in~$\mathbb{T}$, and in particular the base change $\Upsilon_v$ is an essentially atoral subset of~$\mathbb{T}_v$ for all Archimedean place $v$ of~$\mathbb{K}$.

We can now apply the strategy outlined in \cref{sec:an-appr-crefc}.
Since the sequence $(1,\omega_\ell)_{\ell}$ is quasi-strict, by \cref{cor: proper intersection for strict sequences} we have that for $\ell$ sufficiently large the global sections $s_f,\omega_\ell^*s_g$ meet properly and with no components outside of the principal open subset of~$X$.
Therefore, for such $\ell$, and under our assumptions, \cref{prop: induction in general strategy} is written as
\begin{equation}\label{eq: recursive equation in main proof}
\h_{\overline{D}_0,\ldots,\overline{D}_{n-2}}(Z_{\mathbb{T}}(f,\omega_\ell^*g))
=
\h_{\overline{D}_0,\ldots,\overline{D}_{n-2},\overline{D}^{\Ron}_{g}}(Z_{\mathbb{T}}(f))
+
\sum_{v\in\mathfrak{M}}\frac{n_{v}}{\# O(\omega_\ell)}\sum_{\eta\in O(\omega_\ell)_{v}} I_{v}(\eta),
\end{equation}
where for each place $v$ the function $I_v\colon\mathbb{T}(\mathbb{C}_v)\rightarrow\mathbb{R}\cup\{-\infty\}$ is defined by
\[
I_v(t)
=
\int_{X_{v}^{\an}}\log\|t^*s_{g}\|_{\Ron,v}\ \chern(\overline{D}_{0,v})\wedge\ldots\wedge\chern(\overline{D}_{n-2,v})\wedge
\delta_{Z_X(s_f)_v^{\an}}.
\]
Consider now the subset $\mathfrak{S}$ of $\mathfrak{M}$ as in beginning of \cref{sec: adelic equidistribution}.
Then, combining \cref{thm: Gualdi thesis main} and \eqref{eq: recursive equation in main proof} gives
\begin{equation*}
\begin{split}
\h_{\overline{D}_0,\ldots,\overline{D}_{n-2}}(Z_{\mathbb{T}}(f,\omega_\ell^*g))
=&
\sum_{v\in\mathfrak{M}}n_v\MI_M\big(\vartheta_{0,v},\ldots,\vartheta_{n-2,v},\rho_{f,v}^\vee,\rho_{g,v}^\vee\big)
\\&+
\sum_{v\in\mathfrak{S}}\frac{n_{v}}{\# O(\omega_\ell)}\sum_{\eta\in O(\omega_\ell)_{v}} I_{v}(\eta)
\\&+
\sum_{v\in\mathfrak{M}\setminus\mathfrak{S}}\frac{n_{v}}{\# O(\omega_\ell)}\sum_{\eta\in O(\omega_\ell)_{v}} I_{v}(\eta).
\end{split}
\end{equation*}
Since $\mathfrak{S}$ is finite by \cref{lemma: S is finite}, the second sum in the right-hand-side converges to $0$ as $\ell\to\infty$ because of \cref{thm: local equidistribution for main theorem}.
The third sum also converges to $0$ as $\ell\to\infty$ because of the adelic equidistribution result in \cref{thm: adelic equidistribution}.

Therefore, the expression is convergent, and the limit is the desired one.
\end{proof}

As an application, we can affirmatively answer \cite[Conjecture~11.8]{GualdiSombra}.

\begin{example}
Let $d_1,d_2\geq1$ and consider the Fermat polynomials
\[
f=1+x_1^{d_1}+x_2^{d_1}\and g=1+x_1^{d_2}+x_2^{d_2},
\]
seen as bivariate Laurent polynomials in~$\mathbb{Q}[x_1^{\pm1},x_2^{\pm1}]$.
Considering the canonically metrized hyperplane divisor on $\mathbb{P}^2_{\mathbb{Q}}$, and applying \cref{thm: main theorem} we deduce that the limit naive height of the intersection between the twists of $f$ and $g$ is computed by
\[
\lim_{\ell\to\infty}\h(Z_{\mathbb{T}}(\bm{\omega}_{\ell,1}^*f,\bm{\omega}_{\ell,2}^*g))
=
\sum_{v\in\mathfrak{M}}\MI_{\mathbb{Z}^2}(0_{\Delta},\rho_{f,v}^\vee,\rho_{g,v}^\vee),
\]
where $0_\Delta$ denotes the indicator function of the standard simplex $\Delta$ of~$\mathbb{R}^2$.
In the same line of what was done in \cite{GualdiSombra} for~$d_1=d_2=1$, it would be interesting to explore whether this limit can be expressed it in terms of special values of some relevant function.
\end{example}

\subsection{A consequences for strict sequences of sets}

We conclude the article by showing how \cref{thm: main theorem} can be
adapted to compute the asymptotic behaviour of average heights over
certain growing sets of torsion points.

To explain this, recall from \cite[Definition~6.4]{GualdiSombra} that a sequence $(E_\ell)_\ell$ of nonempty finite subsets of $\mathbb{T}(\overline{\mathbb{K}})^k$ is said to be \emph{strict} if for every proper algebraic subgroup $G\subset\mathbb{T}^k$ we have that
\[
\lim_{\ell\to\infty}\frac{\#(E_\ell\cap G(\overline{\mathbb{K}}))}{\#E_\ell}=0.
\]

This following can be seen as a generalization of \cref{thm: main theorem in section}, from which it actually follows.

\begin{proposition}
Let $f,g$ be two nonzero Laurent polynomials in~$\mathbb{K}[M]$, and denote by $H$ the set of $t\in\mathbb{T}(\overline{\mathbb{K}})^2$ for which $t_1^*f,t_2^*g$ do not meet properly.
Let also $(E_\ell)_{\ell}$ be a strict sequence of nonempty finite subsets of torsion points in~$\mathbb{T}(\overline{\mathbb{K}})^2$.
Then
\begin{multline*}
\lim_{\ell\to\infty}\frac{1}{\#E_\ell}\sum_{\omega\in E_\ell\setminus H}
\h_{\overline{D}_0,\ldots,\overline{D}_{n-2}}(Z_{\mathbb{T}}(\omega_1^*f,\omega_2^*g))
\\=
\sum_{v\in\mathfrak{M}}n_v\MI_M(\vartheta_{\overline{D}_0,v},\ldots,\vartheta_{\overline{D}_{n-2},v},\rho_{f,v}^\vee,\rho_{g,v}^\vee).
\end{multline*}
Moreover, if we denote by $L$ the previous limit, for all $\varepsilon>0$ we have that
\[
\lim_{\ell\to\infty}\frac{\#\{\omega\in E_{\ell}\setminus H :  |\h_{\overline{D}_0,\ldots,\overline{D}_{n-2}}(Z_{\mathbb{T}}(\omega_1^*f,\omega_2^*g))-L|<\varepsilon \}}{\#E_\ell}=1.
\]
\end{proposition}
\begin{proof}
Take the set of torsion points $\mathbb{T}(\overline{\mathbb{K}})^{\tors}\subset\mathbb{T}(\overline{\mathbb{K}})$ and consider the function $\varphi\colon\mathbb{T}(\overline{\mathbb{K}})^{\tors}\to\mathbb{R}$ defined by
\begin{equation*}
\varphi(\omega)=
\begin{cases}
\h_{\overline{D}_0,\ldots,\overline{D}_{n-2}}(Z_{\mathbb{T}}(\omega_1^*f,\omega_2^*g))&\text{ if }\omega\notin H,
\\0&\text{ otherwise}.
\end{cases}
\end{equation*}
This function is bounded above because of \cite[Proposition~6.4.1]{Guathesis}, see also \cite{MartinezSombra}.
By considering a finite number of families $s_{0,i},\ldots,s_{n-2,i}$ of global sections of $\mathscr{O}(D_0),\ldots,\mathscr{O}(D_{n-2})$ respectively, such that for each $t\in\mathbb{T}(\overline{\mathbb{K}})^2$ there is $i$ for which $s_{0,i},\ldots,s_{n-2,i},t_1^*f,t_2^*g$ meet properly,  we can also show that $\varphi$ is bounded below.

Then, the statement follows readily from \cite[Lemma~6.7]{GualdiSombra} and \cref{thm: main theorem in section}.
\end{proof}

As the sequence $(\tau(\ell))_\ell$ of the full sets of $\ell$-torsion points in $\mathbb{T}(\overline{\mathbb{K}})^2$ is strict \cite[Example~6.5]{GualdiSombra}, we obtain the following particular case of \cite[Conjecture~6.4.4]{Guathesis}.

\begin{corollary}\label{cor: confirm Gualdi in codimension 2}
For two nonzero Laurent polynomials $f,g$ in~$\mathbb{K}[M]$ we have that
\begin{multline*}
\lim_{\ell\to\infty}\frac{1}{\ell^2}\sum_{\omega\in \tau(\ell)\setminus H}
\h_{\overline{D}_0,\ldots,\overline{D}_{n-2}}(Z_{\mathbb{T}}(\omega_1^*f,\omega_2^*g))
\\=
\sum_{v\in\mathfrak{M}}n_v\MI_M(\vartheta_{\overline{D}_0,v},\ldots,\vartheta_{\overline{D}_{n-2},v},\rho_{f,v}^\vee,\rho_{g,v}^\vee),
\end{multline*}
where $H$ denotes the set of $t\in\mathbb{T}(\overline{\mathbb{K}})^2$ such that $t_1^*f,t_2^*g$ do not meet properly.
\end{corollary}

\appendix
\section{Proof of the reductions}\label{sec: appendix A}

This appendix is dedicated to the proof of \cref{prop: reduction steps} about the reduction steps that one can make when proving the general form of \cref{conj: main conjecture}.
They concern assumptions on different objects playing a role in the statement: the family of Laurent polynomials, the ambient toric variety, the sequence of torsion points and the collection of metrized divisors.
We show each of the corresponding reductions in a dedicated subsection.

\subsection{Proof of \cref{item: reduction of Laurent polynomials} in \cref{prop: reduction steps}}

We start by proving that the conjecture is verified in a trivial situation.

\begin{proposition}\label{prop: conjecture for monomials}
\cref{conj: main conjecture} holds true when one of the Laurent polynomial of the family $(f_1,\ldots,f_k)$ is a monomial.
\end{proposition}
\begin{proof}
Assume, without loss of generality, that~$f_1=\alpha\chi^m$ for some $\alpha\in\mathbb{K}$ and~$m\in M$.
We can show that under this hypothesis both sides in \cref{conj: main conjecture} are equal to zero.

The vanishing of the left hand side is immediate from the fact that
$Z_{\mathbb{T}}(\bm{t}^*\bm{f})=0$ for all admissible choice
of~$\bm{t}\in\mathbb{T}(\overline{\mathbb{K}})^k$.  Concerning the
right hand side, notice that in virtue of \cite[Proposition~2.9 and
Example~2.11]{Gualdi} and \cite[Proposition~2.3.3]{BPS} we have that
\[
\rho_{f_1,v}^\vee=\iota_{\{m\}}+\log|\alpha|_v
\]
for all $v\in\mathfrak{M}$, where $\iota_{\{m\}}$ is the function taking the value $0$ on $\{m\}$, and~$-\infty$ elsewhere.
Applying \cite[Corollary~1.10 and Proposition~1.3]{Gualdi} and \cite[Proposition~1.1.13]{Guathesis}, we obtain the equality
\begin{multline*}
\MI_M(\vartheta_{0,v},\ldots,\vartheta_{n-k,v},\rho_{f_1,v}^\vee,\ldots,\rho_{f_{k},v}^\vee)
\\=\log|\alpha|_v\cdot\MV_M(\Delta_{D_0},\ldots,\Delta_{D_{n-k}},\NP(f_2),\ldots,\NP(f_{k}))
\end{multline*}
for all~$v\in\mathfrak{M}$.
Therefore, the right hand side of \cref{conj: main conjecture} also vanishes, in virtue of the product formula on~$\mathbb{K}$.
\end{proof}

Another useful observation is that the statement of the conjecture behaves well under multiplication of the involved Laurent polynomials.

\begin{remark}\label{rem: additivity of conjecture wrt products}
Let $f,g,f_2,\ldots,f_k\in\mathbb{K}[M]$ be nonzero Laurent polynomials.
Assume that \cref{conj: main conjecture} holds for the family $(f,f_2,\ldots,f_k)$ and for the family~($g,f_2,\ldots,f_k)$.
Then it also holds for~$(f\cdot g,f_2,\ldots,f_k)$.

Indeed, it is enough to show that for the family $\bm{h}=(f\cdot g,f_2,\ldots,f_k)$ both sides of the equality in \cref{conj: main conjecture} coincide with the sum of the corresponding side for $\bm{f}=(f,f_2,\ldots,f_k)$ and~$\bm{g}=(g,f_2,\ldots,f_k)$.
First, the definition of the twist implies that whenever $\bm{t}=(t_1,\ldots,t_k)\in\mathbb{T}(\overline{\mathbb{K}})^k$ is such that $t_1^*(fg), t_2^*f_2,\ldots,t_k^*f_k$ meet properly, then also $t_1^*f, t_2^*f_2,\ldots,t_k^*f_k$ meet properly and $t_1^*g, t_2^*f_2,\ldots,t_k^*f_k$ meet properly, and moreover 
\[
Z_{\mathbb{T}}(\bm{t}^*\bm{h})=Z_{\mathbb{T}}(\bm{t}^*\bm{f})+Z_{\mathbb{T}}(\bm{t}^*\bm{g}).
\]

The linearity of the height function yields then, by passage to the limit, the claimed additivity for the left hand side of \cref{conj: main conjecture}.
The one for the right hand side follows from \cite[Proposition~2.9]{Gualdi}, \cite[Theorem~16.4]{R} and the multilinearity of mixed integrals with respect to sup-convolution.
\end{remark}

As a result of \cref{prop: conjecture for monomials} and \cref{rem:
  additivity of conjecture wrt products}, when proving \cref{conj:
  main conjecture} (or special cases of it), one can safely replace
any Laurent polynomial by its product with a nonzero monomial.  In
particular, up to multiplying by a monomial with suitable support and
with coefficient of high enough valuations, we can assume that for
each $i=1,\ldots,k$ the Laurent polynomial $f_i$ lies in
$O_{\mathbb{K}}[M]$ and that its support contains the lattice
point~$0\in M$.  This proves the first part of \cref{item: reduction
  of Laurent polynomials} in \cref{prop: reduction steps}.

For the second part, we need the following observation about the choice of the base field over which the Laurent polynomials are defined.

\begin{remark}\label{rem: equivalence of conjecture under field extension}
The validity of \cref{conj: main conjecture} for the family $(f_1,\ldots,f_k)$ over $\mathbb{K}$ is equivalent to its validity over any finite field extension $\mathbb{L}$ of~$\mathbb{K}$, as both sides in the statement of the conjecture are not affected by the change of base field.
Indeed, on the one hand, the height of a cycle does not depend on the choice of the number field over which it is defined.
On the other hand, for each place $w$ of $\mathbb{L}$ dividing a given place~$v\in\mathfrak{M}$, the $w$-adic roof function of the extension of $\overline{D}$ to $\mathbb{L}$ agrees with its $v$-adic roof function, and $\rho_{f_i,w}=\rho_{f_i,v}$ for all~$i\in\{1,\ldots,k\}$.
These facts, combined with the equality~$\sum_{w\mid v}n_w=n_v$, yield the invariance of the right hand side.
\end{remark}

Now, we can see each Laurent polynomial $f_i$ as an element of the unique factorization domain~$\overline{\mathbb{K}}[M]$, and take
\[
f_i=\prod_{j=1}^{r_i}f_{ij}
\]
to be its factorization into irreducible elements of~$\overline{\mathbb{K}}[M]$.
Denote by $\mathbb{L}\subset\overline{\mathbb{K}}$ the smallest field containing both $\mathbb{K}$ and the coefficients of the Laurent polynomials~$f_{ij}$ for all $i\in\{1,\ldots,k\}$ and~$j\in\{1,\ldots,r_i\}$.
It is a finite field extension of $\mathbb{K}$.
By using \cref{rem: equivalence of conjecture under field extension} and by repeatedly applying \cref{rem: additivity of conjecture wrt products} over~$\mathbb{L}$, we can reduce the proof of \cref{conj: main conjecture} for the family $(f_1,\ldots,f_k)$ to the proof of its validity for each of the family $(f_{1j_1},\ldots,f_{kj_k})$.
Since each of these consists of absolutely irreducible Laurent polynomials, the second part of \cref{item: reduction of Laurent polynomials} in \cref{prop: reduction steps} follows.

\subsection{Proof of \cref{item: reduction of toric variety} in \cref{prop: reduction steps}}

Let $\Sigma$ denote the fan of the toric variety~$X$.
After subdividing its cones, we can refine $\Sigma$ to a fan that is compatible with the Newton polytopes of the Laurent polynomials~$f_1,\ldots,f_k$.
It is possible to further refine the obtained fan in $N_{\mathbb{R}}$ to the fan~$\Sigma^\prime$ of a projective toric variety \cite[Theorem~6.1.18]{CLS} whose cones are generated by part of a basis of the lattice~$N$ \cite[Section~2.6]{Fulton_toric}.
As a result, by construction and by \cite[Section~2.1]{Fulton_toric} the toric variety $X^\prime$ associated with $\Sigma^\prime$ is a smooth projective compactification of $\mathbb{T}$ and its fan is compatible with the Newton polytopes of~$f_1,\ldots,f_k$.

We want to show that the validity of \cref{conj: main conjecture} for $X^\prime$ implies the one for~$X$.

To do so, consider the toric morphism $\phi\colon X^\prime\to X$ constructed from the identity map on $N$ as in \cite[Section~1.4]{Fulton_toric}.
It is a regular morphism whose restriction to the torus $\mathbb{T}$ agrees with the identity.

The pullbacks $\phi^*\overline{D}_0,\ldots,\phi^*\overline{D}_{n-k}$ are semipositive toric metrized divisors on~$X^\prime$, and thanks to \cite[Proposition~4.8.10]{BPS} their $v$-adic roof functions are the same as the ones of $\overline{D}_0,\ldots,\overline{D}_{n-k}$ at all place $v$ of~$\mathbb{K}$.

On the other hand, the arithmetic projection formula ensures that for all $\bm{t}=(t_1,\ldots,t_k)\in\mathbb{T}(\overline{\mathbb{K}})^k$ for which $t_1^*f_1,\ldots,t_k^*f_k$ meet properly
\[
\h_{\phi^*\overline{D}_0,\ldots,\phi^*\overline{D}_{n-k}}(Z_{\mathbb{T}}(\bm{t}^*\bm{f}))
=
\h_{\overline{D}_0,\ldots,\overline{D}_{n-k}}(Z_{\mathbb{T}}\big(\bm{t}^*\bm{f})),
\]
since $\phi_*\overline{Z_{\mathbb{T}}\big(\bm{t}^*\bm{f})}^{X^\prime}=\overline{Z_{\mathbb{T}}\big(\bm{t}^*\bm{f})}^{X}$ because the cycle has no components outside of~$\mathbb{T}$, and the restriction of $\phi$ to this torus is the identity map.

\cref{item: reduction of toric variety} in \cref{prop: reduction steps} follows then from the above considerations.

\subsection{Proof of \cref{item: reduction of geometric divisors} in \cref{prop: reduction steps}}

The following observation on the well behaviour of \cref{conj: main
  conjecture} under linear combination of semipositive toric metrized
divisors comes in handy.

\begin{remark}\label{rem: multilinearity of conjecture sum of metrized divisors}
Let $\overline{D}_0,\ldots,\overline{D}_{n-k},\overline{E}$ be  semipositive toric metrized divisors, and let~$i\in\{0,\ldots,n-k\}$.
If \cref{conj: main conjecture} holds for the families $\overline{D}_0,\ldots,\overline{D}_{n-k}$ and $\overline{D}_0,\ldots,\overline{D}_{i-1},\allowbreak\overline{E},\overline{D}_{i+1},\ldots,\overline{D}_{n-k}$, then it also holds for the family
\[
\overline{D}_0,\ldots,\overline{D}_{i-1},a\overline{D}_i+b\overline{E},\overline{D}_{i+1},\ldots,\overline{D}_{n-k}
\]
for all choice of~$a,b\in\mathbb{Z}$ for which $a\overline{D}_i+b\overline{E}$ is semipositive.
In fact, both sides of the equality in \cref{conj: main conjecture} are multilinear in the choice of the metrized divisors.

To show this, thanks to the symmetry of the height and of the mixed integral operator, it is enough to show that linearity holds in the first argument; we can then assume~$i=0$.
Whenever $\bm{t}=(t_1,\ldots,t_k)\in\mathbb{T}(\overline{\mathbb{K}})^k$ is such that $t_1^*f_1,\ldots,t_k^*f_k$ meet properly, we have
\[
\h_{a\overline{D}_0+b\overline{E},\overline{D}_1,\ldots,\overline{D}_{n-k}}(Z_{\mathbb{T}}(\bm{t}^*\bm{f}))
=
a\h_{\overline{D}_0,\overline{D}_1,\ldots,\overline{D}_{n-k}}(Z_{\mathbb{T}}(\bm{t}^*\bm{f}))
+
b\h_{\overline{E},\overline{D}_1,\ldots,\overline{D}_{n-k}}(Z_{\mathbb{T}}(\bm{t}^*\bm{f})),
\]
which implies the linearity of the left hand side in \cref{conj: main conjecture}.
On the other hand, and assuming that~$a,b\geq0$, \cite[Propositions~4.3.14, 2.3.1 and~2.3.3]{BPS} ensure that
\[
\vartheta_{a\overline{D}_0+b\overline{E},v}=(\vartheta_{\overline{D}_0,v}\cdot a)\boxplus(\vartheta_{\overline{E},v}\cdot b)
\]
 for all place $v$ of $\mathbb{K}$, with $\boxplus$ denoting the sup-convolution binary operation between concave functions, and $\cdot$ the right scalar multiplication.
This equality, plugged in the right hand side of \cref{conj: main conjecture}, yields its linearity because of the analogous properties of the mixed integral operator with respect to the sup-convolution of its arguments.

In the case in which $a\geq0$ and~$b<0$, the same conclusion is similarly reached after noticing that
\[
\vartheta_{\overline{D}_0,v}\cdot a=\vartheta_{a\overline{D}_0+b\overline{E}+|b|\overline{E},v}=\vartheta_{a\overline{D}_0+b\overline{E},v}\boxplus(\vartheta_{\overline{E},v}\cdot |b|).
\]
\end{remark}

Now, assume that \cref{conj: main conjecture} is valid for families of semipositive toric metrized divisors with very ample underlying divisors, and let us show that we can prove it in general.

Consider an arbitrary family $\overline{D}_0,\ldots,\overline{D}_{n-k}$ of semipositive toric metrized divisors.
Because of \cref{item: reduction of toric variety} in \cref{prop: reduction steps}, we can assume that $X$ is projective, and therefore consider a very ample toric divisor $A$ on~$X$.
Since every nef toric divisor is globally generated, and because of \cite[Exercise~7.5(d) at page~169]{Har}, we have that $D_i+A$ is also very ample for every~$0\leq i\leq n-k$.
Moreover, by ampleness, we can equip $A$ with a semipositive toric metric, and call $\overline{A}$ the corresponding semipositive toric metrized divisor.

By applying \cref{rem: multilinearity of conjecture sum of metrized divisors}, the validity of \cref{conj: main conjecture} for the family $\overline{D}_0,\ldots,\overline{D}_{n-k}$ is implied by the one of the analogous statement for the two families
\[
\overline{D}_0+\overline{A},\overline{D}_1,\ldots,\overline{D}_{n-k}\and\overline{A},\overline{D}_1,\ldots,\overline{D}_{n-k}.
\]

Reasoning in the same manner on the subsequent entries of these new families, we can inductively reduce the validity of \cref{conj: main conjecture} for $\overline{D}_0,\ldots,\overline{D}_{n-k}$ to the one for families whose members are either $\overline{A}$ or of the form~$\overline{D}_i+\overline{A}$, for~$0\leq i\leq n-k$.
As these all have very ample underlying divisors, \cref{item: reduction of geometric divisors} in \cref{prop: reduction steps} is proved.

\subsection{Proof of \cref{item: reduction of metrized divisors} in \cref{prop: reduction steps}}

We aim to show that the validity of \cref{conj: main conjecture} is preserved under uniform limits of the metrics.
We start by proving a lemma on the behaviour of mixed integrals with respect to approximations of concave functions, which can be of independent interest.

\begin{lemma}\label{lem: convergence in mixed integral}
Let $\mu$ be a Haar measure on a real vector space $V$ of dimension~$n$.
For all~$i\in\{0,\ldots,n\}$, let $(g_{i,j})_j$ be a sequence of concave functions having a closed convex body $Q_i$ of $V$ as common domain, and which converges uniformly to a concave function~$g_i$.
Then
\[
\lim_{j\to\infty}\MI_\mu(g_{0,j},\ldots,g_{n,j})
=
\MI_\mu(g_0,\ldots,g_n).
\]
\end{lemma}
\begin{proof}
By definition, the mixed integral $\MI_\mu(g_{0,j},\ldots,g_{n,j})$ is an alternating sum of integrals of the form
\[
\int_{Q_{i_1}+\ldots+Q_{i_k}}(g_{i_1,j}\boxplus\ldots\boxplus g_{i_k,j})\ d\mu,
\]
with $\{i_1,\ldots,i_k\}\subseteq\{0,\ldots,n\}$ and $\boxplus$ denoting the sup-convolution between concave functions.
One can show that this last operation preserves uniform limits, and therefore the sup-convolution of any subset of the~$g_{i,j}$ converges uniformly to the sup-convolution of their respective limits.
As a result, since the domain $Q_{i_1}+\ldots+Q_{i_k}$ has finite $\mu$-measure, each of the above integrals converges to
\[
\int_{Q_{i_1}+\ldots+Q_{i_k}}(g_{i_1}\boxplus\ldots\boxplus g_{i_k})\ d\mu,
\]
which is enough to conclude.
\end{proof}

Assume now that \cref{conj: main conjecture} holds for every choice of semipositive adelic metrized toric divisors whose metrics are smooth at Archimedean places and algebraic at non-Archimedean ones.
We want to show that it holds true for an arbitrary family $\overline{D}_0,\ldots,\overline{D}_{n-k}$ of semipositive toric metrized divisors.
As usual, we refer to \cite{BPS} for the terminology employed in the following.

By definition, the $v$-adic metric of each of the $\overline{D}_0,\ldots,\overline{D}_{n-k}$ is the canonical one for all $v$ outside of a finite set $\mathfrak{P}\subset\mathfrak{M}$ containing the Archimedean places.
Moreover, for each $i\in\{0,\ldots,n-k\}$ there is a sequence of semipositive toric metrized divisors
\[
\overline{D}_{i,j}=\big(D_i,(\|\cdot\|_{i,j,v})_{v\in\mathfrak{M}}\big)
\]
such that the metric $\|\cdot\|_{i,j,v}$ is semipositive smooth if $v$ Archimedean, it is associated with a semipositive algebraic model if $v$ is non-Archimedean, it is the canonical one if~$v\notin\mathfrak{P}$, and
\[
\lim_{j\to\infty}\dist(\|\cdot\|_{i,j,v},\|\cdot\|_{i,v})=0
\]
for all~$v$.
Because of \cite[Proposition~4.3.14(3)]{BPS} and the fact that uniform convergence of concave functions is preserved under Legendre--Fenchel duality, the sequence $(\vartheta_{\overline{D}_{i,j},v})_j$ of continuous concave functions on $\Delta_{D_i}$ converges uniformly to~$\vartheta_{\overline{D}_i,v}$, for all~$i\in\{0,\ldots,n-k\}$.

Let now $(\bm{\omega}_\ell)_\ell$ be a quasi-strict sequence of torsion points in~$\mathbb{T}(\overline{\mathbb{K}})^k$.
In proving \cref{conj: main conjecture} for this sequence, after approximating each $\overline{D}_i$ by $\overline{D}_{i,j}$ we are faced with a double limit, one over $\ell$ and the other over~$j$.
It will be crucial to be able to switch them, for which reason we need to prove that one of the convergence is uniform.

To do so, notice that for $\ell$ suitably large, the $(n-k)$-cycle $Z_{\mathbb{T}}(\bm{\omega}_\ell^*\bm{f})$ is well-defined thanks to the first part of \cref{cor: proper intersection for strict sequences}, and we can choose rational sections $s_i$ of~$\mathcal{O}(D_i)$, possibly depending on~$\ell$, such that $s_0,\ldots,s_{n-k}$ meet its closure $Z_X(\bm{\omega}_\ell^*s_{\bm{f}})$ in $X$ properly.
By the definition of the global height as sum of local contributions,
\begin{multline*}
\left|\h_{\overline{D}_{0,j},\ldots,\overline{D}_{n-k,j}}(Z_{\mathbb{T}}(\bm{\omega}_\ell^*\bm{f}))-\h_{\overline{D}_{0},\ldots,\overline{D}_{n-k}}(Z_{\mathbb{T}}(\bm{\omega}_\ell^*\bm{f}))\right|
\\\leq
\sum_{v\in\mathfrak{P}} n_v\Big|\h_{\overline{D}_{0,j,v},\ldots,\overline{D}_{n-k,j,v}}(Z_X(\bm{\omega}_\ell^*s_{\bm{f}});s_0,\ldots,s_{n-k})
\\-\h_{\overline{D}_{0,v},\ldots,\overline{D}_{n-k,v}}(Z_X(\bm{\omega}_\ell^*s_{\bm{f}});s_0,\ldots,s_{n-k})\Big|.
\end{multline*}

Each $v$-adic summand on the right hand side of the previous inequality can be upper bounded by the product between $n_v$ and the sum
\begin{multline*}
\sum_{r=0}^{n-k}\Big|\h_{\overline{D}_{0,v},\ldots,\overline{D}_{r-1,v},\overline{D}_{r,j,v},\ldots,\overline{D}_{n-k,j,v}}(Z_X(\bm{\omega}_\ell^*s_{\bm{f}});s_0,\ldots,s_{n-k})
\\-\h_{\overline{D}_{0,v},\ldots,\overline{D}_{r,v},\overline{D}_{r+1,j,v},\ldots,\overline{D}_{n-k,j,v}}(Z_X(\bm{\omega}_\ell^*s_{\bm{f}});s_0,\ldots,s_{n-k})\Big|.
\end{multline*}

Applying \cite[Theorem~1.4.17(4)]{BPS}, each of this new summands agrees with
\begin{multline*}
\bigg|\int_{X_v^{\an}}\log\frac{\|s_r\|_{r,j,v}}{\|s_r\|_{r,v}}\ \chern(\overline{D}_{0,v})\wedge\dots \wedge\chern(\overline{D}_{r-1,v}) \\ \wedge\chern(\overline{D}_{r+1,j,v})\wedge \dots\wedge\chern(\overline{D}_{n-k,j,v})\wedge\delta_{Z_X(\bm{\omega}_\ell^*s_{\bm{f}})}\bigg|.
\end{multline*}
The notion of distance between metrics and the knowledge of the total volume of the involved Monge--Amp\`ere measure, together with the second part of \cref{cor: proper intersection for strict sequences}, allow to upper bound the previous expression by
\[
\dist(\|\cdot\|_{r,j,v},\|\cdot\|_{r,v})\cdot\MV_M(\Delta_{D_0},\ldots,\Delta_{D_{r-1}},\Delta_{D_{r+1}},\ldots,\Delta_{D_{n-k}},\NP(f_1),\ldots,\NP(f_{k})).
\]

Putting all these bounds together we finally obtain a constant $c\geq0$ only depending on the polytopes $\Delta_{D_0},\ldots,\Delta_{D_{n-k}}$ and $\NP(f_1),\ldots,\NP(f_k)$ for which
\[
\left|\h_{\overline{D}_{0,j},\ldots,\overline{D}_{n-k,j}}(Z_{\mathbb{T}}(\bm{\omega}_\ell^*\bm{f}))-\h_{\overline{D}_{0},\ldots,\overline{D}_{n-k}}(Z_{\mathbb{T}}(\bm{\omega}_\ell^*\bm{f}))\right|
\leq
c\cdot\sum_{v\in\mathfrak{P}}n_v\sum_{r=0}^{n-k}\dist(\|\cdot\|_{r,j,v},\|\cdot\|_{r,v}).
\]

Therefore, the convergence
\[
\lim_{j\to\infty}\h_{\overline{D}_{0,j},\ldots,\overline{D}_{n-k,j}}(Z_{\mathbb{T}}(\bm{\omega}_\ell^*\bm{f}))=\h_{\overline{D}_{0},\ldots,\overline{D}_{n-k}}(Z_{\mathbb{T}}(\bm{\omega}_\ell^*\bm{f}))
\]
is uniform in the choice of~$\ell$.
Moreover, thanks to the assumptions, for all fixed $j$ the limit
\[
\lim_{\ell\to\infty}\h_{\overline{D}_{0,j},\ldots,\overline{D}_{n-k,j}}(Z_{\mathbb{T}}(\bm{\omega}_\ell^*\bm{f}))
\]
exists and agrees with the corresponding sum of mixed integrals.
Therefore, the classical Moore--Osgood theorem allows to exchange limits, and we obtain
\begin{equation*}
\begin{split}
\lim_{\ell\to\infty}\h_{\overline{D}_{0},\ldots,\overline{D}_{n-k}}(Z_{\mathbb{T}}(\bm{\omega}_\ell^*\bm{f}))
&=
\lim_{\ell\to\infty}\lim_{j\to\infty}\h_{\overline{D}_{0,j},\ldots,\overline{D}_{n-k,j}}(Z_{\mathbb{T}}(\bm{\omega}_\ell^*\bm{f}))
\\&
=\lim_{j\to\infty}\lim_{\ell\to\infty}\h_{\overline{D}_{0,j},\ldots,\overline{D}_{n-k,j}}(Z_{\mathbb{T}}(\bm{\omega}_\ell^*\bm{f}))
\\&
=
\lim_{j\to\infty}\sum_{v\in\mathfrak{M}}n_v\MI_M(\vartheta_{\overline{D}_{0,j},v},\ldots,\vartheta_{\overline{D}_{n-k,j},v},\rho_{f_1,v}^\vee,\ldots,\rho_{f_k,v}^\vee)
\\&
=
\sum_{v\in\mathfrak{M}}n_v\MI_M(\vartheta_{\overline{D}_{0},v},\ldots,\vartheta_{\overline{D}_{n-k},v},\rho_{f_1,v}^\vee,\ldots,\rho_{f_k,v}^\vee),
\end{split}
\end{equation*}
where we have used \cref{lem: convergence in mixed integral} in the last equality.
This completes the proof of \cref{item: reduction of metrized divisors} in \cref{prop: reduction steps}.

\subsection{Proof of \cref{item: reduction of sequence of torsion points} in \cref{prop: reduction steps}}

First of all, notice that in the setting of \cref{conj: main conjecture} the height of~$Z_{\mathbb{T}}(\bm{\omega}_\ell^*\bm{f})$, whenever defined, only depends on the class of $\bm{\omega}_\ell\in\mathbb{T}(\overline{\mathbb{K}})^k$ in the quotient of $\mathbb{T}^k$ by the image of the torsion points of $\mathbb{T}$ under the diagonal embedding.

Indeed, let $\omega$ be a torsion point in~$\mathbb{T}(\overline{\mathbb{K}})$, and consider $\bm{t}=(t_1,\ldots,t_k)\in\mathbb{T}(\overline{\mathbb{K}})^k$ for which $t_1^*f_1,\ldots,t_k^*f_k$ meet properly.
The translation-by-$\omega$ map $\tau_\omega\colon X\to X$ restricts to an isomorphism of the torus~$\mathbb{T}$.
Hence, since by definition $(\tau_\omega\mid_{\mathbb{T}})^*(t_i^*f_i)=(\omega t_i)^*f_i$ for all~$i\in\{1,\ldots,k\}$ as regular functions on~$\mathbb{T}$, we deduce that $(\omega t_1)^*f_1,\ldots,\allowbreak(\omega t_k)^*f_k$ also meet properly in~$\mathbb{T}$.
Moreover, with the notation~$\omega\bm{t}=(\omega t_1,\ldots,\omega t_k)$, the projection formula implies that
\[
(\tau_\omega)_*\overline{Z_{\mathbb{T}}((\omega\bm{t})^*\bm{f})}
=
\overline{(\tau_\omega\mid_{\mathbb{T}})_*Z_{\mathbb{T}}((\omega\bm{t})^*\bm{f})}
=
\overline{Z_\mathbb{T}(\bm{t}^*\bm{f})}.
\]
The invariance of toric metrics under translation by torsion points and the projection formula for heights, see for instance \cite[Theorem~1.5.11(2)]{BPS}, yield the equality
\begin{multline*}
\h_{\overline{D}_0,\ldots,\overline{D}_{n-k}}(Z_{\mathbb{T}}((\omega\bm{t})^*\bm{f}))
=
\h_{\tau_\omega^*\overline{D}_0,\ldots,\tau_\omega^*\overline{D}_{n-k}}(Z_{\mathbb{T}}((\omega\bm{t})^*\bm{f}))
\\=
\h_{\overline{D}_0,\ldots,\overline{D}_{n-k}}((\tau_\omega)_*\overline{Z_{\mathbb{T}}((\omega\bm{t})^*\bm{f})})
=
\h_{\overline{D}_0,\ldots,\overline{D}_{n-k}}(Z_{\mathbb{T}}(\bm{t}^*\bm{f})),
\end{multline*}
proving the claimed invariance.

Now, consider a quasi-strict sequence~$(\bm{\omega}_\ell)_\ell$.
As seen, we can replace each $\bm{\omega}_\ell$ by
\[
\omega_{\ell,1}^{-1}.\bm{\omega}_\ell=(1,\omega_{\ell,1}^{-1}\omega_{\ell,2},\ldots,\omega_{\ell,1}^{-1}\omega_{\ell,k})
\]
without affecting the value of the height of the corresponding
intersection cycle.  Because of the isomorphism in~\eqref{eq:7} and
the quasi-strictness of $(\bm{\omega}_\ell)_\ell$, the sequence
$((\omega_{\ell,1}^{-1}\omega_{\ell,2},\ldots,\omega_{\ell,1}^{-1}\omega_{\ell,k}))_\ell$
is strict in~$\mathbb{T}(\overline{\mathbb{K}})^{k-1}$.  The validity
of \cref{conj: main conjecture} for
$(\omega_{\ell,1}^{-1} \bm{\omega}_\ell)_\ell$ implies the one
for~$(\bm{\omega}_\ell)_\ell$, thus concluding the proof of
\cref{item: reduction of sequence of torsion points} in \cref{prop:
  reduction steps}.

\providecommand{\bysame}{\leavevmode\hbox to3em{\hrulefill}\thinspace}
\providecommand{\MR}{\relax\ifhmode\unskip\space\fi MR }
\providecommand{\MRhref}[2]{%
  \href{http://www.ams.org/mathscinet-getitem?mr=#1}{#2}
}
\providecommand{\href}[2]{#2}

\end{document}